\documentclass[a4paper,12pt]{article}
\usepackage{amsmath,amssymb,amsthm}
\usepackage{amscd}

\makeatletter

\@addtoreset{equation}{subsection}
\makeatother

\newtheorem{dfn}{Definition}[subsection]
\newtheorem{prop}[dfn]{Proposition}
\newtheorem{thm}[dfn]{Theorem}
\newtheorem{lem}[dfn]{Lemma}
\newtheorem{cor}[dfn]{Corollary}
\newtheorem{exple}[dfn]{Example}
\newtheorem{rem}[dfn]{Remark}

\newcommand\ag{\mathfrak{a}_g}
\newcommand\agminus{\mathfrak{a}_g^-}
\newcommand\deromega{{\rm Der}_\omega}
\newcommand\Hom{{\rm Hom}}
\newcommand\hotimes{\hat{\otimes}}
\newcommand\Ker{{\rm Ker}}
\newcommand\LfL{[\widehat{\mathcal{L}}]}
\newcommand\LL{\widehat{\mathcal{L}}}
\newcommand\llg{\mathfrak{l}_g}
\newcommand\Qomega{{\mathbb{Q}[[\omega]]}}
\newcommand\Qzeta{{\mathbb{Q}\langle\zeta\rangle}}
\newcommand\T{\widehat{T}}

\begin{document}
\title{The logarithms of Dehn twists}
\author{Nariya Kawazumi and Yusuke Kuno}
\date{}
\maketitle

\begin{abstract}
By introducing an invariant of loops on a compact oriented surface
with one boundary component, we give an explicit formula
for the action of Dehn twists
on the completed group ring of the fundamental
group of the surface. This invariant can be considered as ``the logarithms"
of Dehn twists. The formula generalizes the classical formula describing
the action on the first homology of the surface, and
Morita's explicit computations of the extended first and
the second Johnson homomorphisms.
For the proof we use a homological interpretation
of the Goldman Lie algebra in the framework of Kontsevich's
formal symplectic geometry.
As an application, we prove the action of the Dehn twist
of a simple closed curve on the $k$-th nilpotent quotient
of the fundamental group of the surface depends only
on the conjugacy class of the curve in the $k$-th quotient.
\end{abstract}

\section{Introduction}
Let $\Sigma$ be a compact oriented $C^{\infty}$-surface
of genus $g>0$ with one boundary component, and $\mathcal{M}_{g,1}$
the mapping class group of $\Sigma$ relative to the boundary.
In other words, $\mathcal{M}_{g,1}$ is the group of diffeomorphisms
of $\Sigma$ fixing the boundary $\partial \Sigma$ pointwise,
modulo isotopies fixing the boundary pointwise.
Choose a basepoint $*$ on the boundary $\partial\Sigma$.
The group $\mathcal{M}_{g,1}$ (faithfully) acts on $\pi=\pi_1(\Sigma,*)$,
hence on the nilpotent quotients of $\pi$. For example,
$\mathcal{M}_{g,1}$ acts on the first homology group
$H_1(\Sigma;\mathbb{Z})\cong \pi/[\pi,\pi]$,
and this gives rise to the classical representation
$$\mathcal{M}_{g,1}\rightarrow Sp(2g;\mathbb{Z}),$$
whose kernel is called the Torelli group $\mathcal{I}_{g,1}$.
Looking at the kernel of the action on the higher nilpotent quotients of $\pi$,
we obtain a series of normal subgroups of $\mathcal{M}_{g,1}$,
denoted by $\mathcal{M}_{g,1}[k]$, $k\ge 1$, so that
$\mathcal{M}_{g,1}[1]=\mathcal{I}_{g,1}$. More precisely,
the group $\mathcal{M}_{g,1}[k]$ consists of the mapping classes
acting trivially on the $k$-th nilpotent quotient of $\pi$ (see \S7.4).

In this view point, the quotients $\mathcal{M}_{g,1}/\mathcal{M}_{g,1}[k]$
serve as approximations of $\mathcal{M}_{g,1}$, and the
successive quotients $\mathcal{M}_{g,1}[k]/\mathcal{M}_{g,1}[k+1]$
can be seen as particles of them. A systematic study of these
particles was initiated by Johnson \cite{Joh1} \cite{Joh2}.
He introduced a series of group homomorphisms
$$\tau_k\colon \mathcal{M}_{g,1}[k]\rightarrow {\rm Hom}(H,\mathcal{L}_k),\ k\ge 1,$$
which induce the injections $\tau_k\colon \mathcal{M}_{g,1}[k]/\mathcal{M}_{g,1}[k+1]
\hookrightarrow {\rm Hom}(H,\mathcal{L}_k)$, $k\ge 1$.
Here $H$ is the first integral homology of the surface and
$\mathcal{L}_k$ is the degree $k$-part of the free Lie algebra generated
by $H$. The homomorphism $\tau_k$ is nowadays called {\it the $k$-th
Johnson homomorphism}. He extensively studied the first and the second
Johnson homomorphisms, and in \cite{Joh3} he proved $\tau_1$
gives the free part of the abelianization of $\mathcal{I}_{g,1}$.

One of several significant developments which followed
the initial works of Johnson is about extensions of
the Johnson homomorphisms to the whole mapping class group.
In \cite{Mo2}, Morita showed that the first Johnson homomorphism
$\tau_1$ extends to the whole group $\mathcal{M}_{g,1}$ as
a crossed homomorphism, denoted by
$\tilde{k}\in Z^1(\mathcal{M}_{g,1};\frac{1}{2}\Lambda^3H)$.
Here, $\Lambda^3H$ is the third exterior product of $H$.
He also showed that the extension is unique up to coboundaries.
The arguments in \cite{Mo2} are supported by
many explicit computations on Humphreys generators,
which are generators of $\mathcal{M}_{g,1}$ consisting of several Dehn twists.
In \cite{Mo3} \cite{Mo4}, Morita further showed that the
second Johnson homomorphism $\tau_2$ also extends to the
whole $\mathcal{M}_{g,1}$ as a crossed homomorphism,
and again did many explicit computations.

After the works of Morita, there have been known several studies
including Hain \cite{Hai} and Day \cite{Day} \cite{Day2}
about extensions of the Johnson homomorphisms to the
whole mapping class group.
Another approach by using the notion of {\it generalized Magnus expansions}
is developed in \cite{Ka}. Hereafter, let $H=H_1(\Sigma;\mathbb{Q})$.
Roughly speaking, a Magnus expansion in the sense of \cite{Ka}
is an identification $\theta\colon \widehat{\mathbb{Q}\pi}
\stackrel{\cong}{\rightarrow}\widehat{T}$ as complete augmented algebras,
where $\widehat{\mathbb{Q}\pi}$ is the completed group
ring of $\pi$ and $\widehat{T}$ is the completed tensor algebra
generated by $H$.
Once we choose a Magnus expansion $\theta$, then we have an
injective homomorphism
$$T^{\theta}\colon \mathcal{M}_{g,1}\rightarrow {\rm Aut}(\widehat{T})$$
called {\it the total Johnson map associated to} $\theta$.
The map $T^{\theta}$ can be understood as a tensor expression of
the action of $\mathcal{M}_{g,1}$ on the completed group ring of $\pi$,
since $\theta\colon \widehat{\mathbb{Q}\pi}\overset\cong\to
\widehat{T}$ is an isomorphism. For details, see \S2.5.
As was clarified in \cite{Ka}, $T^{\theta}$ induces
$\theta$-dependent extensions of all the Johnson homomorphisms $\tau_k$,
denoted by $\tau^{\theta}_k$ where $k\ge 1$, to the whole mapping class group.
$\tau^{\theta}_k$'s are no longer homomorphisms, and are
not crossed homomorphisms if $k\ge 2$,
but satisfy an infinite sequence of coboundary conditions.

Note that the fundamental group $\pi$ is a free group of rank $2g$.
Actually the treatment in \cite{Ka} is on ${\rm Aut}(F_n)$,
the automorphism group of a free group of rank $n$,
rather than the mapping class group.
As long as we just regard $\pi$ as a free group,
it does not seem a matter of concern which Magnus expansion we choose.
However, recently Massuyeau \cite{Mas} introduced the notion
of {\it a symplectic expansion}, which seems suitable for the study
of $\mathcal{M}_{g,1}$ from the view point of Magnus expansions.
A symplectic expansion is a Magnus expansion of $\pi$
respecting the fact that $\pi$ has a particular element
corresponding to the boundary of $\Sigma$.
For precise definition, see \S2.4.

In this paper, we begin a quantitative approach to  the topology of $\Sigma$
and the mapping class group $\mathcal{M}_{g,1}$
via a symplectic expansion.
The primary theme is the Dehn twist formula for the total Johnson
map associated to a symplectic expansion.
As was stated above, Dehn twists generate the mapping class group
$\mathcal{M}_{g,1}$.
We introduce an invariant of loops on $\Sigma$, and
derive a formula of the values of $T^{\theta}$ on Dehn twists
in terms of this invariant.
It is classically known that the action of a Dehn twist on the homology
of an oriented surface is given by transvection. Our formula can be
seen as a generalization of this fact. Moreover, it gives formulas
for the extensions $\tau_k^{\theta}$ and
recovers some computations of Morita on the extended $\tau_1$, and $\tau_2$.
Behind our proof of the above formula, a close relationship
between the Goldman Lie algebra of $\Sigma$ and
formal symplectic geometry plays a vital role. The
relationship is established via a symplectic expansion, and
this is another theme of this paper keeping pace with the first.

\subsection{Statement of the main results}
Let us briefly introduce several notations.
The completed tensor algebra
$\widehat{T}=\prod_{m=0}^{\infty}H^{\otimes m}$
has a decreasing filtration of two-sided ideals given by
$\widehat{T}_p:=\prod_{m\ge p}^{\infty}H^{\otimes m}$,
for $p\ge 1$. For a Magnus expansion $\theta$,
let $\ell^{\theta}:=\log \theta$. Then $\ell^{\theta}$ is a map from
$\pi$ to $\widehat{T}_1$.
Define a linear map $N\colon \widehat{T}_1 \rightarrow \widehat{T}_1$ by
$N|_{H^{\otimes p}}=\sum_{m=0}^{p-1} \nu^m$, for $p\ge 1$,
where $\nu\colon H^{\otimes p}\rightarrow H^{\otimes p}$
is the map induced from the cyclic permutation. For $x\in \pi$,
let $$L^{\theta}(x):=\frac{1}{2}N(\ell^{\theta}(x)\ell^{\theta}(x))
\in \widehat{T}_2.$$
It turns out that $L^{\theta}(x^{-1})=L^{\theta}(x)$,
and $L^{\theta}(yxy^{-1})=L^{\theta}(x)$ for $x,y\in \pi$.
Thus if $\gamma$ is an unoriented loop on $\Sigma$,
$L^{\theta}(\gamma)\in \widehat{T}_2$ is well-defined
by taking a representative of $\gamma$ in $\pi$.
Using the Poincar\'e duality, we make an identification
$\widehat{T}_1= H\otimes \widehat{T}\cong {\rm Hom}(H,\widehat{T})$
and regard $L^{\theta}(\gamma)$ as a derivation of $\widehat{T}$
by applications of the Leibniz rule.

\begin{center}
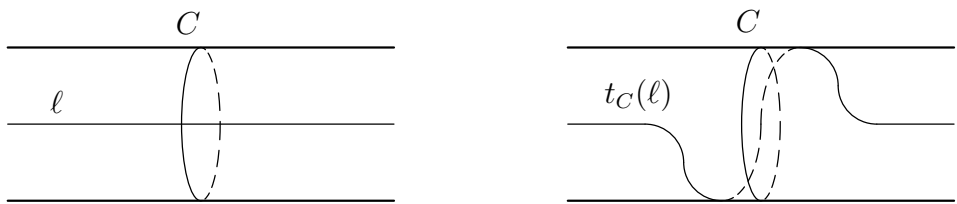
Figure 1: the right handed Dehn twist

\unitlength 0.1in
\begin{picture}( 49.0000, 10.5000)(  2.0000,-12.0000)
%
\special{pn 13}%
\special{pa 200 400}%
\special{pa 2200 400}%
\special{fp}%
%
\special{pn 13}%
\special{pa 200 1200}%
\special{pa 2200 1200}%
\special{fp}%
%
\special{pn 13}%
\special{pa 3100 400}%
\special{pa 5100 400}%
\special{fp}%
%
\special{pn 13}%
\special{pa 3100 1200}%
\special{pa 5100 1200}%
\special{fp}%
%
\special{pn 8}%
\special{ar 1200 800 100 400  1.5707963 4.7123890}%
%
\special{pn 8}%
\special{ar 1200 800 100 400  4.7123890 4.9523890}%
\special{ar 1200 800 100 400  5.0963890 5.3363890}%
\special{ar 1200 800 100 400  5.4803890 5.7203890}%
\special{ar 1200 800 100 400  5.8643890 6.1043890}%
\special{ar 1200 800 100 400  6.2483890 6.4883890}%
\special{ar 1200 800 100 400  6.6323890 6.8723890}%
\special{ar 1200 800 100 400  7.0163890 7.2563890}%
\special{ar 1200 800 100 400  7.4003890 7.6403890}%
\special{ar 1200 800 100 400  7.7843890 7.8539816}%
%
\special{pn 8}%
\special{ar 4100 800 100 400  1.5707963 4.7123890}%
%
\special{pn 8}%
\special{ar 4100 800 100 400  4.7123890 4.9523890}%
\special{ar 4100 800 100 400  5.0963890 5.3363890}%
\special{ar 4100 800 100 400  5.4803890 5.7203890}%
\special{ar 4100 800 100 400  5.8643890 6.1043890}%
\special{ar 4100 800 100 400  6.2483890 6.4883890}%
\special{ar 4100 800 100 400  6.6323890 6.8723890}%
\special{ar 4100 800 100 400  7.0163890 7.2563890}%
\special{ar 4100 800 100 400  7.4003890 7.6403890}%
\special{ar 4100 800 100 400  7.7843890 7.8539816}%
%
\special{pn 8}%
\special{pa 200 800}%
\special{pa 2200 800}%
\special{fp}%
%
\special{pn 8}%
\special{ar 3900 800 200 400  6.2831853 6.4831853}%
\special{ar 3900 800 200 400  6.6031853 6.8031853}%
\special{ar 3900 800 200 400  6.9231853 7.1231853}%
\special{ar 3900 800 200 400  7.2431853 7.4431853}%
\special{ar 3900 800 200 400  7.5631853 7.7631853}%
%
\special{pn 8}%
\special{ar 4300 800 200 400  3.1415927 3.3415927}%
\special{ar 4300 800 200 400  3.4615927 3.6615927}%
\special{ar 4300 800 200 400  3.7815927 3.9815927}%
\special{ar 4300 800 200 400  4.1015927 4.3015927}%
\special{ar 4300 800 200 400  4.4215927 4.6215927}%
%
\special{pn 8}%
\special{ar 3900 1000 200 200  1.5707963 3.1415927}%
%
\special{pn 8}%
\special{ar 3500 1000 200 200  4.7123890 6.2831853}%
%
\special{pn 8}%
\special{pa 3100 800}%
\special{pa 3500 800}%
\special{fp}%
%
\special{pn 8}%
\special{ar 4300 600 200 200  4.7123890 6.2831853}%
%
\special{pn 8}%
\special{ar 4700 600 200 200  1.5707963 3.1415927}%
%
\special{pn 8}%
\special{pa 5100 800}%
\special{pa 4700 800}%
\special{fp}%
\put(4.2000,-7.4000){\makebox(0,0)[lb]{$\ell$}}%
\put(10.7000,-3.3000){\makebox(0,0)[lb]{$C$}}%
\put(39.7000,-3.2000){\makebox(0,0)[lb]{$C$}}%
\put(32.9000,-7.3000){\makebox(0,0)[lb]{$t_C(\ell)$}}%
\end{picture}%

\end{center}

Let $C$ be a simple closed curve on $\Sigma$.
We denote by $t_C\in \mathcal{M}_{g,1}$ the right handed Dehn twist along $C$
(see Figure 1). By the remark above, $L^{\theta}(C)\in \widehat{T}_2$ is defined.
This invariant turns out to be ``the logarithm" of $t_C$:

\begin{thm}[$=$Theorem \ref{thm:7-1-1}]
\label{thm:1-1-1}
Let $\theta$ be a symplectic expansion and $C$ a simple closed curve
on $\Sigma$. Then the total Johnson map $T^{\theta}(t_C)$ is described as
\begin{equation}
\label{eq:1-1-1}
T^{\theta}(t_C)=e^{-L^{\theta}(C)}.
\end{equation}
Here, the right hand side is the algebra automorphism of $\widehat{T}$
defined by the exponential of the derivation $-L^{\theta}(C)$.
\end{thm}

The formula does not hold for a group-like Magnus expansion
which is not symplectic.
It should be remarked here that whether $C$
is non-separating or separating, the formula holds.
Note that (\ref{eq:1-1-1}) is an equality as filter-preserving
automorphisms of $\widehat{T}$.
If we compute (\ref{eq:1-1-1}) modulo $\widehat{T}_2$, we get
the well-known formula for the action on the homology:
\begin{equation}
\label{eq:1-1-2}
t_C(X)=X-(X\cdot [C])[C],\ X\in H.
\end{equation}
Here $(\ \cdot \ )$ is the intersection form
on $H$ and $[C]$ is the homology class
of $C$ with a fixed orientation. By computing (\ref{eq:1-1-1})
modulo higher tensors, we will get formulas of $\tau^{\theta}_k(t_C)$
in terms of $L^{\theta}(C)$. These formulas match the computations
of the extended $\tau_1$ for Humphreys generators and
$\tau_2(t_C)$ for separating $C$ by Morita \cite{Mo} \cite{Mo2}. See \S7.

The classical formula (\ref{eq:1-1-2}) tells us that the action
of $t_C$ on $H_1(\Sigma;\mathbb{Z})$ depends only on
the class $\pm [C]$. As an application of
Theorem \ref{thm:1-1-1}, we get a generalization of this fact.
Let $N_k=N_k(\pi)$ be the $k$-th nilpotent quotient
of $\pi$. We number the indices so that
$N_1=\pi^{\rm abel}\cong H_1(\Sigma;\mathbb{Z})$.
The mapping class group $\mathcal{M}_{g,1}$ naturally acts on $N_k$.
Let $\bar{N_k}$ be the quotient set of $N_k$ by conjugation and
the relation $x\sim x^{-1}$. Then any simple closed curve $C$
defines an element of $\bar{N_k}$, which we denote by $\bar{C}_k$.

\begin{thm}[$=$Theorem \ref{thm:7-4-1}]
\label{thm:1-1-2}
For each $k\ge 1$, the action of $t_C$ on $N_k$
depends only on the class $\bar{C}_k\in \bar{N}_k$.
If $C$ is separating, it depends only on the class
$\bar{C}_{k-1}\in \bar{N}_{k-1}$.
\end{thm}

\subsection{The Goldman Lie algebra and formal symplectic geometry}
The key ingredients for our proof of Theorem \ref{thm:1-1-1}
are {\it the Goldman Lie algebra} of $\Sigma$, see Goldman \cite{Go}, and
its homological interpretation in the framework of
{\it formal symplectic geometry} by Kontsevich \cite{Kon}.

The Goldman Lie algebra is a Lie algebra associated to an
oriented surface, and regarded as an origin of string topology
by Chas-Sullivan \cite{CS}.
It was introduced in \cite{Go} as a universal
object for describing the Poisson brackets of coordinate functions
on the space ${\rm Hom}(\pi,G)/G$,
using his notation, with a natural symplectic structure.
Here $\pi$ is the fundamental group of a closed oriented surface
(hence is not our $\pi$)
and $G$ is a Lie group satisfying very general conditions.

Let $\mathbb{Q}\hat{\pi}$ be the Goldman Lie algebra of $\Sigma$.
Here, $\hat{\pi}$ is the set of conjugacy classes of $\pi$.
In \S3, we show that $\mathbb{Q}\hat{\pi}$ acts on the group ring $\mathbb{Q}\pi$
as a derivation. Namely, we show that there is a Lie algebra homomorphism
$\sigma\colon \mathbb{Q}\hat{\pi} \rightarrow {\rm Der}(\mathbb{Q}\pi)$.
On the other hand, let
$\mathfrak{a}_g^-={\rm Der}_{\omega}(\widehat{T})$ be the
space of derivations of $\widehat{T}$ killing the symplectic form.
This is a variant of ``associative", one of the three Lie algebras in formal symplectic
geometry. In fact, we have a canonical isomorphism
$\mathfrak{a}_g^-=N(\widehat{T}_1)$. For details, see \S2.7.

Then we have the following two theorems. The slogan is:
a symplectic expansion builds a bridge between
the objects in ``surface-side" and ``$\widehat{T}$-side".

\begin{thm}[$=$Theorem \ref{thm:6-3-3}]
\label{thm:1-2-1}
Let $\theta$ be a symplectic expansion. Then the map
$$
-\lambda_{\theta}\colon \mathbb{Q}\hat{\pi}\to N(\T_1) = \agminus, \quad
x \mapsto -N\theta(x)
$$
is a Lie algebra homomorphism. The kernel is the subspace $\mathbb{Q}1$ 
spanned by the constant loop $1$, and the image is dense in $N(\T_1) =
\agminus$ with respect to the $\T_1$-adic topology. 
\end{thm}

\begin{thm}[$=$Theorem \ref{thm:6-4-3}]
\label{thm:1-2-2}
Let $\theta$ be a symplectic expansion. Then, 
for $u \in \mathbb{Q}\hat{\pi}$ and $v \in \mathbb{Q}\pi$, 
we have the equality
$$
\theta(\sigma(u)v) = -\lambda_{\theta}(u)\theta(v).
$$
Here the right hand side means minus the action of
$\lambda_{\theta}(u) \in \agminus$ on the tensor $\theta(v) \in \T$
as a derivation.
In other words, the diagram
$$
\begin{CD}
\mathbb{Q}\hat{\pi} \times \mathbb{Q}\pi @>{\sigma}>>
\mathbb{Q}\pi \\
@V{-\lambda_{\theta}\times \theta}VV @VV{\theta}V \\
\mathfrak{a}_g^{-} \times \widehat{T} @>>> \widehat{T},
\end{CD}
$$
where the bottom horizontal arrow means the derivation, commutes.
\end{thm}

In fact, we can derive Theorem \ref{thm:1-1-1} from these
two theorems and some care about convergence. See \S6.5.
Another application of these theorems will be studied
in our forthcoming paper \cite{KK}.

\subsection{Organization of the paper}
This paper is organized as follows.
In section 2 we start by recalling Magnus expansions,
symplectic expansions, and the total Johnson map
associated to a Magnus expansion.
Then we introduce the invariant
$L^{\theta}$ and prove some properties of it. We close
this section by showing connections to
formal symplectic geometry.

In section 3, we look at the Goldman Lie
algebra of $\Sigma$, and we show that it acts on the group
ring of $\pi$ as a derivation. We also give a homological
interpretation of this action.
In section 5, we construct a counterpart of the story in section 3,
in the framework of formal symplectic geometry.
In particular, we give homological interpretations of
$\mathfrak{a}_g^-$ and its action on $\widehat{T}$.
To do this we need a (co)homology theory of
(complete) Hopf algebras, to which section 4 is devoted.
We mention the relative homology of a pair, cap products, Kronecker products,
and relation to (co)homology of groups.

The theorems in Introduction are proved in sections 6 and 7.
In section 6, the stories in sections 3 and 5 are compared
by a symplectic expansion, and
Theorems \ref{thm:1-2-1} and \ref{thm:1-2-2} are proved.
In section 7 we prove
Theorems \ref{thm:1-1-1} and \ref{thm:1-1-2},
and derive some formulas of $\tau^{\theta}_k(t_C)$,
which recover some computations by Morita.
Finally in section 8 we consider the case of the mapping class group
of a once punctured surface and derive results similar
to Theorems \ref{thm:1-1-1} and \ref{thm:1-1-2}.

In Appendix, partial examples of symplectic expansions
are given.

\subsection{Conventions}
Here we list the conventions of this paper.

\begin{enumerate}
\item
Let $G$ be a group. For $x,y\in G$,
we denote by $[x,y]$ their commutator $xyx^{-1}y^{-1}\in G$.
\item
As usual, we often ignore the distinction between
a path and its homotopy class.
\item
For continuous paths $\gamma_1$, $\gamma_2$ on $\Sigma$
such that the endpoint of $\gamma_1$
coincides with the start point of $\gamma_2$, their product
$\gamma_1\gamma_2$ means the path traversing $\gamma_1$ first,
then $\gamma_2$. The product in the fundamental group is the
induced one.
\item
Sometimes we omit $\otimes$ to express tensors.
For example, if $X,Y,Z\in H$, then
$XYZ$ means $X\otimes Y\otimes Z \in H^{\otimes 3}$.
If $u\in H^{\otimes k}$ and $X\in H$, then
$uX$ means $u\otimes X\in H^{\otimes k+1}$.
\item
Throughout the paper we basically work over $\mathbb{Q}$,
although several results hold over the integers,
especially in \S3,
and it would be possible to present all the main results
with the coefficients in an integral domain including
the rationals $\mathbb{Q}$.
\end{enumerate}

\noindent \textbf{Acknowledgments.}
The authors wish to express their gratitude to
Alex Bene, Shigeyuki Morita, Robert Penner, and Masatoshi Sato
for valuable comments.
The first-named author is grateful to
Yasushi Kasahara for stimulating conversations,
and Akira Kono for helpful comments concerning
the higher-dimensional generalization of our map $\sigma$
(Remark 3.2.3). The second-named author would
like to thank Tadashi Ishibe and Ichiro Shimada
for valuable comments.

The first-named author is partially supported by the Grant-in-Aid for
Scientific Research (A) (No.18204002) from the
Japan Society for Promotion of Sciences. The second-named
author is supported by JSPS Research Fellowships
for Young Scientists (22$\cdot$4810).

\tableofcontents

\section{Magnus expansions and total Johnson map}

\subsection{Our surface and mapping class group}
As in Introduction, $\Sigma$ is
a compact oriented $C^{\infty}$-surface of genus $g>0$
with one boundary component. We choose
a basepoint $*$ on the boundary $\partial \Sigma$. The fundamental
group $\pi:=\pi_1(\Sigma,*)$ is a free group of rank $2g$.
Let $H:=H_1(\Sigma;\mathbb{Q})$ be the first homology
group of $\Sigma$.
$H$ is naturally isomorphic to
$H_1(\pi;\mathbb{Q})\cong
\pi^{\rm abel}\otimes_{\mathbb{Z}}\mathbb{Q}$,
the first homology group of $\pi$.
Here $\pi^{\rm abel}=\pi/[\pi,\pi]$ is
the abelianization of $\pi$. Under this identification, we write
$$[x]:=( x\ {\rm mod\ }[\pi,\pi] )
\otimes_{\mathbb{Z}} 1 \in H,\  {\rm for\ } x\in \pi.$$
Let $\mathcal{M}_{g,1}$ be the mapping class group of $\Sigma$
relative to the boundary, namely the group of
orientation-preserving diffeomorphisms of $\Sigma$
fixing $\partial \Sigma$ pointwise,
modulo isotopies fixing $\partial \Sigma$ pointwise.

Let $\zeta \in \pi$ be a based loop parallel to $\partial \Sigma$
and going by counter-clockwise manner. Explicitly, if we take
symplectic generators $\alpha_1,\beta_1,\ldots,\alpha_g,\beta_g \in \pi$
as shown in Figure 2,
$\zeta=\prod_{i=1}^g [\alpha_i,\beta_i]$.

\begin{center}
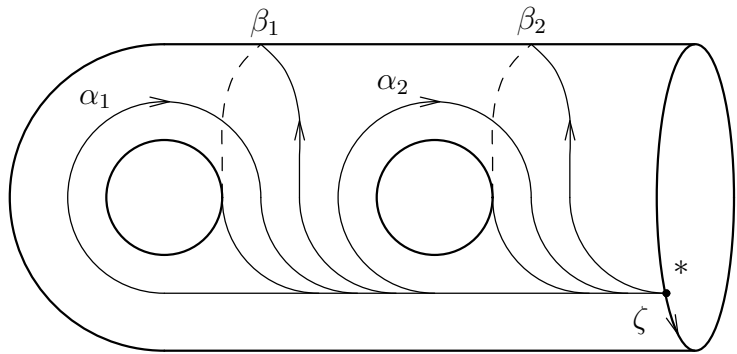
Figure 2: symplectic generators of $\pi$ for $g=2$
\unitlength 0.1in
\begin{picture}( 37.5000, 18.2000)(  2.0000,-18.7000)
%
\special{pn 13}%
\special{ar 3750 1070 200 800  0.0000000 6.2831853}%
%
\special{pn 13}%
\special{ar 1000 1070 800 800  1.5707963 4.7123890}%
%
\special{pn 13}%
\special{ar 1000 1070 300 300  0.0000000 6.2831853}%
%
\special{pn 13}%
\special{ar 2400 1070 300 300  0.0000000 6.2831853}%
%
\special{pn 8}%
\special{ar 1000 1070 500 500  1.5707963 6.2831853}%
%
\special{pn 8}%
\special{ar 2400 1070 500 500  1.5707963 6.2831853}%
%
\special{pn 8}%
\special{ar 2000 1070 500 500  1.5707963 3.1415927}%
%
\special{pn 8}%
\special{ar 3400 1070 500 500  1.5707963 3.1415927}%
%
\special{pn 8}%
\special{ar 1800 1070 500 500  1.5707963 3.1415927}%
%
\special{pn 8}%
\special{ar 3200 1070 500 500  1.5707963 3.1415927}%
%
\special{pn 8}%
\special{ar 2200 1070 500 500  1.5707963 3.1415927}%
%
\special{pn 8}%
\special{ar 3600 1070 500 500  1.5707963 3.1415927}%
%
\special{pn 8}%
\special{pa 2700 1070}%
\special{pa 2700 1038}%
\special{pa 2700 1006}%
\special{pa 2700 974}%
\special{pa 2700 942}%
\special{pa 2700 910}%
\special{pa 2700 878}%
\special{pa 2700 846}%
\special{pa 2700 814}%
\special{pa 2698 782}%
\special{pa 2698 750}%
\special{pa 2698 718}%
\special{pa 2700 686}%
\special{pa 2702 654}%
\special{pa 2706 622}%
\special{pa 2710 590}%
\special{pa 2716 558}%
\special{pa 2726 526}%
\special{pa 2736 494}%
\special{pa 2746 464}%
\special{pa 2760 434}%
\special{pa 2776 406}%
\special{pa 2794 380}%
\special{pa 2812 356}%
\special{pa 2834 332}%
\special{pa 2858 310}%
\special{pa 2882 288}%
\special{pa 2900 270}%
\special{sp 0.070}%
%
\special{pn 8}%
\special{pa 3100 1070}%
\special{pa 3100 1038}%
\special{pa 3100 1006}%
\special{pa 3100 974}%
\special{pa 3100 942}%
\special{pa 3100 910}%
\special{pa 3100 878}%
\special{pa 3100 846}%
\special{pa 3102 814}%
\special{pa 3102 782}%
\special{pa 3102 750}%
\special{pa 3102 718}%
\special{pa 3102 686}%
\special{pa 3100 654}%
\special{pa 3096 622}%
\special{pa 3090 590}%
\special{pa 3084 558}%
\special{pa 3076 526}%
\special{pa 3066 494}%
\special{pa 3054 464}%
\special{pa 3040 434}%
\special{pa 3026 406}%
\special{pa 3008 380}%
\special{pa 2988 356}%
\special{pa 2966 332}%
\special{pa 2944 310}%
\special{pa 2920 288}%
\special{pa 2900 270}%
\special{sp}%
%
\special{pn 8}%
\special{pa 1300 1070}%
\special{pa 1300 1038}%
\special{pa 1300 1006}%
\special{pa 1300 974}%
\special{pa 1300 942}%
\special{pa 1300 910}%
\special{pa 1300 878}%
\special{pa 1300 846}%
\special{pa 1300 814}%
\special{pa 1298 782}%
\special{pa 1298 750}%
\special{pa 1298 718}%
\special{pa 1300 686}%
\special{pa 1302 654}%
\special{pa 1306 622}%
\special{pa 1310 590}%
\special{pa 1316 558}%
\special{pa 1326 526}%
\special{pa 1336 494}%
\special{pa 1346 464}%
\special{pa 1360 434}%
\special{pa 1376 406}%
\special{pa 1394 380}%
\special{pa 1412 356}%
\special{pa 1434 332}%
\special{pa 1458 310}%
\special{pa 1482 288}%
\special{pa 1500 270}%
\special{sp 0.070}%
%
\special{pn 8}%
\special{pa 1700 1070}%
\special{pa 1700 1038}%
\special{pa 1700 1006}%
\special{pa 1700 974}%
\special{pa 1700 942}%
\special{pa 1700 910}%
\special{pa 1700 878}%
\special{pa 1700 846}%
\special{pa 1702 814}%
\special{pa 1702 782}%
\special{pa 1702 750}%
\special{pa 1702 718}%
\special{pa 1702 686}%
\special{pa 1700 654}%
\special{pa 1696 622}%
\special{pa 1690 590}%
\special{pa 1684 558}%
\special{pa 1676 526}%
\special{pa 1666 494}%
\special{pa 1654 464}%
\special{pa 1640 434}%
\special{pa 1626 406}%
\special{pa 1608 380}%
\special{pa 1588 356}%
\special{pa 1566 332}%
\special{pa 1544 310}%
\special{pa 1520 288}%
\special{pa 1500 270}%
\special{sp}%
%
\special{pn 8}%
\special{pa 1000 1570}%
\special{pa 3600 1570}%
\special{fp}%
%
\special{pn 13}%
\special{pa 1000 1870}%
\special{pa 3746 1870}%
\special{fp}%
%
\special{pn 13}%
\special{pa 1000 270}%
\special{pa 3746 270}%
\special{fp}%
%
\special{pn 20}%
\special{sh 1}%
\special{ar 3600 1570 10 10 0  6.28318530717959E+0000}%
\special{sh 1}%
\special{ar 3600 1570 10 10 0  6.28318530717959E+0000}%
%
\special{pn 8}%
\special{pa 1676 770}%
\special{pa 1700 670}%
\special{fp}%
\special{pa 1700 670}%
\special{pa 1726 770}%
\special{fp}%
%
\special{pn 8}%
\special{pa 3076 770}%
\special{pa 3100 670}%
\special{fp}%
\special{pa 3100 670}%
\special{pa 3126 770}%
\special{fp}%
%
\special{pn 8}%
\special{pa 926 546}%
\special{pa 1026 570}%
\special{fp}%
\special{pa 1026 570}%
\special{pa 926 596}%
\special{fp}%
%
\special{pn 8}%
\special{pa 2326 546}%
\special{pa 2426 570}%
\special{fp}%
\special{pa 2426 570}%
\special{pa 2326 596}%
\special{fp}%
%
\special{pn 8}%
\special{pa 3600 1700}%
\special{pa 3660 1784}%
\special{fp}%
\special{pa 3660 1784}%
\special{pa 3646 1682}%
\special{fp}%
\put(5.6000,-6.0500){\makebox(0,0)[lb]{$\alpha_1$}}%
\put(14.4500,-2.3000){\makebox(0,0)[lb]{$\beta_1$}}%
\put(21.0000,-5.3000){\makebox(0,0)[lb]{$\alpha_2$}}%
\put(28.2500,-2.2000){\makebox(0,0)[lb]{$\beta_2$}}%
\put(36.3500,-14.7500){\makebox(0,0)[lb]{$*$}}%
\put(34.2500,-17.8500){\makebox(0,0)[lb]{$\zeta$}}%
\end{picture}%

\end{center}

By the classical theorem of Dehn-Nielsen, the natural action
of $\mathcal{M}_{g,1}$ on $\pi=\pi_1(\Sigma,*)$ is faithful and
we can identify $\mathcal{M}_{g,1}$ as a subgroup of ${\rm Aut}(\pi)$:
\begin{equation}
\label{eq:2-1-1}
\mathcal{M}_{g,1}=\{ \varphi\in {\rm Aut}(\pi);\ \varphi(\zeta)=\zeta \}.
\end{equation}

\subsection{Group ring and tensor algebra}
Let $\mathbb{Q}\pi$ be the group ring of $\pi$.
It has an augmentation given by
$\varepsilon \colon \mathbb{Q}\pi \rightarrow \mathbb{Q},
\ \sum_in_ix_i \mapsto \sum_in_i,$
where $n_i\in \mathbb{Q}, x_i\in \pi$. Let $I\pi$ be
the augmentation ideal, namely the kernel of $\varepsilon$.
The powers of $I\pi$ give
a decreasing filtration of $\mathbb{Q}\pi$. 
The completed group ring of $\pi$, or more precisely the $I\pi$-adic completion
of $\mathbb{Q}\pi$, is
$$\widehat{\mathbb{Q}\pi}:=\varprojlim_m \mathbb{Q}\pi/{I\pi}^m.$$
It naturally has a structure of a complete augmented algebra
(in the sense of Quillen \cite{Qui}, Appendix A) with respect to
a decreasing filtration given by
$\varprojlim_{m\ge p} {I\pi}^p/{I\pi}^m,\ {\rm for\ } p\ge 1$.

Let $\widehat{T}$ be the completed tensor algebra generated by $H$.
Namely
$\widehat{T}=\prod_{m=0}^{\infty}H^{\otimes m},$
where $H^{\otimes m}$ is the tensor space of degree $m$.
Choosing a basis for $H$, it is isomorphic to the ring of
non-commutative formal power series in $2g$ indeterminates.
We can write elements of $\widehat{T}$ uniquely as
$$u=\sum_{m=0}^{\infty}u_m=u_0+u_1+u_2+\cdots,
\ u_m\in H^{\otimes m}.$$
The algebra $\widehat{T}$ has an augmentation given by
$\varepsilon \colon \widehat{T}\rightarrow \mathbb{Q},
\ u=\sum_{m=0}^{\infty}u_m \mapsto u_0$,
and it is a complete augmented algebra
with respect to a decreasing filtration
$$\widehat{T}_p:=\prod_{m\ge p}H^{\otimes m},\ {\rm for\ } p\ge 1.$$

Both $\mathbb{Q}\pi$ and $\widehat{T}$ have a structure of
(complete) Hopf algebra. For simplicity, we use the same letters
$\Delta$ and $\iota$ for the coproducts and the antipodes
of both Hopf algebras. In the case of $\mathbb{Q}\pi$, these are given by
$$\Delta(x)=x\otimes x,\ {\rm and\ } \iota(x)=x^{-1},\ {\rm for\ } x\in \pi,$$
and in the case of $\widehat{T}$, the formulas are
$$\Delta(X)=X\hat{\otimes} 1+1\hat{\otimes} X,\ {\rm and\ }
\iota(X)=-X,\ {\rm for\ } X\in H.$$
Here $\hat{\otimes}$ means the completed tensor product.
The Hopf algebra structure of $\mathbb{Q}\pi$ induces
a structure of a complete Hopf algebra on $\widehat{\mathbb{Q}\pi}$.

By definition the set of group-like elements of $\widehat{T}$ is
the set of $u\in \widehat{T}$ satisfying $\Delta(u)=u\hat{\otimes} u$,
and the set of primitive elements is
$\widehat{\mathcal{L}}:=
\{ u\in \widehat{T}; \Delta(u)=u\hat{\otimes} 1+1\hat{\otimes}u \}$.
As is well-known, $\widehat{\mathcal{L}}$
has a structure of a Lie algebra with the bracket $[u,v]:=uv-vu$.
The degree $p$-part $\mathcal{L}_p:=\widehat{\mathcal{L}}\cap H^{\otimes p}$
is described successively as $\mathcal{L}_1=H$, and
$\mathcal{L}_p=[H,\mathcal{L}_{p-1}]$ for $p\ge 2$.
For $p\ge 1$, we write
$\widehat{\mathcal{L}}_p=\widehat{\mathcal{L}}\cap \widehat{T}_p$.

By the exponential map
$$\exp(u)=\sum_{n=0}^{\infty} \frac{1}{n!}u^n,\ {\rm for\ } u\in \widehat{\mathcal{L}},$$
$\widehat{\mathcal{L}}$ is bijectively mapped to the set of group-like elements
and the inverse is given by the logarithm
$$\log(u)=\sum_{n=1}^{\infty} \frac{(-1)^{n-1}}{n} (u-1)^n.$$

Since the set of group-like elements constitutes a group
with respect to the multiplication of $\widehat{T}$,
the above bijection endows the underlying set of
$\widehat{\mathcal{L}}$ with a group structure, which is
described by the Baker-Campbell-Hausdorff series:
$$u\cdot v=\log (\exp (u) \exp (v))=u+v+\frac{1}{2}[u,v]+
\frac{1}{12}[u-v,[u,v]]+\cdots,\ {\rm for\ } u,v\in \widehat{\mathcal{L}}.$$

\subsection{Magnus expansion}
We recall the notion of a Magnus expansion in our generalized sense. 
Remark that the subset $1+\widehat{T}_1$ constitutes a group
with respect to the multiplication of $\widehat{T}$.

\begin{dfn}[Kawazumi \cite{Ka}]
A map $\theta\colon \pi \rightarrow 1+\widehat{T}_1$ is called
a {\rm (}$\mathbb{Q}$-valued{\rm )} Magnus expansion of $\pi$ if
\begin{enumerate}
\item[{\rm (1)}]
$\theta\colon \pi \rightarrow 1+\widehat{T}_1$ is a group homomorphism, and
\item[{\rm (2)}]
$\theta(x) \equiv 1+[x]\ {\rm mod\ }\widehat{T}_2$ for any $x\in \pi$.
\end{enumerate}
\end{dfn}

As was shown in \cite{Ka}, Theorem 1.3,
any Magnus expansion $\theta$ induces
the filter-preserving isomorphism
\begin{equation}
\label{eq:2-3-1}
\theta\colon \widehat{\mathbb{Q}\pi}
\stackrel{\cong}{\rightarrow} \widehat{T}
\end{equation}
of augmented algebras.
Since $\pi$ is a free group, any Magnus
expansion is determined by its values on free generators of $\pi$,
hence we have many choices of Magnus expansions (see also \S2.8).
\begin{exple}
\label{ex:2-3-2}
{\rm Let $\alpha_1,\beta_1,\ldots,\alpha_g,\beta_g\in \pi$ be
symplectic generators (see \S2.1) and write them as
$x_1,\ldots, x_{2g}$. The Magnus expansion defined by
$\theta(x_i)=1+[x_i]$, for $1\le i \le 2g$,
is called the standard Magnus expansion.
This is introduced by Magnus \cite{Mag}.}
\end{exple}

Among all the Magnus expansions, group-like expansions
respect the Hopf algebra structure of $\mathbb{Q}\pi$
and $\widehat{T}$. For a Magnus expansion $\theta$, let
$\ell^{\theta}:=\log \theta$. Here it should be remarked 
the logarithm is defined on the set $1+\widehat{T}_1$.
A priori, $\ell^{\theta}$
is a map from $\pi$ to $\widehat{T}_1$.

\begin{dfn}
\label{def:2-3-3}
A Magnus expansion $\theta$ is called group-like if
$\theta(\pi)$ is contained in the set of group-like elements
of $\widehat{T}$, or equivalently,
$\ell^{\theta}(\pi)\subset \widehat{\mathcal{L}}$.
\end{dfn}

If $\theta$ is group-like, (\ref{eq:2-3-1}) turns out to be the
isomorphism of complete Hopf algebras (see Massuyeau \cite{Mas}, Proposition 2.10).
The Magnus expansion of Example \ref{ex:2-3-2} is not group-like. 

\begin{exple}
\label{ex:2-3-4}
{\rm
Let $x_i$ be as the same in Example \ref{ex:2-3-2},
then the Magnus expansion defined by
$\theta(x_i)=\exp([x_i])$, for $1\le i\le 2g$,
is group-like because of the Baker-Campbell-Hausdorff formula.
}
\end{exple}

In fact, by the Baker-Campbell-Hausdorff formula, if
$\ell^{\theta}(x)\in \widehat{\mathcal{L}}$ and
$\ell^{\theta}(y)\in \widehat{\mathcal{L}}$, then
$\ell^{\theta}(xy)\in \widehat{\mathcal{L}}$.
Thus if $\theta$ is a Magnus expansion and the values
of $\theta$ on free generators are all group-like,
then $\theta$ is group-like.
Hence we also have many choices of group-like expansions.

\begin{exple}
\label{ex:2-3-5}
{\rm
Bene-Kawazumi-Penner {\rm \cite{BKP}} constructed a group-like
Magnus expansion canonically associated to any trivalent marked fatgraph.
}
\end{exple}

\subsection{Symplectic expansion}
So far we have only used the fact that $\pi$ is a free group.
Here we recall the notion of a symplectic expansion, which
is a Magnus expansion respecting the fact that $\pi=\pi_1(\Sigma,*)$.

Let $\omega \in \mathcal{L}_2 \subset H^{\otimes 2}$
be the symplectic form. Explicitly, $\omega$ is given by
$$\omega=\sum_{i=1}^g A_iB_i-B_iA_i,$$
where $\alpha_1,\beta_1,\ldots,\alpha_g,\beta_g$ are
symplectic generators and $A_i=[\alpha_i]$, $B_i=[\beta_i]\in H$.

\begin{dfn}[Massuyeau \cite{Mas}]
\label{def:2-4-1}
A Magnus expansion $\theta$ is called a symplectic
expansion if
\begin{enumerate}
\item
$\theta$ is group-like, and
\item
$\theta(\zeta)=\exp(\omega)$,
or equivalently, $\ell^{\theta}(\zeta)=\omega$.
\end{enumerate}
\end{dfn}

Unfortunately, the group-like expansions of Examples
\ref{ex:2-3-4} and \ref{ex:2-3-5} are not symplectic.
But symplectic expansions do exist, and they are infinitely many
(see \S2.8).
Here we list some examples.

\begin{exple}
\label{ex:2-4-2}
{\rm
Kawazumi \rm \cite{Ka2} constructed a symplectic expansion
(with coefficients in $\mathbb{R}$),
called the harmonic Magnus expansion,
associated to any triple $(C,P_0,v)$ where
$C$ is a marked compact Riemann surface,
$P_0\in C$, and $v$ is a non-zero tangent vector at $P_0$.
The construction is transcendental.
}
\end{exple}

\begin{exple}
\label{ex:2-4-3}
{\rm
Massuyeau \rm \cite{Mas} constructed
a symplectic expansion using the LMO functor.
}
\end{exple}

\begin{exple}
\label{ex:2-4-4}
{\rm
There is a canonical way of associating a symplectic
expansion with any (not necessary symplectic) free generators
of $\pi$. The construction is purely combinatorial.
The details of this expansion will be given in \rm \cite{Ku}.
}
\end{exple}

\subsection{Total Johnson map}
We denote by ${\rm Aut}(\widehat{T})$
the set of filter-preserving algebra automorphisms
of $\widehat{T}$, which clearly constitutes a group.
Let $\theta$ be a Magnus expansion of $\pi$.
For $\varphi\in \mathcal{M}_{g,1}$
we use the same letter $\varphi$ for the
induced automorphism of $\pi$, in view of (\ref{eq:2-1-1}).
As a consequence of the isomorphism (\ref{eq:2-3-1}),
for each $\varphi \in \mathcal{M}_{g,1}$
there uniquely exists
$T^{\theta}(\varphi)\in {\rm Aut}(\widehat{T})$
such that
$$T^{\theta}(\varphi) \circ \theta
=\theta \circ \varphi.$$
Let $|\varphi|\colon H\rightarrow H$ be the automorphism of $H$
induced by the action of $\varphi$ on the first homology of $\Sigma$.
We also denote by $|\varphi|\in {\rm Aut}(\widehat{T})$
the automorphism induced by $|\varphi|$. Then
$\tau^{\theta}(\varphi):=T^{\theta}(\varphi)\circ |\varphi|^{-1}
\in {\rm Aut}(\widehat{T})$ acts on
$\widehat{T}_1/\widehat{T}_2 \cong H$ as the identity.
Therefore the restriction of $\tau^{\theta}(\varphi)$ to $H$ is
uniquely written as
$$\tau^{\theta}(\varphi)|_H=1_H+\sum_{k=1}^{\infty}\tau^{\theta}_k(\varphi),$$
where $\tau^{\theta}_k(\varphi)\in {\rm Hom}(H,H^{\otimes k+1})$.

\begin{dfn}[\cite{Ka}]
\label{def:2-5-1}
The automorphism $T^{\theta}(\varphi)\in {\rm Aut}(\widehat{T})$
is called the total Johnson map
of $\varphi$ associated to $\theta$, and $\tau^{\theta}_k(\varphi)$
is called the $k$-th Johnson map of $\varphi$ associated to $\theta$.
\end{dfn}
The group homomorphism
$$T^{\theta}\colon \mathcal{M}_{g,1}\rightarrow {\rm Aut}(\widehat{T})$$
is also called the total Johnson map.
It is injective since the natural map
$\pi\rightarrow \widehat{\mathbb{Q}\pi}$ is injective
by the classical fact $\bigcap_{m=1}^{\infty}I\pi^m=0$.
It should be remarked that our use of the terminology
here is different from \cite{Ka}, where $\tau^{\theta}(\varphi)$ is
called the total Johnson map of $\varphi$.

\subsection{The invariant $L^{\theta}$}
We introduce an invariant of unoriented loops on $\Sigma$
associated with a Magnus expansion.

\begin{dfn}
\label{def:2-6-1}
Define a linear map $N\colon \widehat{T}\rightarrow \widehat{T}$ by
$$N|_{H^{\otimes p}}=\sum_{m=0}^{p-1} \nu^m,\ {\rm for\ } p\ge 1,$$
where $\nu$ is the cyclic permutation given by
$X_1X_2\cdots X_p \mapsto
X_2X_3\cdots X_1$ {\rm (}$X_i \in H$ {\rm )}, and
$N|_{H^{\otimes 0}}=0$.
\end{dfn}

The following lemma will be used frequently.

\begin{lem}
\label{lem:2-6-2}
\begin{enumerate}
\item[{\rm (1)}] For $u,v\in \widehat{T}$, $N(uv)=N(vu)$.
\item[{\rm (2)}] For $u,v,w\in \widehat{T}$, $N([u,v]w)=N(u[v,w])$.
\item[{\rm (3)}] For $v\in \widehat{T}_1$,
$v\in N(\widehat{T}_1)$ is equivalent to $\nu(v)=v$.
\item[{\rm (4)}] Under the identification $\widehat{T}_1\cong H\otimes \widehat{T}$,
\begin{equation}
\label{eq:2-6-1}
N(\widehat{T}_1)=
{\rm Ker}([\ ,\ ]\colon H\otimes \widehat{T}\rightarrow \widehat{T}).
\end{equation}
\end{enumerate}
\end{lem}

\begin{proof}
The first assertion is clear if $u$ and $v$ are homogeneous,
since $N(\nu(w))=N(w)$ for a homogeneous $w\in \widehat{T}$.
The general case follows from bi-linearity.
Using (1), we compute
$N([u,v]w)=N(uvw-vuw)=N(wuv-wvu)=N(uvw-uwv)=N(u[v,w])$, which proves (2).
If $v\in \widehat{T}_1$ is homogeneous of degree $p$ and $\nu(v)=v$,
then $v=N(\frac{1}{p}v)\in N(\widehat{T}_1)$. This proves (3).
Finally, $\nu(X\otimes u)-X\otimes u=uX-Xu=-[X,u]$ for
$X\otimes u \in H\otimes \widehat{T}$. Combining this with (3),
we have (4).
\end{proof}

The operator $N$ also appeared in \cite{Ka2}.
Using $N$, we make the following definition.

\begin{dfn}
\label{def:2-6-3}
Let $\theta$ be a Magnus expansion. Define
$L^{\theta}\colon \pi \rightarrow \widehat{T}_2$ by
$$L^{\theta}(x)=\frac{1}{2}N(\ell^{\theta}(x)\ell^{\theta}(x)).$$
\end{dfn}

The following lemma shows $L^{\theta}$ descends to an invariant
for unoriented loops on $\Sigma$.

\begin{lem}
\label{lem:2-6-4}
For any $x,y\in \pi$, we have
\begin{enumerate}
\item[{\rm (1)}] $L^{\theta}(x^{-1})=L^{\theta}(x)$,
\item[{\rm (2)}] $L^{\theta}(yxy^{-1})=L^{\theta}(x)$.
\end{enumerate}
\end{lem}

\begin{proof}
The first part follows from $\ell^{\theta}(x^{-1})=-\ell^{\theta}(x)$.
Since $\ell^{\theta}(yxy^{-1})=
e^{\ell^{\theta}(y)}\ell^{\theta}(x)e^{-\ell^{\theta}(y)}
=\theta(y)\ell^{\theta}(x)\theta(y^{-1})$,
we compute
\begin{eqnarray*}
L^{\theta}(yxy^{-1})
&=& \frac{1}{2}N(\theta(y)\ell^{\theta}(x)\theta(y^{-1})
\theta(y)\ell^{\theta}(x)\theta(y^{-1}))
= \frac{1}{2}N(\theta(y)\ell^{\theta}(x)\ell^{\theta}(x)\theta(y^{-1})) \\
&=& \frac{1}{2}N(\ell^{\theta}(x)\ell^{\theta}(x)\theta(y^{-1})\theta(y))
= \frac{1}{2}N(\ell^{\theta}(x)\ell^{\theta}(x)) \\
&=& L^{\theta}(x),
\end{eqnarray*}
using Lemma \ref{lem:2-6-2} (1). This proves (2).
\end{proof}

Let $\gamma$ be an (un)oriented loop on $\Sigma$.
In view of Lemma \ref{lem:2-6-4}, we can define
$L^{\theta}(\gamma)\in \widehat{T}_2$ as $L^{\theta}(x)$,
where $x$ is a representative of $\gamma$ in $\pi$.

We denote by $L_k^{\theta}$ the degree $k$-part of $L^{\theta}$.
In \S7, we will compute $L_k^{\theta}$ for symplectic $\theta$ and small $k$.

\subsection{Formal symplectic geometry}
The space $N(\widehat{T}_1)$ is closely related to formal symplectic
geometry. In \cite{Kon}, Kontsevich introduced three
Lie algebras ``commutative", ``associative", and ``Lie". We recall two
of the three, namely ``associative" and ``Lie".

First we recall ``associative".
By definition, a derivation of $\widehat{T}$ is a linear map
$D\colon \widehat{T}\rightarrow \widehat{T}$ satisfying the Leibniz rule:
$$D(u_1u_2)=D(u_1)u_2+u_1D(u_2),\ {\rm for\ } u_1,u_2\in \widehat{T}.$$
The space ${\rm Der}(\widehat{T})$ of the derivations of $\widehat{T}$
has the structure of Lie algebra
given by $[D_1,D_2]=D_1\circ D_2-D_2\circ D_1$,
$D_1,D_2\in {\rm Der}(\widehat{T})$.
Since $\widehat{T}$ is freely generated by $H$ as a complete algebra,
any derivation of $\widehat{T}$ is uniquely determined by its values
on $H$, and ${\rm Der}(\widehat{T})$ is identified with
${\rm Hom}(H,\widehat{T})$.

By the Poincar\'e duality, $\widehat{T}_1\cong H\otimes \widehat{T}$
is identified with ${\rm Hom}(H,\widehat{T})$:
\begin{equation}
\label{eq:2-7-1}
\widehat{T}_1\cong H\otimes \widehat{T} \stackrel{\cong}{\rightarrow}
{\rm Hom}(H,\widehat{T}),\ X\otimes u\mapsto (Y\mapsto (Y\cdot X)u).
\end{equation}
Here $(\ \cdot \ )$ is the intersection pairing on $H=H_1(\Sigma;\mathbb{Q})$.

Let $\mathfrak{a}_g^-={\rm Der}_{\omega}(\widehat{T})$ be the Lie subalgebra
of ${\rm Der}(\widehat{T})$ consisting of derivations
killing the symplectic form $\omega$. We call such derivations
{\it symplectic derivations of} $\widehat{T}$.
In view of (\ref{eq:2-7-1}) any derivation $D$ is written as
\begin{equation}
\label{eq:2-7-2}
D=\sum_{i=1}^g B_i\otimes D(A_i)-A_i\otimes D(B_i)\in \widehat{T}_1.
\end{equation}
Since $D(\omega)=\sum_{i=1}^g [D(A_i),B_i]+[A_i,D(B_i)]$ we can write
\begin{equation}
\label{eq:2-7-3}
\mathfrak{a}_g^-={\rm Ker}([\ ,\ ]\colon H\otimes \widehat{T}
\rightarrow \widehat{T})=N(\widehat{T}_1)
\end{equation}
(see also (\ref{eq:2-6-1})).
The Lie subalgebra $\mathfrak{a}_g:=N(\widehat{T}_2)$
is nothing but (the completion of) what Kontsevich \cite{Kon} calls $a_g$.

We next recall ``Lie".
By definition, a derivation of $\widehat{\mathcal{L}}$ is a linear map
$D\colon \widehat{\mathcal{L}}\rightarrow \widehat{\mathcal{L}}$ satisfying
$$D([u_1,u_2])=[D(u_1),u_2]+[u_1,D(u_2)]
,\ {\rm for\ } u_1,u_2\in \widehat{\mathcal{L}}.$$
Let $\mathfrak{l}_g={\rm Der}_{\omega}(\widehat{\mathcal{L}})$
be the space of derivations of $\widehat{\mathcal{L}}$
killing $\omega \in \mathcal{L}_2$. By the same reason as above, we have
\begin{equation}
\label{eq:2-7-4}
\mathfrak{l}_g={\rm Ker}([\ ,\ ]\colon H\otimes \widehat{\mathcal{L}}
\rightarrow \widehat{\mathcal{L}}).
\end{equation}
$\mathfrak{l}_g$ is a Lie subalgebra of $\mathfrak{a}_g$.

\begin{lem}
\label{lem:2-7-1}
Let $m\ge 1$, and $X,Y_1,\ldots,Y_m \in H$. Set
$u=[Y_1,[Y_2,[\cdots ,[Y_{m-1},Y_m] \cdots ]]]\in \mathcal{L}_m$. Then
$$N(X\otimes u)=X\otimes u+
\sum_{i=1}^m Y_i\otimes
[[Y_{i+1},\cdots [Y_{m-1},Y_m]\cdots ],[\cdots [[X,Y_1],Y_2],\cdots,Y_{i-1}]].$$
In particular, we have $N(H\otimes \widehat{\mathcal{L}})
\subset H\otimes \widehat{\mathcal{L}}$.
\end{lem}

\begin{proof}
Consider the tensor algebra $T^{\prime}$
generated by the letters
$X,Y_1,\ldots,Y_m$. The operator $N$ is naturally defined on $T^{\prime}$.
There is a homomorphism $T^{\prime}\rightarrow \widehat{T}$
coming from the universality of $T^{\prime}$. This homomorphism
is compatible with $N$. Thus it suffices to show the formula on $T^{\prime}$.
Let $H^{\prime}$ be the $\mathbb{Q}$-vector space spanned by
$X,Y_1,\ldots,Y_m$. The formula we want to show is an equality
in ${H^{\prime}}^{\otimes m+1}$. There is a direct sum
decomposition
$${H^{\prime}}^{\otimes m+1}=X\otimes {H^{\prime}}^{\otimes m}
\oplus \bigoplus_{i=1}^m Y_i \otimes {H^{\prime}}^{\otimes m}.$$
Let $p_X\colon
{H^{\prime}}^{\otimes m+1}\rightarrow X\otimes {H^{\prime}}^{\otimes m}
\cong {H^{\prime}}^{\otimes m}$ be the projection according to this
direct sum decomposition. Similarly, define $p_{Y_i}$, $1\le i\le m$.
Note that for any $v\in {H^{\prime}}^{\otimes m+1}$ we have
$v=Xp_X(v)+\sum_{i=1}^mY_ip_{Y_i}(v)$.
Now, set $v:=N(X\otimes u)$. It is clear that
$p_X(v)=u$. For each $1\le i\le m$,
we denote $v^{\prime}=[Y_{i+1},\cdots [Y_{m-1},Y_m]\cdots ]$.
By Lemma \ref{lem:2-6-2}, we compute
\begin{eqnarray*}
v &= &N(X[Y_1,[Y_2,\cdots,[Y_i,v^{\prime}]\cdots ]]) \\
&=& N([X,Y_1][Y_2,\cdots,[Y_i,v^{\prime}]\cdots ]) \\
&\cdots & \\
&=& N(v^{\prime \prime}[Y_i,v^{\prime}]) \\
&=& N(v^{\prime \prime}Y_iv^{\prime}-v^{\prime \prime}v^{\prime}Y_i)
=N(Y_iv^{\prime}v^{\prime \prime}-Y_iv^{\prime \prime}v^{\prime})
=N(Y_i[v^{\prime},v^{\prime \prime}]),
\end{eqnarray*}
where $v^{\prime \prime}=[\cdots [[X,Y_1],Y_2],\cdots,Y_{i-1}]$.
This shows $p_{Y_i}(v)=[v^{\prime},v^{\prime \prime}]$,
and completes the proof.
\end{proof}

\begin{lem}
\label{lem:2-7-2}
$$N(\widehat{\mathcal{L}}\hat{\otimes} \widehat{\mathcal{L}})
={\rm Ker}([\ ,\ ]\colon H\otimes \widehat{\mathcal{L}}
\rightarrow \widehat{\mathcal{L}})$$
\end{lem}

\begin{proof}
Using Lemma \ref{lem:2-6-2} (2), we have
$N(\widehat{\mathcal{L}}\hat{\otimes} \widehat{\mathcal{L}})=
N(H\otimes \widehat{\mathcal{L}})$, and
$N(H\otimes \widehat{\mathcal{L}})$ is contained in
$(H\otimes \widehat{\mathcal{L}})\cap N(\widehat{T}_1)$
by Lemma \ref{lem:2-7-1}.
Therefore we get
$N(\widehat{\mathcal{L}}\hat{\otimes} \widehat{\mathcal{L}})
\subset {\rm Ker}([\ ,\ ]\colon H\otimes \widehat{\mathcal{L}}
\rightarrow \widehat{\mathcal{L}})$ by (\ref{eq:2-6-1}).
On the other hand,
if $v\in H\otimes \widehat{\mathcal{L}}\subset \widehat{T}_1$
is homogeneous of degree $p\ge 2$
and $\nu(v)=v$, then
$v=N(v/p)\in N(H\otimes \widehat{\mathcal{L}})=
N(\widehat{\mathcal{L}}\hat{\otimes} \widehat{\mathcal{L}})$.
By Lemma \ref{lem:2-6-2} (3)(4), we get the other inclusion.
\end{proof}

Thus, if $\theta$ is group-like, our invariant $L^{\theta}$ is considered as a map
$L^{\theta}\colon \pi \rightarrow \mathfrak{l}_g$.

\subsection{The space of symplectic expansions}
There are infinitely many Magnus expansions and symplectic expansions.
Here we consider the spaces that parametrize them.
Let $\Theta$ be the set of Magnus expansions of $\pi$, and
let $\Theta^{\rm symp}\subset \Theta$ be the set of symplectic expansions.

Let ${\rm IA}(\widehat{T})$ be the subgroup of ${\rm Aut}(\widehat{T})$
consisting of the automorphisms
acting on $\widehat{T}_1/\widehat{T}_2
\cong H$ as the identity.
Let $\theta$ and $\theta^{\prime}$ be Magnus expansions.
By \cite{Ka}, Theorem 1.3, there uniquely exists
$U=U(\theta,\theta^{\prime})
\in {\rm IA}(\widehat{T})$ such that $\theta^{\prime}=U\circ \theta$.
Conversely, for $\theta \in \Theta$ and $U\in {\rm IA}(\widehat{T})$,
$U\circ \theta$ is a Magnus expansion. Thus if we fix $\theta$,
$\Theta$ is identified with ${\rm IA}(\widehat{T})$
by $\theta^{\prime}\mapsto U(\theta,\theta^{\prime})$.

The group ${\rm IA}(\widehat{T})$ is identified with
its "tangent space" ${\rm Hom}(H,\widehat{T}_2)$, by the logarithms:
$${\rm IA}(\widehat{T})\rightarrow {\rm Hom}(H,\widehat{T}_2),\ 
U\mapsto (\log U)|_H.$$
Note that $\log U$ converges since $U$ acts on
$\widehat{T}_1/\widehat{T}_2\cong H$ as the identity.
We can regard ${\rm Hom}(H,\widehat{T}_2)$ as the space of
derivations of $\widehat{T}$ with positive degrees.
By the Poincar\'e duality (\ref{eq:2-7-1}),
${\rm Hom}(H,\widehat{T}_2)$ is identified with $H\otimes \widehat{T}_2$.
In this way we have an identification
\begin{equation}
\label{eq:2-8-1}
{\rm IA}(\widehat{T})\cong H\otimes \widehat{T}_2,
\end{equation}
and if we fix $\theta$ there is a canonical bijection
$\Theta \cong {\rm IA}(\widehat{T}) \cong H\otimes \widehat{T}_2$.

\begin{prop}
\label{prop:2-8-1}
The set $\Theta^{\rm symp}$ is not empty.
Once we choose a symplectic expansion $\theta$,
the restriction of the canonical bijection
$\Theta \cong H\otimes \widehat{T}_2$ to $\Theta^{\rm symp}$
gives a bijection $$\Theta^{\rm symp}\cong
{\rm Ker}([\ ,\ ]\colon H\otimes \widehat{\mathcal{L}}_2
\rightarrow \widehat{\mathcal{L}}).$$
\end{prop}

\begin{proof}
The examples of symplectic expansions given
in \S2.4 show that $\Theta^{\rm symp}$ is not empty.
Or, see Massuyeau \cite{Mas} Lemma 2.16.
We will prove the latter part.

Suppose $\theta$ and $\theta^{\prime}$ are symplectic.
Since both of the two are group-like and
$\theta(\zeta)=\theta^{\prime}(\zeta)=\omega$,
the automorphism $U=U(\theta,\theta^{\prime})$ must satisfy
\begin{equation}
\label{eq:2-8-2}
U(H)\subset \widehat{\mathcal{L}},\ {\rm and\ } U(\omega)=\omega.
\end{equation}
Conversely for $\theta\in \Theta^{\rm symp}$
and $U\in {\rm IA}(\widehat{T})$ satisfying (\ref{eq:2-8-2}),
$U\circ \theta$ is symplectic.

Let $U\in {\rm IA}(\widehat{T})$.
Under the identification (\ref{eq:2-8-1}),
$U(H)\subset \widehat{\mathcal{L}}$ is equivalent to
$(\log U)|_H \in H\otimes \widehat{\mathcal{L}}_2$.
Also, $U(\omega)=\omega$ is equivalent to $\log U(\omega)=0$,
where $\log U$ acts on $\widehat{\mathcal{L}}$ as a derivation.
By (\ref{eq:2-7-4}), this is equivalent to $(\log U)|_H\in
{\rm Ker}([\ ,\ ]\colon H\otimes \widehat{\mathcal{L}}_2
\rightarrow \widehat{\mathcal{L}})$. This completes the proof.
\end{proof}

\section{The Goldman Lie algebra}
In this section, we recall the Goldman Lie algebra \cite{Go}.
In particular, we show that the Goldman Lie algebra of
$\Sigma$ acts on the group ring $\mathbb{Q}\pi$ as a derivation.
We will work over the rationals, but
all the statements in this section except
Proposition \ref{prop:3-4-3} holds over the integers.

All of the loops that we consider are piecewise differentiable.

\subsection{The Goldman Lie algebra}
Let $S$ be a connected oriented 2-manifold and let
$\hat{\pi}(S)=[S^1,S]$ be the set of free homotopy classes
of oriented loops on $S$. In other words, $\hat{\pi}(S)$
is the set of conjugacy classes of the fundamental group of $S$.
Let $|\ |\colon \pi_1(S)\rightarrow \hat{\pi}(S)$
be the natural quotient map.
For a loop $\alpha\colon S^1\rightarrow S$ and a simple point $p\in \alpha$,
let $\alpha_p$ be the oriented loop $\alpha$ based at $p$.

Let $\mathbb{Q}\hat{\pi}(S)$ be the vector space spanned by $\hat{\pi}(S)$.
We first recall the Goldman bracket on $\mathbb{Q}\hat{\pi}(S)$.
Let $\alpha,\beta$ be immersed loops in $S$ such that
$\alpha \cup \beta \colon S^1\cup S^1 \rightarrow S$
is an immersion with at worst transverse double points.
For each intersection $p\in \alpha \cap \beta$,
the conjunction $\alpha_p\beta_p\in \pi_1(S,p)$ is defined.
Let $\varepsilon(p;\alpha,\beta)\in \{ \pm 1 \}$ be
the local intersection number of $\alpha$ and $\beta$ at $p$ and set
$$[\alpha,\beta]:=\sum_{p\in \alpha \cap \beta}
\varepsilon(p;\alpha,\beta)|\alpha_p\beta_p| \in \mathbb{Q}\hat{\pi}(S).$$
Let $1\in \hat{\pi}(S)$ be the homotopy
class of the constant loop and let
$\hat{\pi}^{\prime}(S)=\hat{\pi}(S)\setminus \{ 1\}$.

\begin{thm}[Goldman \cite{Go}]
\label{thm:3-1-1}
The above bracket defines a well-defined linear map
$$[\ ,\ ]\colon \mathbb{Q}\hat{\pi}(S) \otimes \mathbb{Q}\hat{\pi}(S)
\rightarrow \mathbb{Q}\hat{\pi}(S),$$
and with respect to this bracket $\mathbb{Q}\hat{\pi}(S)$
has a structure of Lie algebra.
Moreover, $\mathbb{Q}\hat{\pi}^{\prime}(S)$ is an ideal of
$\mathbb{Q}\hat{\pi}(S)$ and
$\mathbb{Q}\hat{\pi}(S)=\mathbb{Q}\hat{\pi}^{\prime}(S)
\oplus \mathbb{Q}1$ is a direct sum decomposition as Lie algebras.
\end{thm}

\begin{rem}{\rm
It is true that $[\mathbb{Q}\hat{\pi},
\mathbb{Q}\hat{\pi}] \subset \mathbb{Q}\hat{\pi}'$.
But Goldman's proof for it \cite{Go} pp.294-295 is, unfortunately,
not true. In fact, his assertion $[\alpha, \alpha^{-1}] = 0$
for $\alpha \in \hat{\pi}$ is not true in general.
If we choose $\alpha=\alpha_1\alpha_2$ as in Figure 3, then $[\alpha, \alpha^{-1}]
= \alpha_1\alpha_2{\alpha_1}^{-1}{\alpha_2}^{-1} -
\alpha_2\alpha_1{\alpha_2}^{-1}{\alpha_1}^{-1}$.
For a symplectic expansion $\theta$, we have
$$
N\theta[\alpha,\alpha^{-1}] =
\frac{1}{3}N([X_1,X_2][X_1,X_2][X_1,X_2]) +
\mbox{higher terms} \neq 0
$$
(see Theorem \ref{thm:1-2-1}).
Here we denote $X_1 = [\alpha_1]$ and $X_2 =
[\alpha_2]
\in H$. Hence $[\alpha, \alpha^{-1}] \neq 0$.

\begin{center}
Figure 3: $[\alpha,\alpha^{-1}]\neq 0$

\unitlength 0.1in
\begin{picture}( 44.9000, 13.2400)(  2.8000,-16.3000)
%
\special{pn 13}%
\special{ar 2266 760 406 406  0.3217506 2.8198421}%
%
\special{pn 13}%
\special{ar 2266 1144 384 384  3.7146284 5.6900395}%
%
\special{pn 13}%
\special{ar 3864 760 406 406  0.3217506 2.8198421}%
%
\special{pn 13}%
\special{ar 3864 1144 384 384  3.7160247 5.6915170}%
%
\special{pn 20}%
\special{sh 1}%
\special{ar 2896 920 10 10 0  6.28318530717959E+0000}%
\special{sh 1}%
\special{ar 2896 920 10 10 0  6.28318530717959E+0000}%
%
\special{pn 8}%
\special{pa 2920 952}%
\special{pa 2936 980}%
\special{pa 2952 1008}%
\special{pa 2968 1036}%
\special{pa 2986 1064}%
\special{pa 3002 1092}%
\special{pa 3018 1118}%
\special{pa 3036 1146}%
\special{pa 3054 1172}%
\special{pa 3072 1198}%
\special{pa 3092 1224}%
\special{pa 3112 1248}%
\special{pa 3132 1274}%
\special{pa 3154 1298}%
\special{pa 3176 1320}%
\special{pa 3200 1342}%
\special{pa 3224 1364}%
\special{pa 3250 1386}%
\special{pa 3276 1406}%
\special{pa 3302 1424}%
\special{pa 3330 1444}%
\special{pa 3358 1460}%
\special{pa 3386 1478}%
\special{pa 3416 1494}%
\special{pa 3446 1508}%
\special{pa 3476 1522}%
\special{pa 3508 1536}%
\special{pa 3540 1548}%
\special{pa 3572 1560}%
\special{pa 3604 1570}%
\special{pa 3638 1578}%
\special{pa 3670 1588}%
\special{pa 3704 1594}%
\special{pa 3738 1600}%
\special{pa 3772 1606}%
\special{pa 3806 1610}%
\special{pa 3840 1612}%
\special{pa 3874 1614}%
\special{pa 3908 1616}%
\special{pa 3942 1614}%
\special{pa 3976 1612}%
\special{pa 4010 1610}%
\special{pa 4044 1606}%
\special{pa 4078 1600}%
\special{pa 4110 1594}%
\special{pa 4144 1586}%
\special{pa 4176 1578}%
\special{pa 4208 1568}%
\special{pa 4242 1556}%
\special{pa 4272 1542}%
\special{pa 4304 1528}%
\special{pa 4334 1512}%
\special{pa 4364 1496}%
\special{pa 4394 1478}%
\special{pa 4422 1458}%
\special{pa 4450 1438}%
\special{pa 4476 1418}%
\special{pa 4502 1394}%
\special{pa 4526 1372}%
\special{pa 4548 1348}%
\special{pa 4570 1322}%
\special{pa 4590 1296}%
\special{pa 4610 1270}%
\special{pa 4626 1242}%
\special{pa 4642 1214}%
\special{pa 4656 1184}%
\special{pa 4668 1154}%
\special{pa 4678 1124}%
\special{pa 4686 1094}%
\special{pa 4692 1062}%
\special{pa 4696 1030}%
\special{pa 4698 998}%
\special{pa 4698 964}%
\special{pa 4696 932}%
\special{pa 4692 898}%
\special{pa 4686 866}%
\special{pa 4680 834}%
\special{pa 4670 800}%
\special{pa 4660 768}%
\special{pa 4648 736}%
\special{pa 4634 706}%
\special{pa 4618 674}%
\special{pa 4602 644}%
\special{pa 4584 616}%
\special{pa 4564 588}%
\special{pa 4544 560}%
\special{pa 4522 534}%
\special{pa 4500 508}%
\special{pa 4476 486}%
\special{pa 4450 464}%
\special{pa 4424 442}%
\special{pa 4398 424}%
\special{pa 4370 406}%
\special{pa 4342 390}%
\special{pa 4312 376}%
\special{pa 4282 362}%
\special{pa 4252 350}%
\special{pa 4220 340}%
\special{pa 4190 332}%
\special{pa 4156 324}%
\special{pa 4124 318}%
\special{pa 4092 314}%
\special{pa 4058 310}%
\special{pa 4024 308}%
\special{pa 3990 306}%
\special{pa 3956 306}%
\special{pa 3922 308}%
\special{pa 3888 310}%
\special{pa 3854 314}%
\special{pa 3820 318}%
\special{pa 3786 324}%
\special{pa 3752 332}%
\special{pa 3718 338}%
\special{pa 3684 348}%
\special{pa 3650 358}%
\special{pa 3618 368}%
\special{pa 3586 380}%
\special{pa 3554 392}%
\special{pa 3522 404}%
\special{pa 3492 418}%
\special{pa 3460 434}%
\special{pa 3432 450}%
\special{pa 3402 466}%
\special{pa 3374 482}%
\special{pa 3346 500}%
\special{pa 3320 518}%
\special{pa 3294 538}%
\special{pa 3268 556}%
\special{pa 3242 576}%
\special{pa 3218 598}%
\special{pa 3194 618}%
\special{pa 3170 640}%
\special{pa 3146 662}%
\special{pa 3124 684}%
\special{pa 3100 706}%
\special{pa 3078 728}%
\special{pa 3056 752}%
\special{pa 3034 774}%
\special{pa 3012 798}%
\special{pa 2990 820}%
\special{pa 2968 844}%
\special{pa 2948 868}%
\special{pa 2928 888}%
\special{sp}%
%
\special{pn 13}%
\special{ar 664 758 406 406  0.3217506 2.8198421}%
%
\special{pn 13}%
\special{ar 664 1142 384 384  3.7160247 5.6915170}%
%
\special{pn 8}%
\special{pa 2880 894}%
\special{pa 2864 864}%
\special{pa 2846 836}%
\special{pa 2830 808}%
\special{pa 2812 780}%
\special{pa 2794 754}%
\special{pa 2774 730}%
\special{pa 2752 706}%
\special{pa 2730 684}%
\special{pa 2706 664}%
\special{pa 2680 648}%
\special{pa 2654 632}%
\special{pa 2626 620}%
\special{pa 2596 610}%
\special{pa 2564 600}%
\special{pa 2532 592}%
\special{pa 2500 586}%
\special{pa 2468 582}%
\special{pa 2434 578}%
\special{pa 2400 576}%
\special{pa 2366 574}%
\special{pa 2332 574}%
\special{pa 2300 574}%
\special{pa 2266 574}%
\special{pa 2234 576}%
\special{pa 2202 578}%
\special{pa 2170 580}%
\special{pa 2138 584}%
\special{pa 2106 590}%
\special{pa 2076 596}%
\special{pa 2044 604}%
\special{pa 2014 612}%
\special{pa 1984 622}%
\special{pa 1954 634}%
\special{pa 1924 648}%
\special{pa 1894 664}%
\special{pa 1864 680}%
\special{pa 1834 700}%
\special{pa 1806 720}%
\special{pa 1778 742}%
\special{pa 1752 766}%
\special{pa 1730 790}%
\special{pa 1708 816}%
\special{pa 1690 842}%
\special{pa 1674 870}%
\special{pa 1660 898}%
\special{pa 1652 926}%
\special{pa 1648 956}%
\special{pa 1648 986}%
\special{pa 1652 1016}%
\special{pa 1658 1046}%
\special{pa 1670 1076}%
\special{pa 1684 1104}%
\special{pa 1700 1134}%
\special{pa 1720 1162}%
\special{pa 1742 1188}%
\special{pa 1766 1214}%
\special{pa 1792 1240}%
\special{pa 1820 1264}%
\special{pa 1848 1286}%
\special{pa 1878 1306}%
\special{pa 1908 1324}%
\special{pa 1938 1340}%
\special{pa 1968 1354}%
\special{pa 2000 1366}%
\special{pa 2032 1378}%
\special{pa 2064 1386}%
\special{pa 2094 1392}%
\special{pa 2126 1398}%
\special{pa 2158 1400}%
\special{pa 2190 1402}%
\special{pa 2222 1400}%
\special{pa 2254 1398}%
\special{pa 2286 1394}%
\special{pa 2318 1388}%
\special{pa 2350 1382}%
\special{pa 2380 1372}%
\special{pa 2410 1362}%
\special{pa 2440 1350}%
\special{pa 2470 1336}%
\special{pa 2500 1322}%
\special{pa 2528 1306}%
\special{pa 2556 1290}%
\special{pa 2582 1272}%
\special{pa 2608 1252}%
\special{pa 2634 1232}%
\special{pa 2658 1210}%
\special{pa 2682 1188}%
\special{pa 2704 1166}%
\special{pa 2726 1142}%
\special{pa 2748 1118}%
\special{pa 2768 1094}%
\special{pa 2788 1068}%
\special{pa 2808 1044}%
\special{pa 2826 1018}%
\special{pa 2846 992}%
\special{pa 2864 966}%
\special{pa 2880 942}%
\special{sp}%
%
\special{pn 8}%
\special{pa 3008 1102}%
\special{pa 3030 1078}%
\special{pa 3052 1054}%
\special{pa 3074 1032}%
\special{pa 3096 1008}%
\special{pa 3118 986}%
\special{pa 3142 962}%
\special{pa 3164 940}%
\special{pa 3186 918}%
\special{pa 3210 896}%
\special{pa 3234 874}%
\special{pa 3258 852}%
\special{pa 3282 832}%
\special{pa 3306 810}%
\special{pa 3332 790}%
\special{pa 3356 770}%
\special{pa 3382 752}%
\special{pa 3408 734}%
\special{pa 3434 716}%
\special{pa 3462 698}%
\special{pa 3490 680}%
\special{pa 3516 664}%
\special{pa 3546 650}%
\special{pa 3574 634}%
\special{pa 3602 622}%
\special{pa 3632 608}%
\special{pa 3662 596}%
\special{pa 3692 586}%
\special{pa 3722 576}%
\special{pa 3754 566}%
\special{pa 3786 560}%
\special{pa 3816 554}%
\special{pa 3848 548}%
\special{pa 3880 544}%
\special{pa 3914 542}%
\special{pa 3946 540}%
\special{pa 3978 542}%
\special{pa 4012 544}%
\special{pa 4044 548}%
\special{pa 4078 554}%
\special{pa 4108 560}%
\special{pa 4140 570}%
\special{pa 4170 580}%
\special{pa 4200 594}%
\special{pa 4228 608}%
\special{pa 4256 626}%
\special{pa 4282 646}%
\special{pa 4306 666}%
\special{pa 4330 690}%
\special{pa 4350 716}%
\special{pa 4368 742}%
\special{pa 4386 770}%
\special{pa 4400 800}%
\special{pa 4412 830}%
\special{pa 4422 862}%
\special{pa 4430 894}%
\special{pa 4434 928}%
\special{pa 4436 960}%
\special{pa 4434 994}%
\special{pa 4430 1028}%
\special{pa 4422 1060}%
\special{pa 4414 1092}%
\special{pa 4402 1122}%
\special{pa 4386 1150}%
\special{pa 4370 1178}%
\special{pa 4350 1202}%
\special{pa 4328 1224}%
\special{pa 4304 1244}%
\special{pa 4278 1264}%
\special{pa 4252 1280}%
\special{pa 4222 1294}%
\special{pa 4192 1308}%
\special{pa 4162 1320}%
\special{pa 4130 1332}%
\special{pa 4098 1342}%
\special{pa 4066 1352}%
\special{pa 4034 1360}%
\special{pa 4002 1368}%
\special{pa 3970 1374}%
\special{pa 3938 1380}%
\special{pa 3906 1384}%
\special{pa 3874 1388}%
\special{pa 3842 1390}%
\special{pa 3810 1390}%
\special{pa 3778 1388}%
\special{pa 3746 1386}%
\special{pa 3716 1380}%
\special{pa 3684 1374}%
\special{pa 3654 1366}%
\special{pa 3624 1354}%
\special{pa 3592 1342}%
\special{pa 3564 1330}%
\special{pa 3534 1314}%
\special{pa 3506 1298}%
\special{pa 3478 1280}%
\special{pa 3452 1262}%
\special{pa 3426 1240}%
\special{pa 3400 1220}%
\special{pa 3376 1198}%
\special{pa 3354 1174}%
\special{pa 3332 1152}%
\special{pa 3312 1126}%
\special{pa 3292 1102}%
\special{pa 3272 1076}%
\special{pa 3254 1050}%
\special{pa 3236 1024}%
\special{pa 3220 996}%
\special{pa 3202 968}%
\special{pa 3186 942}%
\special{pa 3170 914}%
\special{pa 3154 886}%
\special{pa 3138 858}%
\special{pa 3122 830}%
\special{pa 3106 802}%
\special{pa 3090 774}%
\special{pa 3072 746}%
\special{pa 3056 720}%
\special{pa 3038 694}%
\special{pa 3018 668}%
\special{pa 3000 642}%
\special{pa 2978 618}%
\special{pa 2958 594}%
\special{pa 2934 570}%
\special{pa 2910 548}%
\special{pa 2886 528}%
\special{pa 2860 508}%
\special{pa 2834 488}%
\special{pa 2806 472}%
\special{pa 2778 456}%
\special{pa 2750 442}%
\special{pa 2720 430}%
\special{pa 2690 418}%
\special{pa 2660 410}%
\special{pa 2628 402}%
\special{pa 2598 396}%
\special{pa 2566 392}%
\special{pa 2534 388}%
\special{pa 2502 386}%
\special{pa 2470 384}%
\special{pa 2436 382}%
\special{pa 2404 382}%
\special{pa 2372 382}%
\special{pa 2338 382}%
\special{pa 2306 382}%
\special{pa 2274 382}%
\special{pa 2242 384}%
\special{pa 2210 384}%
\special{pa 2178 386}%
\special{pa 2146 390}%
\special{pa 2114 392}%
\special{pa 2082 396}%
\special{pa 2050 400}%
\special{pa 2018 406}%
\special{pa 1988 412}%
\special{pa 1956 420}%
\special{pa 1926 428}%
\special{pa 1894 438}%
\special{pa 1864 448}%
\special{pa 1834 460}%
\special{pa 1804 474}%
\special{pa 1776 488}%
\special{pa 1748 504}%
\special{pa 1720 520}%
\special{pa 1692 538}%
\special{pa 1666 558}%
\special{pa 1640 578}%
\special{pa 1616 600}%
\special{pa 1594 622}%
\special{pa 1572 646}%
\special{pa 1550 672}%
\special{pa 1532 698}%
\special{pa 1514 726}%
\special{pa 1496 756}%
\special{pa 1482 784}%
\special{pa 1468 816}%
\special{pa 1458 846}%
\special{pa 1448 878}%
\special{pa 1440 910}%
\special{pa 1434 942}%
\special{pa 1430 974}%
\special{pa 1428 1006}%
\special{pa 1430 1038}%
\special{pa 1432 1070}%
\special{pa 1438 1102}%
\special{pa 1446 1134}%
\special{pa 1456 1164}%
\special{pa 1468 1194}%
\special{pa 1482 1224}%
\special{pa 1498 1252}%
\special{pa 1514 1280}%
\special{pa 1534 1308}%
\special{pa 1554 1334}%
\special{pa 1576 1358}%
\special{pa 1598 1382}%
\special{pa 1622 1406}%
\special{pa 1648 1428}%
\special{pa 1672 1448}%
\special{pa 1700 1466}%
\special{pa 1726 1484}%
\special{pa 1754 1500}%
\special{pa 1784 1516}%
\special{pa 1812 1530}%
\special{pa 1842 1542}%
\special{pa 1872 1554}%
\special{pa 1902 1564}%
\special{pa 1934 1574}%
\special{pa 1966 1582}%
\special{pa 1996 1588}%
\special{pa 2028 1594}%
\special{pa 2060 1600}%
\special{pa 2092 1602}%
\special{pa 2126 1606}%
\special{pa 2158 1606}%
\special{pa 2190 1608}%
\special{pa 2222 1606}%
\special{pa 2254 1604}%
\special{pa 2286 1602}%
\special{pa 2318 1598}%
\special{pa 2350 1592}%
\special{pa 2382 1586}%
\special{pa 2414 1578}%
\special{pa 2444 1570}%
\special{pa 2476 1560}%
\special{pa 2506 1548}%
\special{pa 2536 1536}%
\special{pa 2566 1524}%
\special{pa 2594 1510}%
\special{pa 2624 1494}%
\special{pa 2652 1478}%
\special{pa 2678 1462}%
\special{pa 2706 1444}%
\special{pa 2732 1424}%
\special{pa 2756 1404}%
\special{pa 2780 1384}%
\special{pa 2804 1362}%
\special{pa 2826 1338}%
\special{pa 2848 1316}%
\special{pa 2870 1292}%
\special{pa 2890 1266}%
\special{pa 2910 1242}%
\special{pa 2928 1216}%
\special{pa 2948 1190}%
\special{pa 2966 1164}%
\special{pa 2984 1138}%
\special{pa 3002 1110}%
\special{pa 3008 1102}%
\special{sp}%
%
\special{pn 20}%
\special{sh 1}%
\special{ar 3008 1102 10 10 0  6.28318530717959E+0000}%
%
\special{pn 20}%
\special{sh 1}%
\special{ar 3064 742 10 10 0  6.28318530717959E+0000}%
\put(27.0000,-9.9000){\makebox(0,0)[lb]{$p$}}%
\put(47.7000,-12.1000){\makebox(0,0)[lb]{$\alpha_1$}}%
\put(19.0000,-13.4000){\makebox(0,0)[lb]{$\alpha_2$}}%
%
\special{pn 8}%
\special{pa 2290 1620}%
\special{pa 2366 1588}%
\special{fp}%
\special{pa 2366 1588}%
\special{pa 2282 1582}%
\special{fp}%
%
\special{pn 8}%
\special{pa 3860 1630}%
\special{pa 3942 1614}%
\special{fp}%
\special{pa 3942 1614}%
\special{pa 3860 1590}%
\special{fp}%
%
\special{pn 8}%
\special{pa 2260 1420}%
\special{pa 2180 1398}%
\special{fp}%
\special{pa 2180 1398}%
\special{pa 2262 1382}%
\special{fp}%
%
\special{pn 8}%
\special{pa 3880 1410}%
\special{pa 3800 1386}%
\special{fp}%
\special{pa 3800 1386}%
\special{pa 3884 1370}%
\special{fp}%
\put(26.5000,-16.0000){\makebox(0,0)[lb]{$\alpha^{-1}$}}%
\put(36.8000,-13.3000){\makebox(0,0)[lb]{$\alpha^{-1}$}}%
\put(29.2000,-13.5000){\makebox(0,0)[lb]{$+$}}%
\put(30.0000,-6.3000){\makebox(0,0)[lb]{$-$}}%
\end{picture}%

\end{center}

But we can prove
$[\alpha, \alpha^{-1}] \in \mathbb{Q}\hat{\pi}'$ for any $\alpha
\in \hat{\pi}$, as follows.
Represent $\alpha$ by a generic immersion and let $\alpha^{-1}$
be a generic immersion such that $\alpha\cup \alpha^{-1}$ cobounds
a narrow annulus, as in \cite{Go}, p.295. Let $p$ be a double
point of the loop $\alpha$. It divides the loop $\alpha$ into
two based loops $\alpha_1$ and $\alpha_2$ with basepoint $p$
as in Figure 3. The two intersection points derived from $p$
contributes $\alpha_1\alpha_2{\alpha_1}^{-1}{\alpha_2}^{-1}$
and $\alpha_2\alpha_1{\alpha_2}^{-1}{\alpha_1}^{-1}$,
respectively, with the opposite sign. Then the following three
conditions are equivalent to each other:
\begin{enumerate}
\item $\vert\alpha_1\alpha_2{\alpha_1}^{-1}{\alpha_2}^{-1}\vert
= 1 \in \hat{\pi}$,
\item $\alpha_1\alpha_2 = \alpha_2\alpha_1 \in \pi_1(S, p)$,
\item $\vert\alpha_2\alpha_1{\alpha_2}^{-1}{\alpha_1}^{-1}\vert
= 1 \in \hat{\pi}$.
\end{enumerate}
This implies the contributions of the two points cancel, or are in
$\mathbb{Q}\hat{\pi}'$. Hence we have
$[\alpha, \alpha^{-1}] \in \mathbb{Q}\hat{\pi}'$.
As is observed by Goldman \cite{Go} loc.cit.,
$[\alpha, \beta] \in \mathbb{Q}\hat{\pi}'$ if $\beta\neq\alpha^{-1}$.
Hence we obtain
$[\mathbb{Q}\hat{\pi}, \mathbb{Q}\hat{\pi}] \subset
\mathbb{Q}\hat{\pi}'$. This completes the proof of the
second half of Theorem \ref{thm:3-1-1}.
}
\end{rem}

\subsection{The action on the group ring}
Let $S$ be as above, and choose a basepoint $*\in S$.
Let $\alpha\colon S^1\rightarrow S\setminus \{ *\}$ be
an immersed loop and $\beta \colon S^1\rightarrow S$
an immersed loop based at $*$, and suppose 
$\alpha \cup \beta$ has at worst transverse double points.
For each intersection $p\in \alpha \cap \beta$, let $\alpha_p$
and $\varepsilon(p;\alpha,\beta)$ be the same as before and
let $\beta_{*p}$ (resp. $\beta_{p*}$)
be the path along $\beta$ from $*$ to $p$ (resp. $p$ to $*$).
Then the conjunction $\beta_{*p}\alpha_p\beta_{p*}\in \pi_1(S,*)$
is defined.

\begin{dfn}
\label{def:3-2-1}
For such $\alpha$ and $\beta$, let
$$\sigma(\alpha)\beta:=\sum_{p\in \alpha \cap \beta}
\varepsilon(p;\alpha,\beta) \beta_{*p}\alpha_p\beta_{p*}
\in \mathbb{Q}\pi_1(S,*).$$
\end{dfn}

Let ${\rm Der}(\mathbb{Q}\pi_1(S,*))$ be the
Lie algebra of the derivations
of the group ring $\mathbb{Q}\pi_1(S,*)$.

\begin{prop}
\label{prop:3-2-2}
This definition of $\sigma$ gives rise to a well-defined
homomorphism
$$\sigma\colon \mathbb{Q}\hat{\pi}(S\setminus \{ * \})
\rightarrow {\rm Der}(\mathbb{Q}\pi_1(S,*))$$
of Lie algebras.
\end{prop}

\begin{proof}
One way to prove that $\sigma$ is well-defined is to show that
$\sigma(\alpha)\beta$ is unchanged if $\alpha$ and $\beta$
are replaced by one of the standard moves (see Goldman \cite{Go}, Lemma 5.6).
This can be done by the same argument as Goldman did, so we omit details.
Another way to see this is using our homological interpretation of $\sigma$,
see Proposition \ref{prop:3-5-2}.

To prove that $\mathbb{Q}\hat{\pi}(S\setminus \{ * \})$ acts
on $\mathbb{Q}\pi_1(S,*)$ as derivation via $\sigma$, it suffices
to show $\sigma(\alpha)(\beta \gamma)
=(\sigma(\alpha)\beta) \gamma+\beta \sigma(\alpha)\gamma$,
where $\alpha$ is an immersed loop on $S$, and $\beta,\gamma$ are
immersed based loops on $S$.
We may assume $\alpha$ intersects the conjunction $\beta \gamma$
at worst transverse double points.
Then $\alpha \cap (\beta \gamma)=(\alpha \cap \beta)\cup (\alpha \cap \gamma)$,
and
\begin{eqnarray*}
\sigma(\alpha)(\beta \gamma)
&=& \sum_{p\in \alpha \cap (\beta \gamma)}
\varepsilon(p;\alpha,\beta \gamma)(\beta\gamma)_{*p}\alpha_p
(\beta\gamma)_{p*} \\
&=& \sum_{p\in \alpha \cap \beta}
\varepsilon(p;\alpha,\beta \gamma)(\beta\gamma)_{*p}\alpha_p
(\beta\gamma)_{p*}+
\sum_{p\in \alpha \cap \gamma}
\varepsilon(p;\alpha,\beta \gamma)(\beta\gamma)_{*p}\alpha_p
(\beta\gamma)_{p*} \\
&=& \sum_{p\in \alpha \cap \beta}
\varepsilon(p;\alpha,\beta)\beta_{*p}\alpha_p \beta_{p*}\gamma+
\sum_{p\in \alpha \cap \gamma}
\varepsilon(p;\alpha,\gamma)\beta \gamma_{*p}\alpha_p \gamma_{p*} \\
&=& (\sigma(\alpha)\beta) \gamma+\beta \sigma(\alpha)\gamma.
\end{eqnarray*}
To prove that $\sigma$ is a homomorphism of Lie algebras it
suffices to show $\sigma([\alpha,\beta])\gamma=
\sigma(\alpha)\sigma(\beta)\gamma-\sigma(\beta)\sigma(\alpha)\gamma$,
where $\alpha,\beta$ are immersed loops on $S$, and $\gamma$ is
an immersed based loop on $S$.
We may assume $\alpha\cup \beta \cup \gamma$ is an immersion
with at worst transverse double points. We compute
\begin{eqnarray}
\label{eq:3-2-1}
\sigma(\alpha)\sigma(\beta)\gamma
&=& \sigma(\alpha) \left(\sum_{p\in \beta \cap \gamma}\varepsilon(p;\beta,\gamma)
\gamma_{*p}\beta_p\gamma_{p*} \right) \nonumber \\
&=& \sum_{p\in \beta \cap \gamma}\varepsilon(p;\beta,\gamma)
\sigma(\alpha)\gamma_{*p}\beta_p\gamma_{p*} \nonumber \\
&=&
\sum_{p\in \beta \cap \gamma}\sum_{q\in \alpha \cap \beta}
\varepsilon(p;\beta,\gamma)\varepsilon(q;\alpha,\beta)
(\gamma_{*p}\beta_p\gamma_{p*})_{*q}\alpha_q
(\gamma_{*p}\beta_p\gamma_{p*})_{q*} \nonumber \\
& & +\sum_{p\in \beta \cap \gamma}\sum_{r\in \alpha \cap \gamma}
\varepsilon(p;\beta,\gamma)\varepsilon(r;\alpha,\gamma)
(\gamma_{*p}\beta_p\gamma_{p*})_{*r}\alpha_r
(\gamma_{*p}\beta_p\gamma_{p*})_{r*},
\end{eqnarray}
and
\begin{eqnarray}
\label{eq:3-2-2}
-\sigma(\beta)\sigma(\alpha)\gamma
&=& -\sigma(\beta) \left(\sum_{r\in \alpha \cap \gamma}\varepsilon(r;\alpha,\gamma)
\gamma_{*r}\alpha_r\gamma_{r*} \right) \nonumber \\
&=& -\sum_{r\in \alpha \cap \gamma}\varepsilon(r;\alpha,\gamma)
\sigma(\beta)\gamma_{*r}\alpha_r\gamma_{r*} \nonumber \\
&=& -\sum_{r\in \alpha \cap \gamma}\sum_{p\in \beta \cap \gamma}
\varepsilon(r;\alpha,\gamma)\varepsilon(p;\beta,\gamma)
(\gamma_{*r}\alpha_r\gamma_{r*})_{*p}\beta_p
(\gamma_{*r}\alpha_r\gamma_{r*})_{p*} \nonumber \\
& & -\sum_{r\in \alpha \cap \gamma}\sum_{q\in \beta \cap \alpha}
\varepsilon(r;\alpha,\gamma)\varepsilon(q;\beta,\alpha)
(\gamma_{*r}\alpha_r\gamma_{r*})_{*q}\beta_q
(\gamma_{*r}\alpha_r\gamma_{r*})_{q*}.
\end{eqnarray}
Then the second term of (\ref{eq:3-2-1}) and the first term of
(\ref{eq:3-2-2}) cancel and we have
\begin{eqnarray*}
\sigma(\alpha)\sigma(\beta)\gamma-\sigma(\beta)\sigma(\alpha)\gamma
&=& \sum_{p\in \beta \cap \gamma}\sum_{q\in \alpha \cap \beta}
\varepsilon(p;\beta,\gamma)\varepsilon(q;\alpha,\beta)
(\gamma_{*p}\beta_p\gamma_{p*})_{*q}\alpha_q
(\gamma_{*p}\beta_p\gamma_{p*})_{q*} \\
& & +\sum_{r\in \alpha \cap \gamma}\sum_{q\in \alpha \cap \beta}
\varepsilon(r;\alpha,\gamma)\varepsilon(q;\alpha,\beta)
(\gamma_{*r}\alpha_r\gamma_{r*})_{*q}\beta_q
(\gamma_{*r}\alpha_r\gamma_{r*})_{q*}.
\end{eqnarray*}
Here we use $\varepsilon(q;\beta,\alpha)=-\varepsilon(q;\alpha,\beta)$.
Now $(\gamma_{*p}\beta_p\gamma_{p*})_{*q}\alpha_q
(\gamma_{*p}\beta_p\gamma_{p*})_{q*}
=\gamma_{*p}|\alpha_q\beta_q|_p\gamma_{p*}$ for
$p\in \beta \cap \gamma$, $q\in \alpha \cap \beta$
and $(\gamma_{*r}\alpha_r\gamma_{r*})_{*q}\beta_q
(\gamma_{*r}\alpha_r\gamma_{r*})_{q*}
=\gamma_{*r}|\alpha_q\beta_q|_r\gamma_{r*}$ for
$r\in \alpha \cap \gamma$, $q\in \alpha \cap \beta$.
Therefore, we have
\begin{eqnarray*}
\sigma(\alpha)\sigma(\beta)\gamma-\sigma(\beta)\sigma(\alpha)\gamma
&=& \sum_{x\in (\alpha \cup \beta)\cap \gamma}
\sum_{y\in \alpha \cap \beta}
\varepsilon(x;\alpha\cup \beta,\gamma)
\varepsilon(y;\alpha,\beta) \gamma_{*x}|\alpha_y\beta_y|_x\gamma_{x*} \\
&=& \sigma \left( \sum_{y\in \alpha \cap \beta}
\varepsilon(y;\alpha,\beta)|\alpha_y\beta_y| \right) \gamma \\
&=& \sigma([\alpha,\beta])\gamma.
\end{eqnarray*}
This completes the proof.
\end{proof}

Note that to make Definition \ref{def:3-2-1} work, we need to
delete the basepoint $*$. A simple illustration
of this fact is Figure 4. Here the loops $\alpha_1$ and $\alpha_2$
are homotopic as free loops on $S$, but not homotopic as loops
on $S\setminus \{ *\}$. Following Definition \ref{def:3-2-1},
we have $\sigma(\alpha_1)\gamma=\alpha\gamma-\gamma\alpha$,
and clearly $\sigma(\alpha_2)\gamma=0$.
Hence $\sigma(\alpha_1)\gamma \neq \sigma(\alpha_2)\gamma$.

\begin{center}
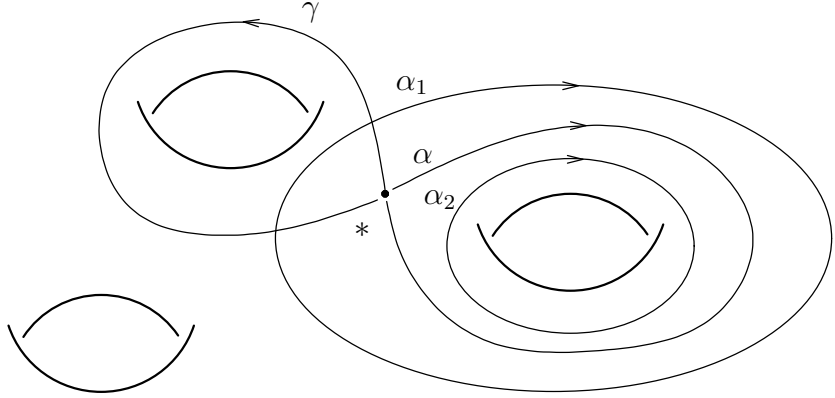
Figure 4: We need to delete the basepoint to define $\sigma$.
\unitlength 0.1in
\begin{picture}( 42.6200, 21.0000)(  0.8000,-21.7600)
%
\special{pn 13}%
\special{ar 1230 502 506 506  0.3217506 2.8198421}%
%
\special{pn 13}%
\special{ar 1230 982 480 480  3.7158975 5.6904879}%
%
\special{pn 20}%
\special{sh 1}%
\special{ar 2030 1142 10 10 0  6.28318530717959E+0000}%
\special{sh 1}%
\special{ar 2030 1142 10 10 0  6.28318530717959E+0000}%
%
\special{pn 13}%
\special{ar 2990 1142 506 506  0.3217506 2.8198421}%
%
\special{pn 13}%
\special{ar 2990 1622 480 480  3.7158975 5.6904879}%
%
\special{pn 8}%
\special{pa 2974 942}%
\special{pa 3054 962}%
\special{fp}%
\special{pa 3054 962}%
\special{pa 2974 982}%
\special{fp}%
%
\special{pn 8}%
\special{pa 2030 1126}%
\special{pa 2026 1092}%
\special{pa 2020 1056}%
\special{pa 2014 1022}%
\special{pa 2008 986}%
\special{pa 2002 952}%
\special{pa 1996 918}%
\special{pa 1990 884}%
\special{pa 1982 850}%
\special{pa 1974 818}%
\special{pa 1966 784}%
\special{pa 1958 752}%
\special{pa 1948 720}%
\special{pa 1938 690}%
\special{pa 1928 660}%
\special{pa 1916 630}%
\special{pa 1904 602}%
\special{pa 1890 574}%
\special{pa 1876 546}%
\special{pa 1860 520}%
\special{pa 1844 496}%
\special{pa 1826 470}%
\special{pa 1808 448}%
\special{pa 1788 426}%
\special{pa 1766 406}%
\special{pa 1742 386}%
\special{pa 1718 368}%
\special{pa 1694 352}%
\special{pa 1666 336}%
\special{pa 1638 322}%
\special{pa 1610 308}%
\special{pa 1580 298}%
\special{pa 1550 286}%
\special{pa 1518 278}%
\special{pa 1486 270}%
\special{pa 1454 262}%
\special{pa 1422 256}%
\special{pa 1388 252}%
\special{pa 1354 248}%
\special{pa 1320 246}%
\special{pa 1284 246}%
\special{pa 1250 246}%
\special{pa 1214 248}%
\special{pa 1180 250}%
\special{pa 1144 254}%
\special{pa 1110 260}%
\special{pa 1076 266}%
\special{pa 1042 274}%
\special{pa 1008 282}%
\special{pa 974 292}%
\special{pa 940 302}%
\special{pa 908 314}%
\special{pa 878 328}%
\special{pa 846 342}%
\special{pa 816 358}%
\special{pa 788 374}%
\special{pa 760 394}%
\special{pa 734 412}%
\special{pa 708 434}%
\special{pa 686 456}%
\special{pa 664 478}%
\special{pa 642 502}%
\special{pa 624 528}%
\special{pa 608 556}%
\special{pa 594 584}%
\special{pa 580 614}%
\special{pa 570 646}%
\special{pa 562 678}%
\special{pa 556 710}%
\special{pa 552 744}%
\special{pa 550 778}%
\special{pa 548 812}%
\special{pa 550 848}%
\special{pa 554 882}%
\special{pa 560 916}%
\special{pa 566 950}%
\special{pa 576 984}%
\special{pa 586 1018}%
\special{pa 598 1050}%
\special{pa 612 1080}%
\special{pa 628 1110}%
\special{pa 646 1138}%
\special{pa 666 1166}%
\special{pa 686 1190}%
\special{pa 710 1214}%
\special{pa 734 1234}%
\special{pa 758 1254}%
\special{pa 786 1270}%
\special{pa 814 1286}%
\special{pa 842 1300}%
\special{pa 872 1312}%
\special{pa 904 1322}%
\special{pa 934 1330}%
\special{pa 968 1338}%
\special{pa 1000 1344}%
\special{pa 1032 1348}%
\special{pa 1066 1352}%
\special{pa 1100 1356}%
\special{pa 1134 1358}%
\special{pa 1166 1358}%
\special{pa 1200 1358}%
\special{pa 1232 1358}%
\special{pa 1266 1358}%
\special{pa 1298 1356}%
\special{pa 1330 1352}%
\special{pa 1362 1350}%
\special{pa 1394 1346}%
\special{pa 1426 1342}%
\special{pa 1456 1336}%
\special{pa 1488 1330}%
\special{pa 1520 1324}%
\special{pa 1550 1318}%
\special{pa 1582 1310}%
\special{pa 1612 1302}%
\special{pa 1642 1294}%
\special{pa 1674 1286}%
\special{pa 1704 1276}%
\special{pa 1734 1266}%
\special{pa 1764 1256}%
\special{pa 1794 1246}%
\special{pa 1826 1236}%
\special{pa 1856 1226}%
\special{pa 1886 1214}%
\special{pa 1916 1202}%
\special{pa 1946 1192}%
\special{pa 1976 1180}%
\special{pa 1990 1174}%
\special{sp}%
%
\special{pn 8}%
\special{ar 2902 1374 1440 800  0.0000000 6.2831853}%
%
\special{pn 8}%
\special{pa 2950 558}%
\special{pa 3030 578}%
\special{fp}%
\special{pa 3030 578}%
\special{pa 2950 598}%
\special{fp}%
%
\special{pn 8}%
\special{pa 1374 270}%
\special{pa 1296 244}%
\special{fp}%
\special{pa 1296 244}%
\special{pa 1378 232}%
\special{fp}%
%
\special{pn 8}%
\special{pa 2070 1118}%
\special{pa 2098 1104}%
\special{pa 2128 1090}%
\special{pa 2156 1076}%
\special{pa 2184 1060}%
\special{pa 2212 1046}%
\special{pa 2242 1032}%
\special{pa 2270 1018}%
\special{pa 2298 1004}%
\special{pa 2328 992}%
\special{pa 2356 978}%
\special{pa 2386 964}%
\special{pa 2414 952}%
\special{pa 2444 940}%
\special{pa 2474 928}%
\special{pa 2504 916}%
\special{pa 2534 904}%
\special{pa 2564 892}%
\special{pa 2594 882}%
\special{pa 2624 872}%
\special{pa 2656 862}%
\special{pa 2686 852}%
\special{pa 2718 842}%
\special{pa 2750 834}%
\special{pa 2782 826}%
\special{pa 2814 820}%
\special{pa 2846 812}%
\special{pa 2880 806}%
\special{pa 2912 802}%
\special{pa 2944 796}%
\special{pa 2978 792}%
\special{pa 3012 790}%
\special{pa 3044 788}%
\special{pa 3078 786}%
\special{pa 3110 784}%
\special{pa 3144 784}%
\special{pa 3176 786}%
\special{pa 3210 788}%
\special{pa 3242 790}%
\special{pa 3276 794}%
\special{pa 3308 800}%
\special{pa 3340 806}%
\special{pa 3372 812}%
\special{pa 3404 820}%
\special{pa 3434 830}%
\special{pa 3466 840}%
\special{pa 3496 850}%
\special{pa 3526 864}%
\special{pa 3556 878}%
\special{pa 3586 892}%
\special{pa 3614 908}%
\special{pa 3642 926}%
\special{pa 3670 946}%
\special{pa 3698 966}%
\special{pa 3724 988}%
\special{pa 3750 1010}%
\special{pa 3774 1034}%
\special{pa 3796 1060}%
\special{pa 3818 1084}%
\special{pa 3838 1112}%
\special{pa 3858 1140}%
\special{pa 3874 1168}%
\special{pa 3890 1196}%
\special{pa 3902 1226}%
\special{pa 3914 1256}%
\special{pa 3922 1288}%
\special{pa 3930 1318}%
\special{pa 3934 1350}%
\special{pa 3934 1382}%
\special{pa 3934 1414}%
\special{pa 3930 1446}%
\special{pa 3924 1478}%
\special{pa 3914 1510}%
\special{pa 3904 1542}%
\special{pa 3892 1572}%
\special{pa 3878 1602}%
\special{pa 3860 1632}%
\special{pa 3842 1660}%
\special{pa 3824 1686}%
\special{pa 3802 1712}%
\special{pa 3780 1738}%
\special{pa 3758 1760}%
\special{pa 3732 1782}%
\special{pa 3708 1802}%
\special{pa 3682 1820}%
\special{pa 3654 1836}%
\special{pa 3626 1852}%
\special{pa 3598 1866}%
\special{pa 3568 1878}%
\special{pa 3540 1890}%
\special{pa 3508 1900}%
\special{pa 3478 1908}%
\special{pa 3446 1916}%
\special{pa 3414 1924}%
\special{pa 3382 1930}%
\special{pa 3350 1936}%
\special{pa 3316 1942}%
\special{pa 3284 1946}%
\special{pa 3250 1950}%
\special{pa 3216 1954}%
\special{pa 3182 1956}%
\special{pa 3150 1960}%
\special{pa 3116 1962}%
\special{pa 3082 1964}%
\special{pa 3048 1966}%
\special{pa 3014 1968}%
\special{pa 2982 1970}%
\special{pa 2948 1970}%
\special{pa 2916 1970}%
\special{pa 2882 1968}%
\special{pa 2850 1968}%
\special{pa 2818 1966}%
\special{pa 2784 1962}%
\special{pa 2754 1958}%
\special{pa 2722 1954}%
\special{pa 2690 1948}%
\special{pa 2660 1942}%
\special{pa 2628 1934}%
\special{pa 2598 1926}%
\special{pa 2568 1914}%
\special{pa 2540 1904}%
\special{pa 2512 1890}%
\special{pa 2482 1876}%
\special{pa 2456 1860}%
\special{pa 2428 1844}%
\special{pa 2402 1826}%
\special{pa 2376 1806}%
\special{pa 2350 1786}%
\special{pa 2326 1764}%
\special{pa 2302 1742}%
\special{pa 2280 1718}%
\special{pa 2258 1692}%
\special{pa 2238 1668}%
\special{pa 2216 1640}%
\special{pa 2198 1614}%
\special{pa 2180 1586}%
\special{pa 2162 1558}%
\special{pa 2146 1528}%
\special{pa 2132 1500}%
\special{pa 2118 1470}%
\special{pa 2106 1440}%
\special{pa 2094 1410}%
\special{pa 2084 1378}%
\special{pa 2076 1348}%
\special{pa 2068 1318}%
\special{pa 2060 1286}%
\special{pa 2054 1254}%
\special{pa 2046 1224}%
\special{pa 2040 1192}%
\special{pa 2038 1182}%
\special{sp}%
%
\special{pn 8}%
\special{pa 2990 766}%
\special{pa 3070 786}%
\special{fp}%
\special{pa 3070 786}%
\special{pa 2990 806}%
\special{fp}%
\put(18.7000,-13.6600){\makebox(0,0)[lb]{$*$}}%
\put(15.9800,-2.4600){\makebox(0,0)[lb]{$\gamma$}}%
\put(20.8600,-6.2200){\makebox(0,0)[lb]{$\alpha_1$}}%
\put(22.3000,-12.2000){\makebox(0,0)[lb]{$\alpha_2$}}%
\put(21.7400,-9.8200){\makebox(0,0)[lb]{$\alpha$}}%
%
\special{pn 8}%
\special{ar 2990 1414 640 456  0.0000000 6.2831853}%
%
\special{pn 13}%
\special{ar 560 1670 506 506  0.3217506 2.8198421}%
%
\special{pn 13}%
\special{ar 560 2150 480 480  3.7158975 5.6904879}%
\end{picture}%

\end{center}

But if $S$ and its basepoint are our $\Sigma$ and $*\in \partial \Sigma$,
then the inclusion $\Sigma\setminus \{ * \} \hookrightarrow \Sigma$
is a homotopy equivalence. Thus $\mathbb{Q}\hat{\pi}(\Sigma \setminus \{ *\})
=\mathbb{Q}\hat{\pi}(\Sigma)$.
Writing $\hat{\pi}(\Sigma)=\hat{\pi}$ for simplicity,
we have a Lie algebra homomorphism
\begin{equation}
\sigma \colon \mathbb{Q}\hat{\pi}
\rightarrow {\rm Der}(\mathbb{Q}\pi).
\end{equation}

\begin{rem}{\rm
Let $M$ be a $d$-dimensional oriented $C^\infty$ manifold,
and choose a basepoint $*\in M$.
We regard $S^1 = [0,1]/0\sim1$, and denote $\Omega M
= {\rm Map}((S^1, 0), (M, *))$, the based loop space of $M$. The evaluation map
${\rm ev}: \Omega M \to M$, $\gamma\mapsto \gamma(1/2)$,
is a Hurewicz fibration, whose fiber ${\rm ev}^{-1}(*)$ is naturally
identified with $\Omega M\times\Omega M$. The map
$\rho: \Omega M\times [0,1] \to \Omega M$, given by
$$
\rho(\gamma, s)(t) := \left\{
\begin{array}{ll}
\gamma(2st), & \mbox{if $t \leq 1/2$,}\\
\gamma(s+(1-s)(2t-1)), & \mbox{if $t \geq 1/2$,}
\end{array}
\right.
$$
induces a map of pairs $\rho: \Omega M\times([0,1], \{0,1\}) \to
(\Omega M, \Omega M\times \Omega M)$. We define a map
$\Delta': H_*(\Omega M) \to H_{*+1}(\Omega M, \Omega M\times
\Omega M)$ by the composite
$$
H_*(\Omega M)\overset{\times[I]}\to H_{*+1}(\Omega M\times
([0,1], \{0,1\}))\overset{\rho_*}\to
H_{*+1}(\Omega M, \Omega M\times\Omega M),
$$
where $[I] \in H_1([0,1], \{0,1\})$ is the fundamental class.
In a way similar to Chas-Sullivan \cite{CS}, we can define a loop product
$$
\bullet: H_i(L(M\setminus\{*\}))\otimes H_j(\Omega M, \Omega M\times
\Omega M) \to H_{i+j-d}(\Omega M).
$$
Here we denote $L(M\setminus\{*\}) = {\rm Map}(S^1, M\setminus\{*\})$,
the free loop space of $M\setminus \{ *\}$.
Let $x: K_x\to L(M\setminus\{*\})$ be an $i$-cell, and $y: K_y \to
\Omega M$
a $j$-cell. We denote by $K_{x\bullet y}$ a transversal preimage of the
diagonal
under the map $K_x\times K_y \to (M\setminus\{*\})\times M$,
$(k_x, k_y) \mapsto (x(k_x)(0), y(k_y)(1/2))$. The $(i+j-d)$-cell
$x\bullet y:
K_{x\bullet y}\to \Omega M$ is defined by
$$
(x\bullet y)(k_x, k_y) = \left\{
\begin{array}{ll}
y(k_y)(2t), & \mbox{if $t \leq 1/4$,}\\
x(k_x)(2(t - 1/4)), & \mbox{if $1/4\leq t \leq 3/4$,}\\
y(k_y)(2t-1), & \mbox{if $3/4\leq t $.}
\end{array}
\right.
$$
Taking the composite of the loop product with the map $\Delta:
H_*(L(M\setminus\{*\})) \to H_{*+1}(L(M\setminus\{*\}))$
introduced by Chas-Sullivan \cite{CS} and the map $\Delta'$,
we obtain a map
$$
H_i(L(M\setminus\{*\})) \otimes H_j(\Omega M) \to
H_{i+j+2-d}(\Omega M), \quad u\otimes v \mapsto
(\Delta u)\bullet(\Delta'v).
$$
This coincides with our action $\sigma$ in the case $d=2$ and $i=j=0$.
}
\end{rem}

\subsection{Conventions about (co)homology of groups}
In the next two subsections, we give a homological interpretation
of the Goldman bracket and the action $\sigma$ for $\Sigma$.
To state these we fix some conventions about (co)homology of groups.
We basically follow Brown \cite{Bro}.

Let $G$ be a group and $M$ a left $\mathbb{Q}G$-module.
For simplicity we use the term $G$-module for left $\mathbb{Q}G$-module
and sometimes write $\otimes_G$ instead of $\otimes_{\mathbb{Q}G}$.
We can always regard $M$ as a right $G$-module by setting
$mg=g^{-1}m$ ($m\in M, g\in G$). The homology group $H_*(G;M)$
is defined by $H_*(G;M):={\rm Tor}_*^{\mathbb{Q}G}(M,\mathbb{Q})$.
Namely taking a $\mathbb{Q}G$-projective resolution
$$\cdots \rightarrow F_2 \rightarrow F_1 \rightarrow F_0
\rightarrow \mathbb{Q} \rightarrow 0$$
of the trivial $G$-module $\mathbb{Q}$,
$$H_*(G;M)=H_*(M\otimes_G F_*).$$
Similarly the cohomology group $H^*(G;M)$ is defined by
$$H^*(G;M):={\rm Ext}_{\mathbb{Q}G}^*(\mathbb{Q},M)
=H^*({\rm Hom}_{\mathbb{Q}G}(F_*,M)).$$

For $n\ge 0$, let $F_n(G)$ be the free $G$-module with
$\mathbb{Q}G$-basis $\{ [g_1|g_2|\cdots | g_n]; g_i\in G\}$.
The boundary map $\partial_n \colon F_n(G)\rightarrow F_{n-1}(G)$,
$n\ge 1$, is given by
$$\partial_n([g_1|g_2|\cdots|g_n])=g_1[g_2|\cdots |g_n]
+\sum_{i=1}^{n-1}(-1)^i[g_1|\cdots|g_ig_{i+1}|\cdots g_n]
+(-1)^n[g_1|\cdots |g_{n-1}],$$
and $\partial_0\colon F_0(G)\cong \mathbb{Q}G \rightarrow \mathbb{Q}$
is given by the augmentation.

The normalized standard complex, denoted by $\bar{F}_n(G)$,
is the quotient of $F_n(G)$ by the
$\mathbb{Q}G$-submodule spanned by
$\{ [g_1|g_2|\cdots |g_n];\ g_i=1 {\rm \ for\ some\ }i \}$.
It gives a $\mathbb{Q}G$-projective resolution of $\mathbb{Q}$.
For a $G$-module $M$, set $C_n(G;M)=M\otimes_{G} \bar{F}_n(G)$.
Of course we have $H_*(C_*(G;M))=H_*(G;M)$.
The boundary maps of $C_*(G;M)$ in low degrees are given by:
$$\partial_1(m\otimes [g])=g^{-1}m-m\in M\cong M[\ ];$$
$$\partial_2(m\otimes [g_1|g_2])
=g_1^{-1}m \otimes [g_2]-m\otimes [g_1g_2]+m\otimes [g_1].$$
Here $\otimes$ means $\otimes_G$.

In this paper, we consider the (co)homology of groups of $\pi$ for
various $\pi$-modules $M$. One way to describe these is to use
the normalized complex $C_*(\pi;M)$, and another way is to use
the following particular resolution.
Since $\pi$ is a free group of rank $2g$,
the augmentation ideal $I\pi$ is a free $\mathbb{Q}\pi$-module
of rank $2g$. Thus
$$0\rightarrow I\pi \rightarrow \mathbb{Q}\pi
\stackrel{\varepsilon}{\rightarrow} \mathbb{Q} \rightarrow 0$$
gives a $\mathbb{Q}\pi$-projective resolution of $\mathbb{Q}$.
In this view point, the homology group of $\pi$ is computed as
$$H_*(\pi;M)=H_*(M\otimes_{\mathbb{Q}\pi} I\pi \rightarrow M,\ m\otimes v
\mapsto \iota(v)m).$$
The canonical isomorphism $\bar{F}_0(\pi)\cong \mathbb{Q}\pi$
and a $\mathbb{Q}\pi$-linear map
$\bar{F}_1(\pi) \rightarrow I\pi$,
$[x] \mapsto x-1$ for $x\in \pi$,
are compatible with the boundary maps of the two resolutions.
In particular, we have a chain-level description of the canonical isomorphism
\begin{equation}
\label{eq:3-3-1}
H_1(C_*(\pi;M))\stackrel{\cong}{\rightarrow}
H_1(M\otimes_{\mathbb{Q}\pi}I\pi \rightarrow M).
\end{equation}

Finally we mention the relative version of homology of groups.
See also \S4.5.
Let $G$ be a group and $K$ a subgroup of $G$, and $M$ a left $G$-module.
Then $C_*(K;M)$ is a subcomplex of $C_*(G;M)$.
We define the relative homology group
as the homology of the quotient complex:
$$H_*(G,K;M):=H_*(C_*(G;M)/C_*(K;M)).$$
Note that since $C_0(G;M)=C_0(K;M)$, any 1-chain of
the complex $C_*(G;M)/C_*(K;M)$ is a cycle.

\subsection{Homological interpretation of the Goldman Lie algebra}
Let $\mathbb{Q}\pi^c$ be the following $\pi$-module. As a vector space,
$\mathbb{Q}\pi^c=\mathbb{Q}\pi$, and the $\pi$-action is given by
the conjugation: $xu:=xux^{-1}$ for $x\in \pi, u\in \mathbb{Q}\pi^c$.

\begin{dfn}
\label{def:3-4-1}
Define a $\mathbb{Q}$-linear map
$\lambda\colon \mathbb{Q}\pi \rightarrow H_1(\pi;\mathbb{Q}\pi^c)$
by $\lambda(x):=x\otimes [x]$, $x\in \pi$. Here we understand
$H_1(\pi;\mathbb{Q}\pi^c)$ as the homology of $C_*(\pi;\mathbb{Q}\pi^c)$.
\end{dfn}

We need to verify that this is well-defined, i.e., $\lambda(x)$ is a cycle.

\begin{lem}
\label{lem:3-4-2}
For $x,y \in \pi$, we have
\begin{enumerate}
\item[{\rm (1)}] $x\otimes [x]\in C_1(\pi;\mathbb{Q}\pi^c)$ is a cycle,
\item[{\rm (2)}] $\lambda(yxy^{-1})=\lambda(x)\in H_1(\pi;\mathbb{Q}\pi^c)$.
\end{enumerate}
\end{lem}

\begin{proof}
The first part follows from
$\partial_1(x\otimes [x])=x^{-1}\cdot x-x=x^{-1}xx-x=x-x=0$.
For the second part, note that for any $x_1,x_2\in \pi$, the 1-chain
$[x_1x_2]$ is homologous to $x_1[x_2]+[x_1]$
since $\partial_2([x_1|x_2])=x_1[x_2]-[x_1x_2]+[x_1]$, and for any $x\in \pi$,
the 1-chain $[x^{-1}]$ is homologous to $-x^{-1}[x]$ since
$\partial_2([x^{-1}|x])=x^{-1}[x]-[1]+[x^{-1}]=x^{-1}[x]+[x^{-1}]$.
Therefore, we compute
\begin{eqnarray*}
\lambda(yxy^{-1})&=& yxy^{-1}\otimes [yxy^{-1}]
\equiv yxy^{-1}\otimes (y[xy^{-1}]+[y]) \\
&\equiv & yxy^{-1}\otimes (y(x[y^{-1}]+[x])+[y]) \\
&\equiv & yxy^{-1}\otimes (y[x]-yxy^{-1}[y]+[y]) \\
&= & x \otimes [x]=\lambda(x).
\end{eqnarray*}
Here $\equiv$ stands for ``homologous". This proves (2).
\end{proof}

By Lemma \ref{lem:3-4-2}, $\lambda$ descends to a map
$\lambda \colon \mathbb{Q}\hat{\pi}\rightarrow H_1(\pi;\mathbb{Q}\pi^c)$.
We introduce a $\pi$-module map
$\mathcal{B} \colon \mathbb{Q}\pi^c \otimes \mathbb{Q}\pi^c \rightarrow
\mathbb{Q}\hat{\pi}$ by $\mathcal{B}(u\otimes v)=|uv|$.
Here $|\ |\colon \mathbb{Q}\pi \rightarrow \mathbb{Q}\hat{\pi}$
is the linear map induced by the projection $|\ |\colon \pi \rightarrow \hat{\pi}$,
the action of $\pi$ on $\mathbb{Q}\pi^c \otimes \mathbb{Q}\pi^c$
is diagonal, and $\mathbb{Q}\hat{\pi}$ is regarded
as a trivial $\pi$-module.

Let $\mathcal{S}^c$ be the local system on $\Sigma$ corresponding to
the left $\pi$-module $\mathbb{Q}\pi^c$. Since $\Sigma$ is a $K(\pi,1)$-space,
there is a canonical isomorphism
$H_*(\pi;\mathbb{Q}\pi^c)\cong H_*(\Sigma;\mathcal{S}^c)$.

Using the intersection form of the surface,
we have the following bilinear form:
\begin{eqnarray*}
(\ \cdot \ ) \colon H_1(\pi;\mathbb{Q}\pi^c)\times H_1(\pi;\mathbb{Q}\pi^c)
&\cong & H_1(\Sigma;\mathcal{S}^c)\times H_1(\Sigma;\mathcal{S}^c) \\
&\rightarrow & H_0(\Sigma;\mathcal{S}^c\otimes \mathcal{S}^c)
\cong H_0(\pi;\mathbb{Q}\pi^c \otimes \mathbb{Q}\pi^c).
\end{eqnarray*}

\begin{prop}
\label{prop:3-4-3}
\begin{enumerate}
\item[{\rm (1)}] The sequence
$$0 \rightarrow \mathbb{Q}\hat{\pi}^{\prime}
\stackrel{\lambda}{\rightarrow}
H_1(\pi;\mathbb{Q}\pi^c) \rightarrow H \rightarrow 0,$$
where the map $H_1(\pi;\mathbb{Q}\pi^c)\rightarrow
H_1(\pi;\mathbb{Q})=H$ is induced by the augmentation,
is exact and canonically splits.
\item[{\rm (2)}] For $u,v\in \mathbb{Q}\hat{\pi}$,
$$[u,v]=\mathcal{B}_*(\lambda(u)\cdot \lambda(v)).$$
Here $\mathcal{B}_* \colon
H_0(\pi;\mathbb{Q}\pi^c \otimes \mathbb{Q}\pi^c)
\rightarrow H_0(\pi;\mathbb{Q}\hat{\pi})=\mathbb{Q}\hat{\pi}$
is the map induced by $\mathcal{B}$.
\end{enumerate}
\end{prop}

\begin{proof}
For $x\in \pi$, let $Z(x)\subset \pi$ be the centralizer of $x$.
Then we have a direct decomposition
$$\mathbb{Q}\pi^c=\bigoplus_{|x|\in \hat{\pi}} \mathbb{Q}(\pi/Z(x))$$
as left $\pi$-modules.
Here the action of $\pi$ on $\mathbb{Q}(\pi/Z(x))$ is given by
the multiplication from the left: $y(z {\rm \ mod\ }Z(x))=
yz {\rm \ mod\ }Z(x)$.
Since the Euler characteristic of $\Sigma$ is negative,
$Z(x)$ is a cyclic group of infinite order for $x\neq 1$.
In fact, the centralizer of a hyperbolic element in $PSL_2(\mathbb{R})$
is an abelian subgroup isomorphic to $\mathbb{R}_{>0}$.
Using the Shapiro lemma
(see \cite{Bro}, p. 73), we have a canonical decomposition
\begin{eqnarray*}
H_1(\pi;\mathbb{Q}\pi^c)=
\bigoplus_{|x|\in \hat{\pi}}H_1(\pi;\mathbb{Q}(\pi/Z(x)))
&\cong & \bigoplus_{|x|\in \hat{\pi}}H_1(Z(x);\mathbb{Q}) \\
&=& H_1(\pi;\mathbb{Q})\oplus
\bigoplus_{|x|\in \hat{\pi}^{\prime}}H_1(Z(x);\mathbb{Q}).
\end{eqnarray*}
Moreover if $x\neq 1$, the cycle $\lambda(x)=x\otimes [x]$
corresponds to a generator of $H_1(Z(x);\mathbb{Q})\cong \mathbb{Q}$
(note: this part does not hold over the integers;
in this case $\lambda(x)$ corresponds to a non-zero multiple
of a generator of $H_1(Z(x);\mathbb{Z})\cong \mathbb{Z}$).
This proves the first part.

Next we proceed to the second part. As in Goldman \cite{Go} \S2,
we regard local systems as flat vector bundles and
their (co)homology as the (co)homology of singular chains
with values in flat vector bundles.
Following this description the fiber of $\mathcal{S}^c$ at $p\in \Sigma$
is the group ring $\mathbb{Q}\pi_1(\Sigma,p)$, and the parallel
transport along a path $\ell\colon[0,1]\rightarrow \Sigma$ is given by
$\mathbb{Q}\pi_1(\Sigma,\ell(0))\stackrel{\cong}{\rightarrow}
\mathbb{Q}\pi_1(\Sigma,\ell(1))$,
$\alpha \mapsto \ell^{-1} \alpha \ell$.

Let $\alpha$ be a based loop on $\Sigma$.
Under the canonical isomorphism
$H_1(\pi;\mathbb{Q}\pi^c)\cong H_1(\Sigma;\mathcal{S}^c)$,
the 1-cycle
$\alpha \otimes [\alpha]\in C_1(\pi;\mathbb{Q}\pi^c)$
corresponds the flat section $s_{\lambda}(\alpha)$
of $\alpha^*\mathcal{S}^c$ over $\alpha$,
whose value at $p\in \alpha$ is just $\alpha_p\in \pi_1(\Sigma,p)$
(to be more precise, we need to write $p=p(t)$ for some $t\in S^1$).
The homology class of the section $s_\lambda(\alpha)$ 
in $H_1(\Sigma; \mathcal{S}^c)$ depends only on the 
free homotopy class of the loop $\alpha$,
because of the homotopy equivalence of twisted homology. 
Let $\beta$ be another free loop on $\Sigma$ and 
suppose $\alpha$ and $\beta$ intersect with at worst transverse 
double points. Similarly $\beta\otimes[\beta]$ is regarded as 
a section $s_\lambda(\beta)$.

Using the same letter let $\mathcal{B}\colon \mathcal{S}^c\otimes
\mathcal{S}^c \rightarrow \mathbb{Q}\hat{\pi}$ be the
pairing of local systems on $\Sigma$ corresponding to
the $\pi$-module map $\mathcal{B}$. Here $\mathbb{Q}\hat{\pi}$
is considered as a trivial local system.
For each $p\in \Sigma$, this pairing is just the conjunction
$\mathbb{Q}\pi_1(\Sigma,p)\otimes \mathbb{Q}\pi_1(\Sigma,p)
\rightarrow \mathbb{Q}\hat{\pi},
u \otimes v \mapsto |u v|$.

By the formula in \cite{Go}, p.276, we have
$$\mathcal{B}_*(\lambda(\alpha)\cdot \lambda(\beta))
=\sum_{p\in \alpha \cap \beta}\varepsilon(p;\alpha,\beta)
\mathcal{B}(s_{\lambda}(\alpha)_p\otimes s_{\lambda}(\beta)_p).$$
But since $\mathcal{B}(s_{\lambda}(\alpha)_p\otimes s_{\lambda}(\beta)_p)
=|\alpha_p\beta_p|$,
this is nothing but the Goldman bracket $[\alpha,\beta]$.
This completes the proof.
\end{proof}

\begin{rem}
\label{rem:3-4-4}
{\rm
It should be remarked that $\lambda$ is related
to Chas-Sullivan's operator $\Delta=ME$ \cite{CS}.
More precisely,
let $L\Sigma$ be the free loop space of the surface $\Sigma$,
$L\Sigma={\rm Map}(S^1, \Sigma)$.
The evaluation map at $0\in S^1=[0,1]/0\sim 1$,
${\rm ev}\colon L\Sigma \to \Sigma$, $\ell \mapsto \ell(0)$,
is a Hurewicz fibration with fiber
$\Omega\Sigma$, the based loop space of $\Sigma$.
Since $\Sigma$ is a $K(\pi,1)$-space, the homology group
$H_*(\Omega\Sigma; \mathbb{Q})$ vanishes in positive degree.
The $0$-th homology group $H_0(\Omega\Sigma; \mathbb{Q})$
constitutes the local system $\mathcal{S}^c$ stated above.
Hence we have an isomorphism ${\rm ev}_*\colon H_*(L\Sigma; \mathbb{Q})
\cong H_*(\Sigma; \mathcal{S}^c)$. The diagram
$$
\begin{CD}
H_0(L\Sigma; \mathbb{Q}) @>{\Delta}>> H_1(L\Sigma; \mathbb{Q})\\
@| @V{{\rm ev}_*}VV\\
\mathbb{Q}\hat{\pi} @>{\lambda}>> H_1(\Sigma; \mathcal{S}^c)
\end{CD}
$$
commutes by the definition of $\Delta = ME$ and $\lambda$.
}
\end{rem}

\subsection{Homological interpretation of the action}
Let $\mathbb{Q}\pi^r$ (resp. $\mathbb{Q}\pi^l$) be the following
$\pi$-module. As a vector space,
$\mathbb{Q}\pi^r=\mathbb{Q}\pi^l=\mathbb{Q}\pi$, and the $\pi$-action
is given by the multiplication from the right (resp. the left):
$xu:=ux^{-1}$ (resp. $xu:=xu$) for $x\in \pi$, and $u\in \mathbb{Q}\pi^r$
(resp. $\mathbb{Q}\pi^l$).

Let $\langle \zeta \rangle$ be the cyclic subgroup of $\pi$ generated by $\zeta$.
We consider the relative homology of the pair $(\pi,\langle \zeta \rangle)$.

\begin{dfn}
\label{def:3-5-1}
Define a $\mathbb{Q}$-linear map
$\xi\colon \mathbb{Q}\pi \rightarrow
H_1(\pi,\langle \zeta \rangle; \mathbb{Q}\pi^r \otimes \mathbb{Q}\pi^l)$
by $\xi(x)=(1\otimes x)\otimes [x]$, $x\in \pi$.
Here we understand
$H_1(\pi,\langle \zeta \rangle;
\mathbb{Q}\pi^r \otimes \mathbb{Q}\pi^l)$
as the homology of the relative complex {\rm (}see {\rm \S3.3)}.
\end{dfn}

We denote by $\mathbb{Q}\pi^t$ the vector space $\mathbb{Q}\pi$
with the trivial $\pi$-action.
We introduce a $\pi$-module map
$\mathcal{C}\colon \mathbb{Q}\pi^c \otimes
\mathbb{Q}\pi^r \otimes \mathbb{Q}\pi^l \rightarrow
\mathbb{Q}\pi^t$ by $\mathcal{C}(u\otimes v \otimes w)=vuw$.
Here we consider the diagonal $\pi$-action on
$\mathbb{Q}\pi^c \otimes
\mathbb{Q}\pi^r \otimes \mathbb{Q}\pi^l$.

Let $\mathcal{S}^r$, $\mathcal{S}^l$, and $\mathcal{S}^t$
be the local system on $\Sigma$ corresponding to the
$\pi$-modules $\mathbb{Q}\pi^r$, $\mathbb{Q}\pi^l$,
and $\mathbb{Q}\pi^t$, respectively.
Then we have the canonical isomorphism
$H_1(\pi,\langle \zeta \rangle;\mathbb{Q}\pi^r \otimes \mathbb{Q}\pi^l)\cong
H_1(\Sigma,\partial \Sigma; \mathcal{S}^r \otimes \mathcal{S}^l)$, etc.

Using the intersection form of the surface, we have the
following bilinear form:
\begin{eqnarray*}
(\ \cdot \ ) \colon H_1(\pi;\mathbb{Q}\pi^c) \times
H_1(\pi,\langle \zeta \rangle;\mathbb{Q}\pi^r \otimes \mathbb{Q}\pi^l)
&\cong &
H_1(\Sigma;\mathcal{S}^c) \times
H_1(\Sigma,\partial \Sigma; \mathcal{S}^r \otimes \mathcal{S}^l) \\
&\rightarrow &
H_0(\Sigma; \mathcal{S}^c \otimes \mathcal{S}^r \otimes \mathcal{S}^l) \\
&\cong &
H_0(\pi; \mathbb{Q}\pi^c \otimes \mathbb{Q}\pi^r \otimes \mathbb{Q}\pi^l).
\end{eqnarray*}

\begin{prop}
\label{prop:3-5-2}
For $u\in \mathbb{Q}\hat{\pi}$ and $v\in \mathbb{Q}\pi$, we have
$$\sigma(u)v=\mathcal{C}_*(\lambda(u)\cdot \xi(v)).$$
Here $\mathcal{C}_* \colon
H_0(\pi; \mathbb{Q}\pi^c \otimes \mathbb{Q}\pi^r \otimes \mathbb{Q}\pi^l)
\rightarrow H_0(\pi; \mathbb{Q}\pi^t)=\mathbb{Q}\pi$ is the map
induced by $\mathcal{C}$.
\end{prop}

\begin{proof}
The proof is similar to the proof of Proposition \ref{prop:3-4-3}.
Let $\alpha$ be an immersed loop and $\beta$ an immersed based loop.
Suppose they intersect with at worst transverse double points.
The fiber of the local system $\mathcal{S}^r$ (resp. $\mathcal{S}^l$)
at $p\in \Sigma$
is $\mathbb{Q}\pi(\Sigma,*,p)$ (resp. $\mathbb{Q}\pi(\Sigma,p,*)$).
Here $\pi(\Sigma,p,q)$ is the set of homotopy classes
of paths from $p$ to $q$, and $\mathbb{Q}\pi(\Sigma,p,q)$
is the $\mathbb{Q}$-vector space spanned by $\pi(\Sigma,p,q)$.

By the canonical isomorphism
$H_1(\pi,\langle \zeta \rangle;
\mathbb{Q}\pi^r \otimes \mathbb{Q}\pi^l)
\cong H_1(\Sigma, \partial \Sigma; \mathcal{S}^r \otimes \mathcal{S}^l)$,
the relative cycle $\xi(\beta)=(1\otimes \beta)\otimes [\beta]$
corresponds to the flat section $s_{\xi}(\beta)$ of
$\beta^*(\mathcal{S}^r \otimes \mathcal{S}^l)$ whose value
at $p\in \beta$ is just $(\beta_{*p}\otimes \beta_{p*})$.
Let $\mathcal{C}\colon \mathcal{S}^c
\otimes \mathcal{S}^r \otimes \mathcal{S}^l
\rightarrow \mathcal{S}^t$
be the pairing of local systems on $\Sigma$ corresponding
to the $\pi$-module map $\mathcal{C}$ (using the same letter).
For each $p\in \Sigma$, this pairing is just the conjunction
$\mathbb{Q}\pi_1(\Sigma,p)\otimes
\mathbb{Q}\pi(\Sigma,*,p) \otimes \mathbb{Q}\pi(\Sigma,p,*)
\rightarrow \mathbb{Q}\pi,
u\otimes v\otimes w\mapsto vuw$.

By the formula in \cite{Go}, p.276, we have
$$\mathcal{C}_*(\lambda(\alpha),\xi(\beta))
=\sum_{p\in \alpha \cap \beta} \varepsilon(p;\alpha,\beta)
\mathcal{C}(s_{\lambda}(\alpha)_p \otimes s_{\xi}(\beta)_p).$$
But since $\mathcal{C}(s_{\lambda}(\alpha)_p \otimes s_{\xi}(\beta)_p)
=\beta_{*p}\alpha_p\beta_{p*}$,  this equals
$\sigma(\alpha)\beta$. This completes the proof.
\end{proof}

\section{(Co)homology theory for Hopf algebras}

In this section, we discuss a general theory of 
relative homology and cap products for (complete) 
Hopf algebras. Theory of cap products on the absolute 
(co)homology of a single (complete) Hopf algebra 
was already discussed in Cartan-Eilenberg \cite{CE} Chapter XI. 
But, unfortunately, the authors do not find 
an appropriate reference for cap products 
on the relative (co)homology of a pair of 
(complete) Hopf algebras.
In the succeeding sections 
these notions relate the Goldman
Lie  algebra to symplectic 
derivations of the algebra $\widehat{T}$.
\par

\subsection{The mapping cone of a chain map}

We begin by recalling the notion of the mapping cone of a chain map.
See, for example, \cite{Bro} pp.6-7. 
Let $f\colon (C_*, d') \to (D_*, d)$ be a chain map of chain 
complexes. The mapping cone $D_*\rtimes_fC_{*-1}$ 
of the map $f$ is defined by 
$$
(D_*\rtimes_fC_{*-1})_n := D_n\oplus C_{n-1},
\quad\mbox{and}\quad d'' := \begin{pmatrix} d & f\\ 0 & -d'\end{pmatrix}. 
$$
Then we have a natural long exact sequence

\begin{equation}
\label{eq:4-1-1}
\cdots \to H_n(C_*) \overset{f_*}\to H_n(D_*) \to
H_n(D_*\rtimes_fC_{*-1}) \to H_{n-1}(C_*) \to \cdots \quad\mbox{(exact)}.
\end{equation}

\begin{lem}
\label{lem:4-1-1}
If the chain map $f$ is injective, then the natural projection
$$
\varpi\colon D_*\rtimes_fC_{*-1} \to D_*/f_*C_*, \quad
(u,v) \mapsto u\bmod f_*C_*
$$
is a quasi-isomorphism. 
\end{lem}
\begin{proof}
Since $f$ is injective, we have $\Ker\varpi\cong
C_*\rtimes_{1_{C_*}}C_{*-1}$, so that $H_*(\Ker\varpi) = 0$
from (\ref{eq:4-1-1}). We have the short exact 
sequence $0\to \Ker\varpi\to D_*\rtimes_fC_{*-1} 
\overset{\varpi}\to D_*/f_*C_* \to 0$, since $\varpi$ is surjective. 
Hence $\varpi$ is a quasi-isomorphism.
\end{proof}
The following lemma will play a fundamental role in this section.
\begin{lem}
\label{lem:4-1-2}
\begin{enumerate}
\item[{\rm (1)}] If a chain homotopy $\Phi\colon C_* \to D_{*+1}$ connects 
$f$ to another chain map $g\colon C_*\to D_*$, namely, 
$d\Phi+\Phi d^{\prime}= g-f$, then the map
$$
h(\Phi):= \begin{pmatrix} 1& -\Phi\\ 0 & 1\end{pmatrix}\colon 
D_*\rtimes_fC_{*-1}\to D_*\rtimes_gC_{*-1}
$$
is a chain map and a quasi-isomorphism.
\item[{\rm (2)}]  Assume another chain homotopy $\Phi'\colon C_* \to D_{*+1}$ 
connecting $f$ to $g$ is homotopic to $\Phi$, 
in other words, there exists a map $\Psi\colon C_* \to D_{*+2}$ 
satisfying the relation
$$
\Phi'_n -\Phi_n = (-1)^n(d\Psi_n + \Psi_{n-1}d')\colon 
C_n \to D_{n+1}
$$
for each degree $n$. Then we have 
$$
h(\Phi) \simeq h(\Phi')\colon D_*\rtimes_fC_{*-1}\to D_*\rtimes_gC_{*-1}.
$$
\end{enumerate}
\end{lem}
\begin{proof}
By a straightforward computation, $h(\Phi)$ is a chain map. 
It defines a homomorphism between the long exact sequences 
(\ref{eq:4-1-1}). Hence it is a quasi-isomorphism from the 
five-lemma. We have 
$$
\begin{pmatrix} d & g\\ 0 & -d'\end{pmatrix}
\begin{pmatrix} 0 & (-1)^{n-2}\Psi_{n-1}\\ 0 & 0\end{pmatrix}
+
\begin{pmatrix} 0 & (-1)^{n-3}\Psi_{n-2}\\ 0 & 0\end{pmatrix}
\begin{pmatrix} d & f\\ 0 & -d'\end{pmatrix}
= h(\Phi') - h(\Phi).
$$
This implies the second part of the lemma.
\end{proof}

The followings are well-known.
\begin{lem}
\label{lem:4-1-3}
Let $R$ be an associative algebra, 
$C_*$ a left $R$-projective chain complex, and 
$D_*$ a left $R$-acyclic chain complex. Then
\begin{enumerate}
\item[{\rm (1)}] For any $R$-map $f\colon H_0(C_*) \to H_0(D_*)$, there exists 
an $R$-chain map $\varphi\colon C_* \to D_*$ inducing the map $f$ on $H_0$.
\item[{\rm (2)}] If two $R$-chain maps $\varphi$ and $\psi\colon C_* \to D_*$ satisfy 
$\varphi_* = \psi_*\colon H_0(C_*) \to H_0(D_*)$, then we have a 
$R$-chain homotopy $\varphi\simeq \psi\colon C_* \to D_*$. 
\item[{\rm (3)}] Moreover, if $\Phi$ and $\Phi'$ are $R$-chain homotopies 
connecting $\varphi$ to $\psi$, then $\Phi$ and $\Phi'$ are 
chain homotopic to each other. In other words, there exists an $R$-map 
$\Psi\colon C_* \to D_{*+2}$ satisfying the relation
$$
\Phi'_n -\Phi_n = (-1)^n(d\Psi_n + \Psi_{n-1}d')\colon 
C_n \to D_{n+1}
$$
for each $n \geq 0$. 
\end{enumerate}
\end{lem}

Let $R'$, $S'$, $R$ and $S$ be associative algebras, and
$C'_*$, $D'_*$, $C_*$ and $D_*$ chain complexes of 
left $R'$, $S'$, $R$ and $S$ modules, respectively. 
Suppose 
\begin{equation}
\label{eq:4-1-2}
\begin{CD}
R' @>{\varphi}>> R\\
@V{f'}VV @V{f}VV\\
S' @>{\psi}>> S
\end{CD}
\quad\mbox{and}\quad
 \begin{CD}
C'_* @>{\varphi}>> C_*\\
@V{f'}VV @V{f}VV\\
D'_* @>{\psi}>> D_*
\end{CD}
\end{equation}
are a commutative diagram of algebra homomorphisms and 
a {\it homotopy} commutative diagram of chain maps, 
respectively, such that the chain maps $f'$, $f$, $\varphi$ and 
$\psi$ respect the algebra homomorphisms $f'$, $f$, $\varphi$
and $\psi$, respectively, and the augmentations. 
Then we have a left $R'$-chain homotopy 
$\Theta\colon C'_* \to D_{*+1}$ connecting $\psi f^{\prime}$ to $f\varphi$. 
Let $M$ be a right $S$-module. Then we define a chain map 
$h(\varphi, \Theta, \psi) := \begin{pmatrix} \psi&-\Theta\\0&\varphi\end{pmatrix}\colon 
(M\otimes_{S'}D'_*)\rtimes_{f'}(M\otimes_{R'}C'_{*-1})
\to (M\otimes_{S}D_*)\rtimes_{f}(M\otimes_{R}C_{*-1})$
by the composite
\begin{eqnarray}
(M\otimes_{S'}D'_*)\rtimes_{f'}(M\otimes_{R'}C'_{*-1}) 
&\overset{\begin{pmatrix} \psi&0\\ 0&1\end{pmatrix}}\longrightarrow&
(M\otimes_{S}D_*)\rtimes_{\psi f'}(M\otimes_{R'}C'_{*-1}) \nonumber\\
&\overset{h(\Theta)}\longrightarrow&
(M\otimes_{S}D_*)\rtimes_{f\varphi}(M\otimes_{R'}C'_{*-1}) \nonumber\\
&\overset{\begin{pmatrix} 1&0\\ 0&\varphi\end{pmatrix}}\longrightarrow &
(M\otimes_{S}D_*)\rtimes_{f}(M\otimes_{R}C_{*-1}).
\label{eq:4-1-3}
\end{eqnarray}
Here we regard $M$ as a module on which $R'$, $R$ and $S'$ 
act through the homomorphisms $f\circ\varphi = \psi\circ f'$, 
$f$, and $\psi$, respectively.
\begin{lem}
\label{lem:4-1-4}
Assume $C'_*$ is left $R'$-projective 
and $D_*$ acyclic. Then the map 
$$
(\varphi, \psi)_* := h(\varphi, \Theta, \psi)_*\colon 
H_*((M\otimes_{S'}D'_*)\rtimes_{f'}(M\otimes_{R'}C'_{*-1}))\to 
H_*((M\otimes_{S}D_*)\rtimes_{f}(M\otimes_{R}C_{*-1}))
$$
induced by the chain map $h(\varphi, \Theta, \psi)$ depends only 
on the homotopy classes of the chain maps $\varphi$ and $\psi$. 
\end{lem}
\begin{proof}
Suppose $\varphi'\colon C'_* \to C_*$ and $\psi'\colon D'_* \to D_*$ are 
chain maps homotopic to $\varphi$ and $\psi$, respectively, 
 and $\Theta'$ a chain homotopy
connecting $\psi'f'$ to $f\varphi'$. Take a chain homotopy 
$\Phi\colon C'_* \to C_{*+1}$ connecting $\varphi$ to $\varphi'$, and 
 $\Psi\colon D'_* \to D_{*+1}$ connecting $\psi$ to 
$\psi'$. Then the three diagrams 
$$
\begin{CD}
(M\otimes_SD_*)\rtimes_{f}(M\otimes_RC_{*-1}) @>{\psi}>>
(M\otimes_SD'_*)\rtimes_{\psi\circ f}(M\otimes_RC_{*-1})\\
@| @V{h(\Psi\circ f)}VV\\
(M\otimes_SD_*)\rtimes_{f}(M\otimes_RC_{*-1}) @>{\psi'}>>
(M\otimes_SD'_*)\rtimes_{\psi'\circ f}(M\otimes_RC_{*-1})\\
\end{CD}
$$
$$
\begin{CD}
(M\otimes_SD'_*)\rtimes_{\psi\circ f}(M\otimes_RC_{*-1}) @>{h(\Theta)}>>
(M\otimes_SD'_*)\rtimes_{f\circ\varphi}(M\otimes_RC_{*-1})\\
@V{h(\Psi\circ f)}VV @V{h(f\circ\Phi)}VV\\
(M\otimes_SD'_*)\rtimes_{\psi'\circ f}(M\otimes_RC_{*-1})
@>{h(\Theta')}>>
(M\otimes_SD'_*)\rtimes_{f\circ\varphi'}(M\otimes_RC_{*-1})\\
\end{CD}
$$
$$
\begin{CD}
(M\otimes_SD'_*)\rtimes_{f\circ\varphi}(M\otimes_RC_{*-1}) @>{\varphi}>>
(M\otimes_SD'_*)\rtimes_{f}(M\otimes_RC'_{*-1})\\
@V{h(f\circ\Phi)}VV @|\\
(M\otimes_SD'_*)\rtimes_{f\circ\varphi'}(M\otimes_RC_{*-1})
@>{\varphi'}>>
(M\otimes_SD'_*)\rtimes_{f}(M\otimes_RC'_{*-1})\\
\end{CD}
$$
commute up to homotopy. Here the horizontal $\psi$, $\psi'$, $\varphi$
and $\varphi'$ mean the chain maps $\begin{pmatrix}\psi&0\\0&1\end{pmatrix}$, 
$\begin{pmatrix}\psi'&0\\0&1\end{pmatrix}$, $\begin{pmatrix}1&0\\0&\varphi\end{pmatrix}$
and $\begin{pmatrix}1&0\\0&\varphi'\end{pmatrix}$, respectively. 
In fact, the chain homotopies $\begin{pmatrix}\Psi&0\\0&1\end{pmatrix}$ and 
$\begin{pmatrix}1&0\\0&-\Phi\end{pmatrix}$ induce the homotopy commutativity 
of the first and the third diagrams, respectively. The maps $\Theta + 
f\circ\Phi$ and $\Theta'+\Psi\circ f$ are both chain homotopies 
connecting $\psi\circ f$ to $f\circ\varphi'$. Since 
$C'_*$ is $R'$-projective and $D_*$ acyclic, 
Lemma \ref{lem:4-1-3} (3) implies there exists a homotopy connecting 
$\Theta + f\circ\Phi$ to $\Theta'+\Psi\circ f$. Hence, by Lemma 
\ref{lem:4-1-2} (2), we have $h(f\circ\Phi)h(\Theta) = 
h(\Theta + f\circ\Phi) \simeq h(\Theta'+\Psi\circ f) = 
h(\Theta')h(\Psi\circ f)$. This means the homotopy commutativity of 
the second diagram. Hence we obtain $h(\varphi,\Theta,\psi) \simeq
h(\varphi',\Theta',\psi')\colon
(M\otimes_{S'}D'_*)\rtimes_{f'}(M\otimes_{R'}C'_{*-1})\to 
(M\otimes_{S}D_*)\rtimes_{f}(M\otimes_{R}C_{*-1})$.
This proves the lemma.
\end{proof}
Moreover, suppose 
\begin{equation}
\begin{CD}
R'' @>{\alpha}>> R'\\
@V{f''}VV @V{f'}VV\\
S'' @>{\beta}>> S'
\end{CD}
\quad\mbox{and}\quad
 \begin{CD}
C''_* @>{\alpha}>> C'_*\\
@V{f'}VV @V{f}VV\\
D''_* @>{\beta}>> D'_*
\end{CD}\nonumber
\end{equation}
are a commutative diagram and a homotopy commutative 
diagram as in (\ref{eq:4-1-2}). 
Let $\Xi$ be a chain homotopy connecting $\beta f''$ to 
$f'\alpha$, and $\Upsilon$ connecting $\psi\beta f''$ 
to $f\varphi \alpha$. 
\begin{lem}
\label{lem:4-1-5}
Assume $C''_*$ is left $R''$-projective and $D_*$ acyclic. 
Then we have 
\begin{eqnarray}
&&h(\varphi,\Theta,\psi)_*h(\alpha,\Xi,\beta)_*
= h(\varphi\alpha,\Upsilon,\psi\beta)_*\colon\nonumber\\
&& H_*((M\otimes_{S''}D''_*)\rtimes_{f''}(M\otimes_{R''}C''_{*-1}))\to 
H_*((M\otimes_{S}D_*)\rtimes_{f}(M\otimes_{R}C_{*-1})).\nonumber
\end{eqnarray}
\end{lem}
\begin{proof}
By a straightforward computation, we have 
$h(\varphi,\Theta,\psi)h(\alpha,\Xi,\beta)
= h(\varphi\alpha, \Theta\alpha+\psi\Xi, \psi\beta)$ 
as chain maps. 
The chain homotopy $\Theta\alpha+\psi\Xi$ connects 
$\psi\beta f''$ to $f\varphi \alpha$. 
Hence, from Lemma \ref{lem:4-1-4}, we have 
$h(\varphi\alpha, \Upsilon, \psi\beta)_* = 
h(\varphi\alpha, \Theta\alpha+\psi\Xi, \psi\beta)_*
= h(\varphi,\Theta,\psi)_*h(\alpha,\Xi,\beta)_*$.
This proves the lemma.
\end{proof}

\subsection{Relative homology of a pair of Hopf algebras}

Let $S$ be a (complete) Hopf algebra over $\mathbb{Q}$
with the augmentation $\varepsilon$, the antipode $\iota$
and the coproduct $\Delta$. For a left $S$-module $M$, 
we can always regard it as a right $S$-module by 
$ms = \iota(s)m$, $s \in S$, $m\in M$. Define the 
(co)homology groups by $H_*(S; M) := {\rm Tor}^S_*(M, \mathbb{Q})$
and $H^*(S; M) := {\rm Ext}^*_S(\mathbb{Q}, M)$. 
Here $\mathbb{Q}$ means the trivial $S$-module via the 
augmentation map $\varepsilon$. 
Let $P_*\overset\varepsilon\to \mathbb{Q}$ be a left
$S$-projective resolution of the $S$-module $\mathbb{Q}$. 
Then they are given by $H_*(S; M) = H_*(M\otimes_SP_*)$ 
and $H^*(S; M) = H^*(\Hom_S(P_*, M))$. \par
Consider a homomorphism $f\colon R \to S$ of (complete) Hopf algebras. 
We regard $M$ as a left $R$-module through the homomorphism $f$. 
Let $F_*\overset\varepsilon\to \mathbb{Q}$ be a left
$R$-projective resolution of $\mathbb{Q}$. 
By Lemma \ref{lem:4-1-3} (1), we can choose 
an $R$-chain map $f\colon F_* \to P_*$ which respects
the homomorphism $f\colon R \to S$ and the augmentations. 
The (co)chain maps
$f := 1_M\otimes f\colon M\otimes_RF_*\to M\otimes_SP_*$ and 
$f:=\Hom(f, 1_M)\colon \Hom_S(P_*, M) \to \Hom_R(F_*, M)$ define the 
induced maps $f_*\colon H_*(R; M) \to H_*(S; M)$ and 
$f^*\colon H^*(S; M) \to H^*(R; M)$, which are independent of 
the choice of the chain map $f\colon F_*\to P_*$. \par
We define the relative homology group $H_*(S,R;M)$ by the 
homology group of the mapping cone
$$
H_*(S,R;M) := H_*((M\otimes_SP_*)\rtimes_f(M\otimes_RF_{*-1})),
$$
which we call {\it the relative homology of the pair $(S,R)$ 
with coefficients in $M$}. 
It does not depend on the choice of the resolutions $P_*$, $F_*$, 
and the chain map $f$. In fact, let $P'_*$ and $F'_*$ 
be other resolutions and $f'\colon F'_*\to P'_*$ a chain map 
respecting the homomorphism $f$ and the augmentations. By Lemma \ref{lem:4-1-3} 
(1)(2), we have homotopy equivalences $\varphi\colon F'_* \to 
F_*$ and $\psi\colon P'_*\to P_*$ respecting the identities 
$1_R$ and $1_S$, respectively. Lemma \ref{lem:4-1-3}
(2) implies $f\varphi \simeq \psi f'\colon F'_*\to P_*$. 
Hence, by Lemma \ref{lem:4-1-4}, we obtain a uniquely 
determined map $(\varphi, \psi)_*\colon 
H_*((M\otimes_SP'_*)\rtimes_{f'}(M\otimes_RF'_{*-1}))
\to H_*((M\otimes_SP_*)\rtimes_f(M\otimes_RF_{*-1}))$. 
Homotopy inverses of $\varphi$ and $\psi$ induce 
a uniquely determined map 
$H_*((M\otimes_SP_*)\rtimes_f(M\otimes_RF_{*-1})) \to 
H_*((M\otimes_SP'_*)\rtimes_{f'}(M\otimes_RF'_{*-1}))$. 
It is the inverse of the map $(\varphi, \psi)_*$ 
from Lemmas \ref{lem:4-1-4} and \ref{lem:4-1-5}. 
Hence the relative homology group $H_*(S, R; M)$ 
is well-defined.\par
By the sequence (\ref{eq:4-1-1}), we have a natural 
exact sequence
\begin{equation}
\label{eq:4-2-1}
\cdots \to H_n(R;M) \overset{f_*}\to H_n(S;M) \overset{j_*}\to
H_n(S,R;M) \overset{\partial_*}\to H_{n-1}(R;M) \to \cdots
\quad\mbox{(exact)}.
\end{equation}
We may choose $P_0=S$ and $F_0=R$. Then the boundary operator
$\begin{pmatrix} d&f\\0&-d\end{pmatrix} = \begin{pmatrix} d&1\end{pmatrix}\colon
(M\otimes_SP_1)\oplus(M\otimes_RF_0) =
(M\otimes_SP_1)\oplus M 
\to 
(M\otimes_SP_0)\oplus(M\otimes_RF_{-1}) = M
$
is surjective. Hence we have 
\begin{equation}
\label{eq:4-2-2}
H_0(S,R;M) = 0.
\end{equation}
\par
For a commutative diagram 
\begin{equation}
\label{eq:4-2-3}
\begin{CD}
R' @>{\varphi}>> R\\
@V{f'}VV @V{f}VV\\
S' @>{\psi}>> S
\end{CD}
\end{equation}
of (complete) Hopf algebras, we can take a homotopy commutative 
diagram of resolutions as in (\ref{eq:4-1-2}). By Lemma \ref{lem:4-1-4},  
it induces a well-defined map
$$
(\varphi, \psi)_*\colon H_*(S',R';M) \to H_*(S,R;M)
$$
for any $S$-module $M$. From Lemma \ref{lem:4-1-5} 
the relative homology of a pair of (complete) Hopf algebras 
satisfies a functoriality. \par
Next consider the coproducts $\Delta\colon S\to S\otimes S$ and 
$\Delta\colon R\to R\otimes R$. By the K\"unneth formula, $P_*\otimes P_*$ 
and $F_*\otimes F_*$ are acyclic. We regard them as left $S$-
and left $R$- chain complexes by using the coproducts, respectively. 
In the case where $S$ and $R$ are complete Hopf algebras,
we consider $P_*\hat{\otimes}P_*$ and $F_*\hat{\otimes}F_*$
instead, and {\it assume} they are acyclic. 
In any cases, by Lemma \ref{lem:4-1-3}, we have chain maps $\Delta\colon 
P_* \to P_*\otimes P_*$ and $\Delta\colon F_* \to F_*\otimes F_*$. 
By Lemma \ref{lem:4-1-4}, we can define a uniquely determined map
\begin{equation}
\label{eq:4-2-4}
\Delta_* := (\Delta, \Delta)_*\colon H_*(S, R; M) \to
H_*((M\otimes_S(P_*\otimes P_*))\rtimes_{f\otimes
f}(M\otimes_R(F_*\otimes F_*)_{*-1})),
\end{equation}
which we call {\it the diagonal map}.
Consider a commutative diagram of (complete) Hopf algebras as in 
(\ref{eq:4-2-3}). Take resolutions $P'_*$ and $F'_*$ over 
$S'$ and $R'$, respectively. By Lemma \ref{lem:4-1-3},  
we have chain maps $\varphi\colon F'_*\to F_*$ and $\psi\colon P'_* \to P_*$ 
respecting the Hopf algebra homomorphisms $\varphi$ and $\psi$, 
respectively, and the augmentations. The homotopy commutative diagrams 
$$
\begin{CD}
F'_* @>{\Delta}>> F'_*\otimes F'_* @>{\varphi\otimes\varphi}>>
F_*\otimes F_*\\
@V{f'}VV @V{f'\otimes f'}VV @V{f\otimes f}VV \\
P'_* @>{\Delta}>>
P'_*\otimes P'_* @>{\psi\otimes\psi}>> P_*\otimes P_*
\end{CD}\quad\mbox{and}\quad\quad
\begin{CD}
F'_* @>{\varphi}>> F_* @>{\Delta}>>
F_*\otimes F_*\\
@V{f'}VV @V{f}VV @V{f\otimes f}VV \\
P'_* @>{\psi}>>
P_* @>{\Delta}>> P_*\otimes P_*
\end{CD}
$$
respect the same commutative diagram (\ref{eq:4-2-3}). 
Hence, from Lemma \ref{lem:4-1-5}, we obtain 
the commutative diagram
\begin{equation}
\label{eq:4-2-5}
\begin{CD}
H_*(S', R'; M) @>{\Delta_*}>>
H_*((M\otimes_{S'}(P'_*\otimes P'_*))\rtimes_{f'\otimes
f'}(M\otimes_{R'}(F'_*\otimes F'_*)_{*-1}))\\
@V{(\varphi, \psi)_*}VV @V{(\varphi\otimes\varphi, \psi\otimes\psi)_*}VV\\
H_*(S, R; M) @>{\Delta_*}>>
H_*((M\otimes_S(P_*\otimes P_*))\rtimes_{f\otimes
f}(M\otimes_R(F_*\otimes F_*)_{*-1})),
\end{CD}
\end{equation}
namely, the naturality 
of the map in (\ref{eq:4-2-4}). 
Here, if $\Theta$ is a chain homotopy connecting $\psi f'$ to 
$f\varphi$, the vertical map $(\varphi,\psi)_*$ is given by 
$h(\varphi, \Theta, \psi)$, and 
$(\varphi\otimes\varphi, \psi\otimes\psi)_*$ by 
$h(\varphi\otimes\varphi, (f\varphi)\otimes\Theta+
\Theta\otimes(\psi f'), \psi\otimes\psi)$ as in (\ref{eq:4-1-3}). 

\subsection{Cap products on the relative (co)homology}

Now we introduce the cap product on the relative homology of 
a pair of (complete) Hopf algebras. In this paper our Poincar\'e 
duality is given by $H \cong H^*$, $X \mapsto (Y \mapsto Y\cdot X)$.
This means our cap product on the surface $\Sigma$ is given by 
$\langle[\Sigma], u \cup v\rangle = \langle[\Sigma]\cap u, v\rangle$
for $u, v \in H^*$. See \cite{Ka} \S5, for
details. In other words, we evaluate the cocycle $u$ on the front 
face of the cycle $[\Sigma]$ in the product $[\Sigma]\cap u$. 
Thus we define our cap product on the Hopf algebra (co)homology 
in the following way. 
We remark our sign convention here is different from 
\cite{Ka} and \cite{BKP}.
\par
Let $f\colon R \to S$ be a homomorphism of (complete) Hopf algebras 
over $\mathbb{Q}$, $M_1$ and $M_2$ left $S$-modules, $P_*$ and 
$F_*$ projective resolutions of $\mathbb{Q}$ over $S$ and $R$, 
respectively, and $f\colon F_*\to P_*$ a chain map respecting the 
homomorphism $f$ and the augmentations. We define the cap product
\begin{equation}
\label{eq:4-3-1}
\cap\colon M_1\otimes_R(F_*\otimes F_*)\otimes\Hom_S(P_*,M_2)
\to (M_1\otimes M_2)\otimes_RF_*
\end{equation}
by $\cap(u\otimes x\otimes y\otimes v) = (u\otimes x\otimes y)\cap v:=
(-1)^{\deg(x\otimes y)\deg v} u\otimes v(f(x))\otimes y$ for $u \in M_1$,
$x,y \in F_*$ and 
$v \in \Hom_S(P_*, M_2)$. Here $M_1\otimes M_2$ is regarded as an 
$R$-module by the homomorphism $f$ and the coproduct $\Delta$.
In the case $R$ is a complete Hopf algebra, we consider
the completed tensor product $M_1\hat{\otimes}M_2$ instead.
By a straightforward computation, we find out $\cap$ is a chain map. 
In the case where $R=S$ and $f=1_S$, we have a chain map
$$
\cap\colon M_1\otimes_S(P_*\otimes P_*)\otimes\Hom_S(P_*,M_2)
\to (M_1\otimes M_2)\otimes_SP_*,
$$
which is compatible with the map (\ref{eq:4-3-1}). 
Hence we obtain a chain map 
\begin{eqnarray}
(M_1\otimes_S(P_*\otimes P_*)\rtimes_{f\otimes f}
M_1\otimes_R(F_*\otimes F_*)_{*-1})&\otimes&\Hom_S(P_*,M_2)
\nonumber\\
&\to & (M_1\otimes M_2)\otimes_SP_*\rtimes_f
(M_1\otimes M_2)\otimes_RF_{*-1}
\nonumber
\end{eqnarray}
and the induced map
\begin{equation}
\label{eq:4-3-2}
\cap\colon H_*((M_1\otimes_S(P_*\otimes P_*)\rtimes_{f\otimes f}
M_1\otimes_R(F_*\otimes F_*)_{*-1})) \otimes H^*(S;M_2) 
\to H_*(S,R; M_1\otimes M_2). 
\end{equation}
We have to prove the naturality of the cap product 
(\ref{eq:4-3-2}). For the commutative diagram of 
(complete) Hopf algebras (\ref{eq:4-2-3}), choose chain maps
$\varphi\colon F'_* \to F_*$ and $\psi\colon P'_* \to P_*$ of resolutions 
as in (\ref{eq:4-2-5}).
\begin{lem}
\label{lem:4-3-1}
For any $\xi \in H_*((M_1\otimes_{S^{\prime}}
(P^{\prime}_*\otimes P^{\prime}_*))\rtimes_{f\otimes f}
(M_1\otimes_{R^{\prime}}(F^{\prime}_*\otimes F^{\prime}_*)_{*-1}))$
and $\eta \in H^*(S;M_2)$,  
we have 
$$
(\varphi, \psi)_*(\xi\cap\psi^*\eta) = ((\varphi, \psi)_*\xi)\cap \eta
\in H_*(S,R; M_1\otimes M_2).
$$
Here $(\varphi, \psi)_*\xi$ in the right hand side means 
the homology class 
$h(\varphi\otimes\varphi, (f\varphi)\otimes\Theta+
\Theta\otimes(\psi f'), \psi\otimes\psi)_*\xi$. 
\end{lem}
The lemma in the case where $R' = R$, $S'=S$, $\varphi = 1_R$ and 
$\psi = 1_S$ implies that the cap product is independent of the 
choice of the resolutions and the chain maps.
\begin{proof} Let $u, u' \in M_1$, $x, y \in P'_*$, $x', y' \in 
F'_*$ and $v \in \Hom_S(P_*, M_2)$. We denote 
$\Xi:= (f\varphi)\otimes\Theta+\Theta\otimes(\psi f')$. 
Then, by a straightforward computation, we have 
\begin{eqnarray*}
&&(-1)^{\deg(x'\otimes y')\deg v}
\begin{pmatrix}\psi\otimes\psi&-\Theta\\0&\varphi\end{pmatrix}
((u\otimes x\otimes y, u'\otimes x'\otimes y')\cap(\psi^*v))\\
&&-(-1)^{\deg(x'\otimes y')\deg v}
\left(\begin{pmatrix}\psi\otimes\psi&-\Xi\\0&\varphi\otimes\varphi\end{pmatrix}
(u\otimes x\otimes y, u'\otimes x'\otimes y')\right)\cap v\\ 
&=& (-(-1)^{\deg x'}u'\otimes(dv)(\Theta x')\otimes\Theta y', 
 (-1)^{\deg v}u'\otimes (dv)(\Theta x')\otimes\varphi y')\\
&&- ((-1)^{\deg v}u'\otimes(v\Theta\otimes\Theta)d(x'\otimes y'), 
 u'\otimes(v\Theta\otimes\varphi)d(x'\otimes y'))\\
&& + \begin{pmatrix} d&f\\0&-d\end{pmatrix} \left(
(-1)^{\deg x'}u'\otimes(v\Theta x')\otimes\Theta y', 
(-1)^{\deg v}u'\otimes(v\Theta x')\otimes\varphi y'\right). 
\end{eqnarray*}
If $v$ is a cocycle, and $(u\otimes x\otimes y, u'\otimes x'\otimes y')$
is a cycle, then the right hand side is null-homologous. 
This proves the lemma.
\end{proof}
Taking the composite of the map $\cap$ in (\ref{eq:4-3-2}) 
and the diagonal map $\Delta_*$ in (\ref{eq:4-2-4}), 
we obtain the cap product
\begin{equation}
\label{eq:4-3-3}
\cap:= \cap\circ\Delta_*\colon H_*(S,R; M_1)\otimes H^*(S; M_2) 
\to H_*(S,R; M_1\otimes M_2).  
\end{equation}
From the naturality of the diagonal map
$\Delta_*$ (\ref{eq:4-2-5}) and Lemma
\ref{lem:4-3-1}, this is independent of 
the choice of resolutions and chain maps. 
We also obtain the naturality of the cap product:
\begin{prop}
\label{prop:4-3-2}
In the situation of the commutative diagram {\rm (\ref{eq:4-2-3})},
let $M_1$ and $M_2$ be left $S$-modules.
For any $\xi \in H_*(S^{\prime},R^{\prime}; M_1)$ and $\eta \in H^*(S;M_2)$, 
we have 
$$
(\varphi, \psi)_*(\xi\cap\psi^*\eta) = ((\varphi, \psi)_*\xi)\cap \eta
\in H_*(S,R; M_1\otimes M_2).
$$
\end{prop}

\subsection{Kronecker product}
We recall the Kronecker product on the (co)homology 
of a Hopf algebra. 
Let $S$ be a (complete) Hopf algebra over $\mathbb{Q}$, 
$P_*$ an $S$-projective resolution of $\mathbb{Q}$,
and $M_1$ and $M_2$ left $S$-modules. 
The Kronecker product on the (co)chain level
\begin{equation}
\label{eq:4-4-1}
\langle\,, \,\rangle\colon (M_1\otimes_SP_*)\otimes\Hom_S(P_*, M_2) 
\to M_1\otimes_SM_2
\end{equation}
is defined by $\langle u\otimes x, v\rangle = (-1)^{\deg x\deg v}
u\otimes v(x)$ for $u \in M_1$, $x \in P_*$ and $v \in 
\Hom_S(P_*, M_2)$. Since $\langle d(u\otimes x), v\rangle
= (-1)^{\deg x}\langle u\otimes x, dv\rangle$, we have 
the Kronecker product on the (co)homology level
$$
\langle\,, \,\rangle\colon H_*(S;M_1)\otimes H^*(S;M_2) 
\to M_1\otimes_SM_2.
$$
Let $\psi\colon S'\to S$ be a homomorphism of (complete) 
Hopf algebras, $P_*$ an $S'$-projective resolution of 
$\mathbb{Q}$, and $\psi\colon P'_*\to P_*$ a chain map 
which respects the homomorphism $\psi$ and the augmentations. Then we have 
\begin{equation}
\label{eq:4-4-2}
\langle\psi_*u, v\rangle = \langle u, \psi^*v\rangle
\end{equation}
for any $u \in H_*(M_1\otimes_{S^{\prime}}P^{\prime}_*)$ and
$v \in H^*(\Hom_S(P_*, M_2))$. Hence the Kronecker product 
is independent of the choice of the resolution $P_*$, 
and has a naturality. 

\subsection{Homology of a pair of groups}
Let $G$ be a group, $K$ a subgroup of $G$, and 
$M$ a left $\mathbb{Q}G$-module. As was stated 
in \S3.3, we have $H_*(G;M) = H_*(\mathbb{Q}G;M)$ 
and $H^*(G;M) = H^*(\mathbb{Q}G;M)$. The normalized 
standard complex $\bar{F}_*(G)$ is a 
$\mathbb{Q}G$-projective resolution of $\mathbb{Q}$. 
Since the inclusion map $C_*(K;M) = 
M\otimes_{\mathbb{Q}K}\bar{F}_*(K)
\to C_*(G;M) = 
M\otimes_{\mathbb{Q}G}\bar{F}_*(G)$
is injective, the mapping cone $C_*(G;M)\rtimes
C_{*-1}(K;M)$ is naturally quasi-isomorphic to 
the quotient complex $C_*(G;M)/C_*(K;M)$ from 
Lemma \ref{lem:4-1-1}. Hence we have a natural 
isomorphism
\begin{equation}
\label{eq:4-5-1}
H_*(G,K;M) = H_*(\mathbb{Q}G,\mathbb{Q}K;M).
\end{equation}
The standard complex $F_*(G) = \{F_n(G)\}$ is also 
a natural $\mathbb{Q}G$-projective resolution of $\mathbb{Q}$, 
so that it can be used for computing the relative homology
$H_*(\mathbb{Q}G,\mathbb{Q}K; M)$. The Alexander-Whitney map
$$
\Delta\colon F_*(G) \to F_*(G)\otimes F_*(G)
$$
is a $\mathbb{Q}G$-chain map respecting the augmentation maps.
See, for example, \cite{Bro} p.108. Hence the cap product on the relative 
homology $H_*(\mathbb{Q}G,\mathbb{Q}K;M)$ of the pair $(\mathbb{Q}G, 
\mathbb{Q}K)$ introduced in \S4.3 coincides with the usual cap
product  on the relative homology $H_*(G,K;M)$ of the pair $(G,K)$ via 
the isomorphism (\ref{eq:4-5-1}).\par
On the other hand, consider the classifying spaces $BG$ and $BK$. 
We assume $BK$ is realized as a subspace of $BG$. Choose a basepoint
$* \in BK$. Denote by $\Delta^n$ the standard $n$-simplex, and 
by $S_*(X)$ the rational singular chain complex of a topological space 
$X$. For any $g \in G$ we choose a continuous map $\rho(g)\colon \Delta^1 
\to BG$ satisfying the conditions
\begin{enumerate}
\item $\rho(g)(0) = \rho(g)(1) = *$ under the natural identification 
$\Delta^1\approx [0,1]$,
\item the based homotopy class of $\rho(g)$ is exactly $g \in G =
\pi_1(BG,*)$, and
\item $\rho(k)(\Delta^1) \subset BK$ if $k \in K$.
\end{enumerate}
This assignment defines a $\mathbb{Q}G$-map $\rho\colon F_1(G) \to S_1(EG)$
and a $\mathbb{Q}K$-map $\rho\colon F_1(K) \to S_1(EK)$, where $EG$ and 
$EK$ is the universal covering spaces of $BG$ and $BK$, respectively.
Since the spaces $BG$ and $BK$ are aspherical, the map $\rho$ extends
to a $\mathbb{Q}G$-chain map $\rho\colon F_*(G) \to S_*(EG)$ and 
a $\mathbb{Q}K$-chain map $\rho\colon F_*(K) \to S_*(EK)$. 
Since the map $\rho$ respects the augmentations, it induces 
a natural isomorphism
\begin{equation}
\label{eq:4-5-2}
\rho_*\colon H_*(G,K; M) \to H_*(BG, BK; M).
\end{equation}
In the right hand side we regard $M$ as the local system 
on the space $BG$ associated with the $G$-module $M$. 
From the construction of the map $\rho$, we have a commutative diagram
of $\mathbb{Q}G$-chain maps
$$\begin{CD}
F_*(G) @>{\Delta}>> F_*(G)\otimes F_*(G)\\
@V{\rho}VV @V{\rho\otimes\rho}VV\\
S_*(EG)@>{\Delta}>> S_*(EG)\otimes S_*(EG),
\end{CD}
$$
where the lower $\Delta$ is the Alexander-Whitney map on the singular 
chain complex.
Hence the cap product on the relative homology $H_*(G,K;M)$ of the pair
$(G,K)$ coincides with the cap product  on the relative homology
$H_*(BG,BK;M)$ of the pair $(BG,BK)$ of topological spaces 
via the isomorphism (\ref{eq:4-5-2}).

\section{(Co)homology theory of $\widehat{T}$ and 
$(\widehat{T}, \mathbb{Q}[[\omega]])$}

Following \S4, $H_*(\T;M)$,
$H^*(\T;M)$ and $H_*(\T, \Qomega;M)$ are defined for any $\T$-module
$M$. Here $\mathbb{Q}[[\omega]]$ is the ring of formal power
series in the symplectic form $\omega$, which is
regarded as a Hopf subalgebra of $\widehat{T}$ in an obvious way.
In this section we describe them in an explicit way, 
prove the Poincar\'e duality for the pair $(\T, \Qomega)$, and 
give a homological interpretation of symplectic derivations of the 
algebra $\T$.  \par

\subsection{Explicit description of (co)homology of 
$\widehat{T}$ and 
$(\widehat{T}, \mathbb{Q}[[\omega]])$}

Let $S$ be a (complete) Hopf algebra over $\mathbb{Q}$. 
We denote by $IS$ the augmentation ideal of $S$, namely, 
$IS := \Ker(\varepsilon\colon S \to \mathbb{Q})$, and by 
$\partial$ the inclusion map $IS \hookrightarrow S$. 
Then $P_*(S) := (IS\overset\partial\to S)$ is a left 
$S$-resolution of the trivial $S$-module $\mathbb{Q}$. 
For a left $S$-module $M$ we denote 
\begin{eqnarray*}
&& D_*(S;M) := M\otimes_SP_*(S) =(M\otimes_SIS\to M\otimes_SS), 
\quad\mbox{and}\\
&& D^*(S;M) := \Hom_S(P_*(S),M) =(\Hom_S(IS,M) \leftarrow
\Hom_S(S,M)).
\end{eqnarray*}
Let $f\colon R\to S$ be a homomorphism of (complete) Hopf algebras. 
It induces a natural homomorphism $f\colon IR\to IS$ and natural 
(co)chain maps $f\colon D_*(R;M) \to D_*(S;M)$ and 
$f^*\colon D^*(S;M) \to D^*(R;M)$. The mapping cone
$D_*(S,R;M) := D_*(S;M)\rtimes_fD_{*-1}(R;M)$ has an acyclic 
subcomplex $M\otimes_RR = M \overset{1_M}\to M= M\otimes_SS$. 
We denote the quotient complex by $\bar{D}_*(S,R;M)$, 
which is given by 
\begin{equation*}
\bar{D}_*(S,R;M) = \left\{
\begin{array}{ll}
M\otimes_RIR,& \quad \mbox{if $*=2$,}\\
M\otimes_SIS,& \quad \mbox{if $*=1$,}\\
0,& \quad \mbox{otherwise,}
\end{array}\right.
\end{equation*}
and 
$$
\partial_2=1_M\otimes f\colon \bar{D}_2(S,R;M) = M\otimes_RIR
\to M\otimes_SIS = \bar{D}_1(S,R;M).
$$
The natural projection $\varpi\colon D_*(S,R;M) \to \bar{D}_*(S,R;M)$ 
is a quasi-isomorphism. 
\par
We call the (complete) Hopf algebra $S$ {\it free},
if $IS$ is a left $S$-free module. 
For example, the algebras $\T$, $\Qomega$, $\mathbb{Q}\pi$ 
and $\Qzeta$ are free. Then $P_*(S)$ is a left $S$-projective 
resolution of $\mathbb{Q}$. Hence we have 
\begin{eqnarray}
&&H_*(S;M) = H_*(D_*(S;M)) = 
H_*(M\otimes_SIS \overset{1_M\otimes\partial}\longrightarrow M\otimes_SS)
\label{eq:5-1-1}\\
&&H^*(S;M) = H^*(D^*(S;M)) = 
H^*(\Hom_S(IS,M) \overset{\partial^*}\longleftarrow \Hom_S(S,M))
\label{eq:5-1-2}
\end{eqnarray}
as in (\ref{eq:3-3-1}). 
If $R$ is also free, then we have
\begin{equation}
\label{eq:5-1-3}
H_*(S,R;M) = H_*(\bar{D}_*(S,R;M)) = 
H_*(M\otimes_RIR\overset{1_M\otimes f}\longrightarrow M\otimes_SIS\to 0).
\end{equation}

\begin{lem}
\label{lem:5-1-1}
Let $S$ and $R$ be free {\rm (}complete{\rm )} Hopf algebras, $f\colon R\to S$ 
a homomorphism of {\rm (}complete{\rm )} Hopf algebras, and 
$M$ a trivial $S$-module. Then
\begin{enumerate}
\item[{\rm (1)}] 
\begin{equation*}
H_*(S;M) = \left\{
\begin{array}{ll}
M,& \quad \mbox{if $*=0$,}\\
M\otimes(IS/IS^2),& \quad \mbox{if $*=1$,}\\
0,& \quad \mbox{otherwise.}
\end{array}\right.
\end{equation*}
\item[{\rm (2)}] If $f(IR) \subset IS^2$, then 
\begin{equation*}
H_*(S,R;M) = \left\{
\begin{array}{ll}
H_1(R;M),& \quad \mbox{if $*=2$,}\\
H_1(S;M),& \quad \mbox{if $*=1$,}\\
0,& \quad \mbox{otherwise.}
\end{array}\right.
\end{equation*}
In particular, $\partial_*\colon H_2(S,R;M) \to H_1(R;M)$ 
is an isomorphism.
\end{enumerate}
\end{lem}
\begin{proof}
Since $M$ is a trivial module, $1_M\otimes\partial\colon
M\otimes_SIS\to M\otimes_SS$ is a zero map. 
Hence $H_0(S;M) = M$ and $H_1(S;M) = M\otimes_SIS$. 
The map $M\otimes_SIS \to M\otimes_\mathbb{Q}(IS/IS^2)$,
$u\otimes a \mapsto u\otimes(a\bmod IS^2)$, is a well-defined 
isomorphism. This proves the first part.\par
From the assumption $f(IR) \subset IS^2$,
$f_*\colon M\otimes (IR/IR^2) \to M\otimes (IS/IS^2)$ is a zero map. 
Hence the homology exact sequence (\ref{eq:4-2-1}) implies 
the second part.
\end{proof}

Consider the case $S=\T$ and $R=\Qomega$. 
The inclusion map $i\colon\Qomega\to \T$ is a homomorphism of 
complete Hopf algebras. Then we have $IS=\T_1 = \T\otimes H$ 
as a left $\T$-module, so that $M\otimes_SIS = M\otimes_{\T}\T
\otimes H= M\otimes H$ and $\Hom_S(IS,M) = \Hom_{\T}(\T\otimes H,M)
= \Hom(H,M)$. 
Under these isomorphisms, the operators 
$1_M\otimes\partial$ and $\partial^*$ are given by
\begin{eqnarray*}
&&\partial_M\colon M\otimes H \to M, \quad m\otimes X \mapsto \iota(X)m,
\quad\mbox{and}\\
&&\delta_M\colon M \to \Hom(H,M), \quad m \mapsto (X\mapsto Xm),
\end{eqnarray*}
respectively.
Hence we have
\begin{eqnarray}
H_*(\T; M) &=& H_*(M\otimes H\overset{\partial_M}\to M),
\quad \mbox{and}\label{eq:5-1-4}\\
H^*(\T;M) &=& H^*(\Hom(H,M) \overset{\delta_M}\leftarrow M).
\label{eq:5-1-5}
\end{eqnarray}
A similar result holds for $R=\Qomega$.
Under the isomorphism $M\otimes_RIR = M\otimes\mathbb{Q}\omega
= M$, the boundary operator in $\bar{D}_*(S,R;M)$ is given by 
$$
d_M\colon M \to M\otimes_{\T}\T_1 = M\otimes H, \quad
m \mapsto m\otimes\omega 
= \sum^g_{i=1} -(A_im)\otimes B_i+(B_im)\otimes A_i.
$$
Hence we have 
\begin{equation}
\label{eq:5-1-6}
\bar{D}_*(\T,\Qomega;M) = (M\overset{d_M}\to M\otimes H\to 0).
\end{equation}
Now we recall the space $H$ and its dual $H^*$ are identified by the map 
$$
\vartheta\colon H \overset\cong\to H^*, \quad X \mapsto (Y\mapsto Y\cdot X),
$$
as in (\ref{eq:2-7-1}), and introduce the isomorphisms
\begin{eqnarray*}
&&\vartheta\colon \bar{D}_1(\T,\Qomega;M) = M\otimes H \overset\cong\to 
H^*\otimes M = D^1(\T;M),\quad m\otimes X \mapsto -\vartheta(X)\otimes m, 
\\
&&\vartheta\colon \bar{D}_2(\T,\Qomega;M) = M \overset\cong\to 
M = D^0(\T;M),\quad m \mapsto -m.
\end{eqnarray*}
It is easy to check they constitute a chain map up to sign, and 
induce an isomorphism of cochain complexes
\begin{equation}
\label{eq:5-1-7}
\vartheta\colon \bar{D}_{2-*}(\T,\Qomega;M) \overset\cong\to D^*(\T;M).
\end{equation}
Hence we have an isomorphism $H_{2-*}(\T,\Qomega;M) \cong H^*(\T;M)$. 
In the next subsection we interpret this isomorphism as a certain 
kind of the Poincar\'e duality. 

\subsection{Poincar\'e duality for the pair 
$(\T, \Qomega)$}
We begin by introducing the fundamental class $\LfL \in 
H_2(\T,\Qomega; \mathbb{Q})$, which is a counterpart of 
the fundamental class $[\Sigma] \in H_2(\Sigma,\partial\Sigma; 
\mathbb{Q})$ of the surface $\Sigma$. For $R=\Qomega$, we have 
$IR/IR^2 = \Qomega\omega/\Qomega\omega^2 = \mathbb{Q}\omega$. 
By Lemma \ref{lem:5-1-1}(1), we have $H_1(\Qomega;\mathbb{Q}) 
= \mathbb{Q}\omega \cong\mathbb{Q}$.
Since $i(IR) \subset IS^2$ for $S = \T$, the connecting 
homomorphism $\partial_*\colon H_2(S,R;\mathbb{Q})\to 
H_1(R;\mathbb{Q})$ is an isomorphism from Lemma \ref{lem:5-1-1} 
(2). We define 
\begin{equation}
\label{eq:5-2-1}
\LfL := -{\partial_*}^{-1}(\omega) \in H_2(\T,\Qomega; \mathbb{Q}),
\end{equation}
which spans $H_2(\T,\Qomega; \mathbb{Q})\cong \mathbb{Q}$
and is represented by $(0, -\omega)$ in $D_2(\T,\Qomega;\mathbb{Q})
= 0\oplus\mathbb{Q}\otimes_{\Qomega}\Qomega\omega$. 
We call it {\it the fundamental class of the pair $(\T,\Qomega)$}.
We have a certain kind of the Poincar\'e duality with respect to this
fundamental class $\LfL$. 
\begin{prop}
\label{prop:5-2-1}
The cap product by the fundamental class $\LfL$ gives an isomorphism
$$
\LfL\cap\colon H^*(\T;M) \overset\cong\to H_{2-*}(\T,\Qomega;M)
$$
for any left $\T$-module $M$. 
In particular, the cochain map $\vartheta$ in 
{\rm (\ref{eq:5-1-7})} induces the inverse of the map 
$\LfL\cap$. 
\end{prop}
\begin{proof}
We begin by computing the diagonal map (\ref{eq:4-2-4})
\begin{equation}
\label{eq:5-2-2}
\Delta_*\colon H_*(\T, \Qomega; M) \to
H_*((M\otimes_S(P_*\hat{\otimes} P_*))\rtimes_{i\otimes
i}(M\otimes_R(F_*\hat{\otimes} F_*)_{*-1}))
\end{equation}
explicitly. 
Here we write simply $P_* = P_*(\T)$ and $F_* = P_*(\Qomega)$.
It should be remarked the completed tensor product $P_*(\T)\hotimes
P_*(\T)$ given by 
$$
\begin{array}{rcccl}
(\T_1\hotimes\T_1&\overset{\partial_2}\to
&\T_1\hotimes\T\oplus\T\hotimes\T_1&\overset{\partial_1}\to&
\T\hotimes\T)\\
u&\mapsto&(-u,u)&&\\
&&(v,w)&\mapsto& v+w
\end{array}
$$
is acyclic. We construct a chain map $\Delta\colon P_*(\T) \to 
P_*(\T)\hotimes P_*(\T)$ respecting the coproduct $\Delta$ 
as follows. In degree $0$ we define $\Delta\colon P_0(\T) = \T 
\to (P_*(\T)\hotimes P_*(\T))_0 = \T\hotimes\T$ by the 
coproduct $\Delta$ itself. In degree $1$ we define $\Delta(X) 
:= (X\hat{\otimes} 1, 1\hat{\otimes} X)
\in \T_1\hotimes\T\oplus\T\hotimes\T_1
= (P_*(\T)\hotimes P_*(\T))_1$ for $X \in H$, and extend it to 
the whole $\T_1= \T\otimes H$ as a left $\T$-homomorphism. 
Since $X\hat{\otimes} 1+1\hat{\otimes} X = \Delta(X) \in \T\hotimes\T$, 
this map $\Delta$ is a $\T$-chain map. We define a 
$\Qomega$-chain map $\Delta\colon P_*(\Qomega) \to 
P_*(\Qomega)\hotimes P_*(\Qomega)$ in a similar way. 
We have $\Delta(\omega) = (\omega\hat{\otimes} 1, 1\hat{\otimes}\omega)$ 
for $\omega \in \Qomega\omega = P_1(\Qomega)$. 
Then the (homotopy commutative) diagram
$$
\begin{CD}
P_*(\Qomega) @>{\Delta}>> P_*(\Qomega)\hotimes P_*(\Qomega)\\
@V{i}VV @V{i\otimes i}VV\\
P_*(\T) @>{\Delta}>> P_*(\T)\hotimes P_*(\T)
\end{CD}
$$
does {\it not} commute. If we denote
$$
\hat{\omega} := \sum^g_{i=1}A_i\hat{\otimes} B_i - B_i\hat{\otimes} A_i
\in \T_1\hotimes\T_1,
$$
then 
$$
\Delta i(\omega) = (\omega\hat{\otimes} 1 - \hat{\omega},\
\hat{\omega}+1\hat{\otimes}\omega) = (i\otimes i)\Delta(\omega) +
\partial_2\hat{\omega}.
$$
This means the $\Qomega$-homomorphism
$$
\Phi\colon P_*(\Qomega) \to (P_*(\Qomega)\hotimes P_*(\Qomega))_{*+1}
$$
defined by $\Phi\vert_{P_0} =0$ and $(\Phi\vert_{P_1})(\omega) =
-\hat{\omega}$ satisfies the relation 
$(i\otimes i)\Delta - \Delta i = d\Phi +\Phi d$.  
Hence the diagonal map (\ref{eq:5-2-2}) is given by 
$h(\Delta, \Phi, \Delta) = \begin{pmatrix} \Delta&-\Phi\\0& \Delta\end{pmatrix}$. 
In particular,
the homology class $\Delta_*\LfL$ is represented by the cycle
$$
\begin{pmatrix}\Delta&-\Phi\\0&
\Delta\end{pmatrix}\begin{pmatrix}0\\-\omega\end{pmatrix} = 
\begin{pmatrix}-\hat{\omega}\\(-\omega\hat{\otimes} 1,\ -1\hat{\otimes}\omega)\end{pmatrix}
\in (M\otimes_S(P_*\hat{\otimes} P_*))\rtimes_{i\otimes
i}(M\otimes_R(F_*\hat{\otimes} F_*)_{*-1}).
$$
By the explicit definition of the cap product (\ref{eq:4-3-1}), 
we have 
\begin{eqnarray*}
&&(\Delta_*\LfL)\cap m = (0, -m\otimes\omega)\\
&&(\Delta_*\LfL)\cap v = \left(\sum_{i=1}^g-v(A_i)\otimes B_i+v(B_i)\otimes A_i,
\sum_{i=1}^g A_iv(B_i)-B_iv(A_i)\right)
\end{eqnarray*}
for $m \in M = D^0(\T;M)$ and $v \in \Hom_{\T}(\T_1,M) = D^1(\T;M)$. 
Hence $\varpi\circ((\Delta_*\LfL)\cap)\colon D^*(\T;M) \to
\bar{D}_{2-*}(\T,\Qomega;M)$ is exactly the inverse of the map 
$\vartheta$ (\ref{eq:5-1-7}). This proves the proposition.
\end{proof}

In a way similar to the surface $\Sigma$ we can introduce the
intersection form
\begin{equation}
\label{eq:5-2-3}
(\ \cdot\ )\colon H_1(\T; M_1)\otimes H_1(\T,\Qomega;M_2) \to
M_1\otimes_{\T}M_2, \quad u\otimes v \mapsto \langle u,
(\LfL\cap)^{-1}v\rangle
\end{equation}
for any left $\T$-modules $M_1$ and $M_2$. Here $\langle\ ,
\ \rangle$ is the Kronecker product (\ref{eq:4-4-1}). 
Under the identifications (\ref{eq:5-1-4}) and (\ref{eq:5-1-6}), 
the intersection form coincides with the pairing
\begin{equation}
\label{eq:5-2-4}
(\ \cdot\ )\colon M_1\otimes H\otimes M_2\otimes H \to M_1\otimes_{\T}M_2,
\quad m_1\otimes X_1\otimes m_2\otimes X_2 
\mapsto (X_1\cdot X_2)m_1\otimes m_2.
\end{equation}
In fact, we have $\langle m_1\otimes X_1, \vartheta(m_2\otimes
X_2)\rangle = -\langle m_1\otimes X_1, \vartheta(X_2)\otimes m_2\rangle 
= (X_1\cdot X_2)m_1\otimes m_2$. The inclusion homomorphism 
$j_*\colon H_1(\T; M_1) \to H_1(\T,\Qomega; M_1)$ is induced by 
the composite 
$H_1(\T; M_1)\hookrightarrow M_1\otimes H \twoheadrightarrow 
H_1(\T,\Qomega; M_1)$. Hence the intersection form
$$
(\ \cdot\ )\colon H_1(\T; M_1)\otimes H_1(\T;M_2) \to
M_1\otimes_{\T}M_2, \quad u\otimes v \mapsto \langle u,
(\LfL\cap)^{-1}j_*v\rangle
$$
also coincides with the pairing (\ref{eq:5-2-4}). \par
In the succeeding subsections we use these intersections 
to give a homological interpretation of the Lie algebras 
$\mathfrak{a}_g^-$ and $\mathfrak{l}_g$ and 
symplectic derivations of the algebra $\T$.

\subsection{Homological interpretation of 
$\mathfrak{a}_g^-$ and $\mathfrak{l}_g$}

The space $H$ acts on the spaces $\T$ and $\LL$ by 
$Xu := [X,u]$ and $Xv := [X,v]$ for $X \in H$, $u \in \T$ 
and $v \in \LL$, respectively. This action extends to the whole 
algebra $\T$. In fact, we introduce an action of the algebra 
$\T\hotimes\T$ on the space $\T$ by 
$$
\mathcal{C}'\colon (\T\hotimes\T)\otimes\T \to \T, \quad
(v'\hat{\otimes} v'')\otimes u\mapsto v'u\iota(v'')
$$
for $u, v', v'' \in \T$. The space $\T$ is a left $\T\hotimes\T$-module
by the map $\mathcal{C}'$.
We have $\mathcal{C}'(\Delta(X_1\cdots X_n)\otimes u) =
[X_1,[X_2,[\cdots[X_n,u]\cdots]]]$ for $X_i\in H$ and $u \in \T$. 
Hence the action
$$
\T\otimes\T \to \T, \quad 
v\otimes u\mapsto \mathcal{C}'((\Delta v)\otimes u)
$$
is exactly an extension of the action of $H$ stated above. 
We denote by $\T^c$ and $\LL^c$ the left $\T$-modules 
defined by this action. 
In particular, if $v \in \widehat{T}$ is group-like, we have 
$\mathcal{C}'((\Delta v)\otimes u) = vu\iota(v)$. 
Hence these modules correspond to the $\mathbb{Q}\pi$-modules
$\mathbb{Q}\pi^c$ and ${\rm Lie}\widehat{\mathbb{Q}\pi}^c$, 
respectively. We denote by ${\T_1}^c$ the $\T$-submodule of $\T^c$
whose underlying subspace is $\T_1$. 
\par
As was stated in (\ref{eq:2-7-3}), the Lie algebra $\agminus =\deromega(\T)$ 
is identified with $\Ker([\ ,\ ]\colon H\otimes\T\to\T) = N(\T_1)$, 
and the Lie algebra $\llg=\deromega(\LL)$ with $\Ker([\ ,\ ]\colon
H\otimes\LL\to\LL) = N(\LL\hotimes\LL)$. 
Hence, from (\ref{eq:5-1-4}), we obtain 
\begin{eqnarray}
&& \agminus = N(\T_1) = H_1(\T; \T^c),
\label{eq:5-3-1}\\ 
&& \ag = N(\T_2) = H_1(\T; {\T_1}^c),
\quad\mbox{and}\label{eq:5-3-2}\\ 
&& \llg = N(\LL\hotimes\LL) =
H_1(\T; \LL^c).\label{eq:5-3-3}
\end{eqnarray}
\par
The brackets on the Lie algebras $\agminus$ and $\llg$ can be 
interpreted as intersection forms on the homology introduced 
in (\ref{eq:5-2-3}). We introduce a map
$$
\mathcal{B}\colon \T^c\otimes_{\T}\T^c \to N(\T_1) = \agminus, \quad
u\otimes v \mapsto N(uv),
$$
which is well-defined, since $N([u,X]v) = N(u[X,v])$ for $X \in H$
(Lemma 2.6.2 (2)).\par
For positive integers $n$ and $m$, by a straightforward computation, 
we have 
\begin{lem}
\label{lem:5-3-1}
\begin{eqnarray*}
&&[N(X_1\cdots X_n), N(Y_1\cdots Y_m)] \\
&=& N((N(X_1\cdots X_n))(Y_1\cdots Y_m)) \\
&=& -\mathcal{B}(N(X_1\cdots X_n)\cdot N(Y_1\cdots Y_m))\\
&=& -\sum^n_{i=1}\sum^m_{j=1}(X_i\cdot Y_j)N(X_{i+1}\cdots X_nX_1\cdots
X_{i-1}Y_{j+1}\cdots Y_mY_1\cdots Y_{j-1})
\end{eqnarray*}
for $X_i, Y_j \in H$. Here the bracket $[\ ,\ ]$ is 
that as derivations of $\T$,
and $(N(X_1\cdots X_n))(Y_1\cdots Y_m)$ is
the action of $N(X_1\cdots X_n)$ on the tensor $Y_1\cdots Y_m$
as a derivation.
The third term is minus the pairing $(\ \cdot\ )$ in {\rm (\ref{eq:5-2-4})} 
of $N(X_1\cdots X_n)$ and $N(Y_1\cdots Y_m) \in \T\otimes H$
applied by the map $\mathcal{B}$.
\end{lem}
Hence we obtain
\begin{prop}
\label{prop:5-3-2}
Under the identifications {\rm (\ref{eq:5-3-1})} and {\rm (\ref{eq:5-3-3})},
the brackets on the Lie algebras $\agminus$ and $\llg$ coincide with 
minus the intersection forms
\begin{eqnarray*}
&& -\mathcal{B}(\ \cdot\ )\colon H_1(\T; \T^c)\otimes H_1(\T; \T^c)
\to N(\T_1) = \agminus, \quad\mbox{and}\\
&& -\mathcal{B}(\ \cdot\ )\colon H_1(\T; \LL^c)\otimes H_1(\T; \LL^c)
\to N(\LL\otimes\LL) = \llg,
\end{eqnarray*}
respectively.
\end{prop}

\subsection{Homological interpretation of symplectic
derivations of $\T$}

In order to interpret symplectic derivations of the algebra 
$\T$, we introduce three left $\T$-modules $\T^r$, $\T^l$ and 
$\T^t$, which correspond to the left $\mathbb{Q}\pi$-modules 
$\mathbb{Q}\pi^r$, $\mathbb{Q}\pi^l$ and $\mathbb{Q}\pi^t$. 
As vector spaces these three modules are the same $\T$.
The action of the algebra $\T$ is given by the multiplication
$$
u(v^r) := v^r\iota(u), \quad  
u(v^l) := uv^l, \quad  \mbox{and}\quad
u(v^t) :=\varepsilon(u)v^t
$$
for $u \in \T$, $v^r \in \T^r$, $v^l \in \T^l$ and $v^t \in \T^t$. 
Denote by $T$ the tensor algebra of $H$, $T:= \bigoplus^\infty_{n=0}
H^{\otimes n}$. We define a map $\xi\colon T \to (\T^r\hotimes
\T^l)\otimes_{\T}\T_1$ by 
$$
\xi(u) := 1 \otimes(1\otimes(1-\varepsilon))(\Delta u) 
= 1\otimes (\Delta u-u\otimes 1)
$$
for $u \in T$. In this expression, we regard $(\T^r\hotimes
\T^l)\otimes_{\T}\T_1$ as the natural quotient of
$(\T^r\hotimes \T^l)\otimes \T_1$.
Then we have 
\begin{lem}
\label{lem:5-4-1}
$$
\xi(X_1\cdots X_n) = \sum^n_{i=1}(X_1\cdots X_{i-1}\otimes
X_{i+1}\cdots X_n)\otimes_{\widehat{T}} X_i
$$
for $n \geq 1$ and $X_i \in H$. 
\end{lem}
\begin{proof}
First note that $\widehat{T}\otimes \widehat{T}$
acts on $(\T^r\hotimes \T^l)\otimes \T_1$
from the right, by $(u \otimes v\otimes w)(x\otimes y)
=u\otimes vx\otimes wy$,
and this action is compatible with the quotient map
$(\T^r\hotimes \T^l)\otimes \T_1 \rightarrow
(\T^r\hotimes \T^l)\otimes_{\T}\T_1$.
In the below, $1\otimes(\Delta w-w\otimes 1)(1\otimes X_n)$
means the result of the application of $1\otimes X_n$
to $1\otimes(\Delta w-w\otimes 1)$ with respect to this action, etc.

We prove the lemma by induction on $n \geq 1$. 
If $n=1$, we have $\xi(X_1) = 1\otimes(\Delta X_1 - X_1\otimes 1) 
= 1\otimes 1\otimes X_1$. Suppose $n \geq 2$. 
Denote $w := X_1\cdots X_{n-1}$. By the inductive assumption, 
we compute
\begin{eqnarray*}
&& 1\otimes(\Delta w-w\otimes 1)(1\otimes X_n)
= \sum^{n-1}_{i=1}X_1\cdots X_{i-1}\otimes 
X_{i+1}\cdots X_{n-1}\otimes X_iX_n\\
&=& -\sum^{n-1}_{i=1}(\Delta X_i)(X_1\cdots X_{i-1}\otimes 
X_{i+1}\cdots X_{n-1})\otimes X_n\\
&=& -\sum^{n-1}_{i=1}(X_i\otimes 1+1\otimes X_i)
(X_1\cdots X_{i-1}\otimes 
X_{i+1}\cdots X_{n-1})\otimes X_n\\
&=& \sum^{n-1}_{i=1}
(X_1\cdots X_{i}\otimes 
X_{i+1}\cdots X_{n-1})\otimes X_n
- (X_1\cdots X_{i-1}\otimes 
X_{i}\cdots X_{n-1})\otimes X_n\\
&=& X_1\cdots X_{n-1}\otimes 1\otimes X_n 
- 1\otimes X_1\cdots X_{n-1}\otimes X_n = w\otimes 1\otimes X_n -
1\otimes w\otimes X_n.
\end{eqnarray*}
Hence we have $1\otimes(\Delta w)(1\otimes X_n) = w\otimes 1\otimes X_n$.
Using the inductive assumption again, we compute
\begin{eqnarray*}
&& \xi(wX_n) = 1\otimes(\Delta(wX_n)-wX_n\otimes 1)\\
&=& 1\otimes \left( \Delta w(X_n\otimes 1+1\otimes X_n)
-wX_n\otimes1 \right)\\
&=& 1\otimes(\Delta w-w\otimes 1)(X_n\otimes 1) + 1\otimes(\Delta w)(1\otimes
X_n)\\
&=& \sum^{n-1}_{i=1}X_1\cdots X_{i-1}\otimes 
X_{i+1}\cdots X_{n-1}X_n\otimes X_i 
+ X_1\cdots X_{n-1}\otimes 1\otimes X_n\\
&=& \sum^{n}_{i=1}X_1\cdots X_{i-1}\otimes 
X_{i+1}\cdots X_n\otimes X_i.
\end{eqnarray*}
This completes the induction.
\end{proof}
From this lemma, the map $\xi$ is a graded homomorphism 
of degree $0$. Hence it extends to the whole $\T$ and induces 
a map
$$
\xi\colon \T \to (\T^r\hotimes\T^l)\otimes_{\T}\T_1 
= \bar{D}_1(\T,\Qomega; \T^r\hotimes\T^l) \to 
H_1(\T,\Qomega; \T^r\hotimes\T^l),
$$
which corresponds to the map in Definition \ref{def:3-5-1}. Consider the map
$$
\mathcal{C}\colon \T^c\otimes_{\T}(\T^r\hotimes\T^l) \to \T^t, \quad
w\otimes u\otimes v \mapsto uwv,
$$
which is well-defined, since $\mathcal{C}(Xw\otimes u\otimes v) 
+\mathcal{C}(w\otimes X(u\otimes v)) = u(Xw-wX)v -uXwv+uwXv = 0$ 
for any $X \in H$. Then we have
\begin{lem}
\label{lem:5-4-2}
$$
\mathcal{C}(w\cdot\xi(u)) = (\vartheta w)(u) \in \T
$$ 
for any $w \in \T^c\otimes H$ and $u \in \T$. Here $\vartheta\colon 
\T^c\otimes H \to H^*\otimes\T^c$, $m\otimes Y \mapsto -(\vartheta
Y)\otimes m$, is the map given in {\rm (\ref{eq:5-1-7})}. The right hand side 
means the action of $\vartheta w$ on the tensor $u$
as a derivation.
\end{lem}
\begin{proof}
It suffices to prove the lemma for $u=X_1\cdots X_n$, $X_i\in H$, 
$w=m\otimes Y$, $m \in \T^c$ and $Y \in H$. From Lemma \ref{lem:5-4-1}, 
\begin{eqnarray*}
&& \mathcal{C}(w\cdot \xi(u)) = \mathcal{C}\left((m\otimes
Y)\cdot\sum^n_{i=1}X_1\cdots X_{i-1}\otimes X_{i+1}\cdots X_n\otimes
X_i\right)\\
&=&\sum^n_{i=1}(Y\cdot X_i)\mathcal{C}(m\otimes X_1\cdots X_{i-1}\otimes
X_{i+1}\cdots X_n)\\
&=&-\sum^n_{i=1}(X_i\cdot Y)(X_1\cdots X_{i-1}m X_{i+1}\cdots X_n)\\
&=&\sum^n_{i=1}X_1\cdots X_{i-1}\vartheta(m\otimes Y)(X_i)X_{i+1}\cdots
X_n = (\vartheta w)(u).
\end{eqnarray*}
This proves the lemma. 
\end{proof}

Hence we obtain
\begin{prop}
\label{prop:5-4-3}
Under the identification {\rm (\ref{eq:5-3-1})} and the map 
$\xi\colon \T \to H_1(\T,\Qomega; \T^r\hotimes\T^l)$,
the action of the Lie algebra $\agminus$ on the algebra $\T$
as derivations coincides with minus the intersection form
$$
-\mathcal{C}(\ \cdot\ )\colon H_1(\T; \T^c) \otimes H_1(\T,\Qomega;
\T^r\hotimes\T^l) \to \T^t=\T.
$$
In other words, we have 
$$
\mathcal{C}(w\cdot\xi(u)) = -w(u)
$$
for any $w \in H_1(\T; \T^c)$ and $u \in \T$. 
\end{prop}

\section{Comparison via a symplectic expansion}
In this section we prove Theorems \ref{thm:1-2-1} and \ref{thm:1-2-2}
in Introduction.

\subsection{Comparison via a Magnus expansion}
Let $F_n = \langle x_1, \dots, x_n\rangle$ be a free group 
of rank $n \geq 1$ with standard generators $x_1, \dots, x_n$, 
$\T$ the completed tensor algebra of the rational homology group 
$H_1(F_n; \mathbb{Q})$, and $\theta\colon F_n \to \T$ a Magnus expansion 
of $F_n$ as in Definition 2.3.1. Then $\theta$ induces an algebra 
homomorphism $\theta\colon \mathbb{Q}F_n\to \T$. We regard 
a left $\T$-module $M$ as a left $\mathbb{Q}F_n$-module via $\theta$. 
\begin{lem}
\label{lem:6-1-1}
For any right $\T$-module $M_1$ and left $\T$-module $M_2$, $\theta$
induces isomorphisms
\begin{eqnarray*}
&&\theta_*\colon H_*(F_n; M_1) \overset\cong\to
{\rm Tor}_*^{\T}(M_1,\mathbb{Q}), \quad\mbox{and}\\
&&\theta^*\colon {\rm Ext}^*_{\T}(\mathbb{Q},M_2)
\overset\cong\to H^*(F_n; M_2).
\end{eqnarray*}
In particular, if $\theta$ is group-like, then we have isomorphisms
$\theta_*\colon H_*(F_n; M_1) \stackrel{\cong}{\rightarrow}
H_*(\T;\mathbb{Q})$ and
$\theta^*\colon H^*(\T;M_2) \overset\cong\to H^*(F_n;M_2)$.
\end{lem}
\begin{proof}
There exists a filter-preserving automorphism $U$ of the algebra $\T$, 
such that $\theta(x_i) = U(1+[x_i])$ for any $1\leq i \leq n$
(see \cite{Ka}, Theorem 1.3). 
Since $\{[x_i]\}^n_{i=1} \subset H_1(F_n;\mathbb{Q})$
is a free basis of the left $\T$-module $\T_1$, the set 
$\{\theta(x_i)-1\}^n_{i=1}$ is also a free basis of $\T_1$. 
Hence we have a decomposition $M_1\otimes_{\T}\T_1 =
\bigoplus^n_{i=1}M_1\otimes (\theta(x_i)-1)$. On the other hand, 
by using Fox' free differential, we find out 
$\{x_i-1\}^n_{i=1}$ is a free basis of the left
$\mathbb{Q}\pi$-module $IF_n$. This implies a decomposition
$M_1\otimes_{\mathbb{Q}F_n}IF_n = 
\bigoplus^n_{i=1}M_1\otimes (x_i-1)$. Hence we obtain an isomorphism 
of chain complexes $\theta_*\colon D_*(\mathbb{Q}F_n;M_1) \cong 
M_1\otimes_{\T} P_*(\T)$, and so the isomorphism $\theta_*\colon H_*(F_n; M_1)
\cong {\rm Tor}_*^{\T}(M_1,\mathbb{Q})$.
A similar argument holds for
${\rm Ext}^*_{\T}(\mathbb{Q},M_2)$ and $H^*(F_n; M_2)$. 
\end{proof}

Let $\theta\colon \pi\to \T$ be a symplectic expansion
of the fundamental group $\pi$ of the surface $\Sigma$.
Then the restriction of $\theta$ to the subgroup $\langle\zeta\rangle$ 
is a Magnus expansion of the infinite cyclic group
$\langle\zeta\rangle$. Hence, from Lemma \ref{lem:6-1-1}
and the five-lemma, we obtain
\begin{cor}
\label{cor:6-1-2}
Let $\theta$ be a symplectic expansion of the fundamental
group $\pi$ of the surface $\Sigma$. Then
the algebra homomorphism $\theta$ induces an isomorphism
$$
\theta_*\colon H_*(\pi,\langle\zeta\rangle; M) 
\overset\cong\to H_*(\T,\Qomega;M)
$$
for any left $\T$-module $M$. Here we write simply $\theta_*$ for 
$(\theta,\theta)_*$.
\end{cor}

\subsection{Symplectic expansion}
Hereafter suppose $\theta$ is a symplectic expansion of the fundamental
group $\pi$. Then we have a commutative diagram of (complete) Hopf
algebras
\begin{equation}
\begin{CD}
\Qzeta@>{\theta}>>\Qomega\\
@V{i}VV@V{i}VV\\
\mathbb{Q}\pi @>{\theta}>>\T.
\end{CD}
\label{eq:6-2-1}
\end{equation}
As was proved in Lemma \ref{lem:6-1-1} and Corollary
\ref{cor:6-1-2}, we have isomorphisms
\begin{eqnarray*}
&&\theta_*\colon H_*(\pi; M) \overset\cong\to H_*(\T; M), \\
&&\theta^*\colon H^*(\T; M) \overset\cong\to H^*(\pi; M),\quad\mbox{and}\\
&&\theta_*\colon H_*(\pi,\langle\zeta\rangle; M) 
\overset\cong\to H_*(\T,\Qomega;M)
\end{eqnarray*}
for any left $\T$-module $M$. 
Now we have 
\begin{lem}
\label{lem:6-2-1}
$$
\theta_*[\Sigma] = \LfL \in H_2(\T,\Qomega;\mathbb{Q}).
$$
\end{lem}
\begin{proof} In the commutative diagram
$$
\begin{CD}
H_2(\Sigma, \partial\Sigma;\mathbb{Q}) @=
H_2(\pi,\langle\zeta\rangle;\mathbb{Q})
@>{\theta_*}>>H_2(\T,\Qomega;\mathbb{Q})\\
@V{\partial_*}VV@V{\partial_*}VV@V{\partial_*}VV\\
H_1(\partial\Sigma;\mathbb{Q}) @=
H_1(\langle\zeta\rangle;\mathbb{Q})
@>{\theta_*}>>H_1(\Qomega;\mathbb{Q}),
\end{CD}
$$
the fundamental class $[\Sigma]$ is mapped to 
$-[\zeta] \in H_1(\langle\zeta\rangle;\mathbb{Q})$. 
In fact, the loop $\zeta$ goes around the boundary 
$\partial\Sigma$ in the opposite direction. Since 
$\theta$ is a symplectic expansion, we have 
$-\theta_*[\zeta] = -\omega = \partial_*\LfL \in 
H_1(\Qomega; \mathbb{Q})$. This implies $\theta_*[\Sigma] 
= \LfL$, as was to be shown.
\end{proof}
Hence, by Proposition \ref{prop:4-3-2}, 
\begin{cor}
\label{cor:6-2-2} 
We have a commutative diagram
$$
\begin{CD}
H^1(\pi;M) @<{\theta^*}<< H^1(\T;M)\\
@V{[\Sigma]\cap}VV @V{\LfL\cap}VV\\
H_1(\pi,\langle\zeta\rangle;M) @>{\theta_*}>> H_1(\T,\Qomega;M)
\end{CD}
$$
for any left $\T$-module $M$.
\end{cor}
Here it should be remarked 
the cap product on the pair of spaces $(\Sigma,\partial\Sigma)$ 
coincides with that on the pair of Hopf algebras $(\mathbb{Q}\pi,
\Qzeta)$ from what we proved in \S4.5. 
Thus the intersection form on the pair $(\T,\Qomega)$ is directly 
related to that on the surface $\Sigma$. 
\begin{prop}
\label{prop:6-2-3}
For any left $\T$-modules $M_1$ and $M_2$, we have a commutative 
diagram
$$
\begin{CD}
H_1(\pi;M_1)\otimes H_1(\pi,\langle\zeta\rangle;M_2) 
@>{(\ \cdot\ )}>> M_1\otimes_{\mathbb{Q}\pi}M_2\\
@V{\theta_*\otimes\theta_*}VV @VVV\\
H_1(\T;M_1)\otimes H_1(\T,\Qomega;M_2) 
@>{(\ \cdot\ )}>> M_1\otimes_{\T}M_2.
\end{CD}
$$
\end{prop}

\subsection{``Completion" of the Goldman Lie algebra}
Recall the map $\lambda\colon \mathbb{Q}\hat{\pi} \to H_1(\pi;
\mathbb{Q}\pi^c)$ in \S3.4, whose kernel is the subspace 
$\mathbb{Q}1$ spanned by the constant loop $1 = \vert 1\vert \in 
\hat{\pi}$. Since the group $\pi$ is free, the map
$$
H_*(\theta)\colon H_1(\pi; \mathbb{Q}\pi^c)\to H_1(\pi;\T^c)
$$
induced by the injection $\theta\colon \mathbb{Q}\pi^c \to \T$ is injective. 
As was proved in Lemma \ref{lem:6-1-1}, 
$\theta_*\colon H_1(\pi;\T^c) \to H_1(\T; \T^c)$ is an isomorphism. 
Let $\lambda_{\theta}$ be the composite
$$
\lambda_{\theta}\colon \mathbb{Q}\hat{\pi}\overset\lambda\to 
H_1(\pi; \mathbb{Q}\pi^c)\overset{H_*(\theta)}\to H_1(\pi;\T)
\overset{\theta_*}\to H_1(\T;\T^c) = N(\T_1) = \agminus.
$$
From what we showed above, the kernel of $\lambda_{\theta}$ is the 
subspace $\mathbb{Q}1$.\par
In order to describe the map $\lambda_{\theta}$ explicitly, 
we introduce some notation around the algebra $\T$. 
Let $\hat{N}\colon \T\to\T_1$ be the map defined by 
$\hat{N}\vert_{H^{\otimes 0}} = 0$ and
$\hat{N}\vert_{H^{\otimes n}} = \frac{1}{n}N\vert_{H^{\otimes n}} 
= \sum^{n-1}_{i=0}\frac{1}{n}\nu^i$ for $n \geq 1$. Clearly we have 
$\hat{N}\vert_{N(\T_1)} = 1_{N(\T_1)}$. We denote by $\chi$ 
the composite
$$
\chi\colon \T^c\otimes_{\T}\T_1 = \T^c\otimes H \hookrightarrow \T_1.
$$
We have $\chi(\Ker(\T^c\otimes_{\T}\T_1\overset{1\otimes\partial}\to 
\T^c\otimes_{\T}\T)) = N(\T_1)$. Let $\Phi\colon \T \to \LL$ be the map 
defined by $\Phi(X_1\cdots X_n) = [X_1,[\cdots[X_{n-1}, X_n]\cdots]]$ 
for $X_i \in H^{\otimes n}$, $n \geq 1$. We have $\Phi(u) = nu$ and 
$[u, \Phi(v)] = \Phi(uv)$ for any $u \in \LL\cap H^{\otimes n}$ and 
$v \in \T_1$. See \cite{Ser} Part I, Theorem 8.1, p.28.
$\frac{1}{n}\Phi\vert_{H^{\otimes n}}$ is exactly the 
Dynkin idempotent.
\begin{lem}
\label{lem:6-3-1}
We have
$$
\hat{N}\chi(u\otimes v) = \hat{N}(u\Phi(v))
$$
for any $u \in \T$ and $v \in \T_1$.
\end{lem}
\begin{proof} It suffices to prove the lemma for $v \in H^{\otimes q}$
by induction on $q \geq 1$. If $q=1$, then $\hat{N}\chi(u\otimes v) 
= \hat{N}(uv) = \hat{N}(u\Phi(v))$. Suppose $q \geq 2$ and
$v\in H^{\otimes q-1}$. 
For any $X\in H$ we have $\hat{N}\chi(u\otimes Xv) 
= \hat{N}\chi([u,X]\otimes v) = \hat{N}([u,X]\Phi(v)) =
\hat{N}(u[X,\Phi(v)]) = \hat{N}(u\Phi(Xv))$. This completes the
induction. 
\end{proof}

\begin{lem}
\label{lem:6-3-2}
For any $x \in \pi$, we have 
$$
\lambda_{\theta}(x) = N\theta(x) = N(\theta(x)-1) \in N(\T_1) = \agminus.
$$
\end{lem}
\begin{proof}
The homology class $\lambda(x) = x\otimes [x] \in
H_1(\pi;\mathbb{Q}\pi^c)$ is represented by $x\otimes(x-1) 
\in \mathbb{Q}\pi^c\otimes_{\mathbb{Q}\pi}I\pi = D_1(\mathbb{Q}\pi;
\mathbb{Q}\pi^c)$. Hence $\lambda_{\theta}(x) =
\chi\theta_*H_*(\theta)(x\otimes(x-1))
=\chi\theta_*(\theta(x)\otimes(x-1)) =
\chi(\theta(x)\otimes\theta(x-1))$.   
Since $[\ell^{\theta}(x), \theta(x)] = 0$,
we have
$\theta(x)\otimes\theta(x-1) =
\sum^\infty_{k=1}\frac{1}{k!}\theta(x)\otimes \ell^{\theta}(x)^k =
\theta(x)\otimes\ell^{\theta}(x) \in \T^c\otimes_{\T}\T_1$. \par
Clearly we have $N\ell^{\theta}(x) = [x] = \hat{N}\Phi\ell^{\theta}(x)$.
We denote by $\ell^\theta_p(x) \in \mathcal{L}_p = 
\widehat{\mathcal{L}}\cap H^{\otimes p}$ the degree $p$-part 
of $\ell^\theta(x) \in \widehat{\mathcal{L}}$.
For $n \geq 2$, we have
\begin{eqnarray*}
&& n\hat{N}(\ell^{\theta}(x)^{n-1}\Phi\ell^{\theta}(x))\\
&=& \sum^n_{i=1}\sum_{p_1,\cdots,p_n}\hat{N}(\ell^{\theta}_{p_{i+1}}(x)\cdots
\ell^{\theta}_{p_n}(x)\ell^{\theta}_{p_1}(x)
\cdots\ell^{\theta}_{p_{i-1}}(x)\Phi\ell^{\theta}_{p_i}(x))\\
&=&
\sum^n_{i=1}\sum_{p_1,\cdots,p_n}\frac{p_i}{p_1+\cdots+p_n}
N(\ell^{\theta}_{p_{i+1}}(x)\cdots
\ell^{\theta}_{p_n}(x)\ell^{\theta}_{p_1}(x)
\cdots\ell^{\theta}_{p_{i-1}}(x)\ell^{\theta}_{p_i}(x))\\
&=&
\sum^n_{i=1}\sum_{p_1,\cdots,p_n}\frac{p_i}{p_1+\cdots+p_n}
N(\ell^{\theta}_{p_1}(x)\cdots\ell^{\theta}_{p_{n}}(x)) = \sum_{p_1,\cdots,p_n}
N(\ell^{\theta}_{p_1}(x)\cdots\ell^{\theta}_{p_{n}}(x))\\
&=& N(\ell^{\theta}(x)^n). 
\end{eqnarray*}
Hence, by Lemma \ref{lem:6-3-1}, we have 
\begin{eqnarray*}
&& \lambda_{\theta}(x) = \chi(\theta(x)\otimes\ell^{\theta}(x)) =
\hat{N}\chi(\theta(x)\otimes\ell^{\theta}(x))\\
&=& \hat{N}(\Phi\ell^{\theta}(x)) + 
\sum^\infty_{k=1}\frac{1}{(k+1)!}(k+1)
\hat{N}(\ell^{\theta}(x)^k\Phi\ell^{\theta}(x))\\
&=& N\ell^{\theta}(x) + 
\sum^\infty_{k=1}\frac{1}{(k+1)!}N(\ell^{\theta}(x)^{k+1}) = N(\theta(x)-1).
\end{eqnarray*}
This proves the lemma.
\end{proof} 
With respect to the $\T_1$-adic topology, the image of the map 
$\theta\colon \mathbb{Q}\pi \to \T$ is dense in the space $\T$. 
Clearly the map $N\colon \T \to N(\T_1)$ is a continuous surjection. 
Hence the image of the map $\lambda_{\theta}\colon \mathbb{Q}\hat{\pi}
\to N(\T_1)$ is dense in $N(\T_1)$. 
Summing up Propositions \ref{prop:3-4-3} (2), \ref{prop:5-3-2}, 
\ref{prop:6-2-3} and Lemma \ref{lem:6-3-2}, 
we have a commutative diagram
$$
\begin{CD}
(\mathbb{Q}\hat{\pi})^{\otimes 2} @>{\lambda^{\otimes 2}}>>
(H_1(\pi; \mathbb{Q}\pi^c))^{\otimes 2} @>{(\theta_*\circ
H_*(\theta))^{\otimes 2}}>>
(H_1(\T; \T^c))^{\otimes 2} @= (\agminus)^{\otimes 2}\\
@V{[\ ,\ ]}VV @V{\mathcal{B}_*(\ \cdot\ )}VV@V{\mathcal{B}(\ \cdot\
)}VV@V{-[\ ,\ ]}VV\\
\mathbb{Q}\hat{\pi} @= \mathbb{Q}\hat{\pi} @>{\lambda_{\theta}}>>
N(\T_1) @= \agminus.
\end{CD}
$$
This means the map $-\lambda_{\theta}\colon \mathbb{Q}\hat{\pi} 
\to N(\T_1)= \agminus$ is a Lie algebra homomorphism. 
Hence we obtain
\begin{thm}
\label{thm:6-3-3}
For any symplectic expansion 
$\theta$ of the fundamental group $\pi$ of the surface $\Sigma$, 
the map
$$
-N\theta\colon \mathbb{Q}\hat{\pi}\to N(\T_1) = \agminus, \quad
x \mapsto -N\theta(x)
$$
is a Lie algebra homomorphism. The kernel is the subspace $\mathbb{Q}1$ 
spanned by the constant loop $1$, and the image is dense in $N(\T_1) =
\agminus$ with respect to the $\T_1$-adic topology. 
\end{thm}
By this theorem, we may regard the formal symplectic geometry 
$\agminus$ as a completion of the Goldman Lie algebra
$\mathbb{Q}\hat{\pi}'$. In our forthcoming paper \cite{KK} we use this idea 
to compute the center of the Goldman Lie algebra of an oriented surface 
of infinite genus. 

\subsection{Geometric interpretation of symplectic
derivations}

In this subsection we show the action of $\agminus$ 
on the algebra $\T$ as symplectic derivations can be
interpreted as the action $\sigma$ 
of the Goldman Lie algebra $\mathbb{Q}\hat{\pi}$ on the group ring 
$\mathbb{Q}\pi$ in a geometric way. 
In order to prove it, we need to prepare some lemmas.
\begin{lem}
\label{lem:6-4-1}
If $u \in \T$ is group-like, namely, $u$ satisfies $\Delta(u) 
= u\hat{\otimes} u$, then we have
$$
\xi(u) = (1\hat{\otimes} u)\otimes (u-1) 
\in (\T^r\hotimes\T^l)\otimes_{\T}\T_1. 
$$
\end{lem}
\begin{proof}
When $u$ is given by $u=\sum^\infty_{k=0}u_k$, $u_k \in H^{\otimes k}$, 
we denote $u_{\leq m} := \sum^m_{k=0}u_k \in T$ for any $m \geq 1$. 
Then we have $\xi(u)\equiv \xi(u_{\leq m}) \equiv 
1\otimes u_{\leq m}\otimes (u_{\leq m}-1) \equiv (1\hat{\otimes} u)\otimes (u-1)$
modulo the elements $\in (\T^r\hotimes\T^l)\otimes_{\T}\T_1$ whose degree 
are greater than $m$. Since we can choose $m$ arbitrarily, we obtain 
$\xi(u) = (1\hat{\otimes} u)\otimes(u-1)$. This proves the lemma.
\end{proof}
\begin{lem}
\label{lem:6-4-2}
We have a commutative diagram
$$
\begin{CD}
\mathbb{Q}\pi @>{\xi}>> H_1(\pi,\langle\zeta\rangle;
\mathbb{Q}\pi^r\otimes\mathbb{Q}\pi^l)\\
@V{\theta}VV@V{\theta_*\circ H_*(\theta)}VV\\
\T @>{\xi}>> H_1(\T,\Qomega; \T^r\hotimes\T^l).
\end{CD}
$$
Here $H_*(\theta)\colon  H_1(\pi,\langle\zeta\rangle;
\mathbb{Q}\pi^r\otimes\mathbb{Q}\pi^l) \to 
 H_1(\pi,\langle\zeta\rangle;\T^r\hotimes\T^l)$ is the map 
induced by $\theta\colon
\mathbb{Q}\pi^r\otimes\mathbb{Q}\pi^l\to \T^r\hotimes\T^l$. 
\end{lem}
\begin{proof} For any $x \in \pi$,  $\theta(x)$ is group-like, 
so that, from Lemma \ref{lem:6-4-1}, $\xi\theta(x) =
(1\hat{\otimes}\theta(x))\otimes(\theta(x)-1) =
\theta_*((1\hat{\otimes}\theta(x))\otimes[x]) = \theta_*H_*(\theta) ((1\otimes
x)\otimes[x]) = \theta_*H_*(\theta)\xi(x)$. This proves the lemma.
\end{proof}

The main result in this subsection is
\begin{thm}
\label{thm:6-4-3}
Let $\theta$ be a symplectic expansion of the fundamental group 
$\pi$ of the surface $\Sigma$. Then, 
for $u \in \mathbb{Q}\hat{\pi}$ and $v \in \mathbb{Q}\pi$, 
we have the equality
$$
\theta(\sigma(u)v) = -\lambda_{\theta}(u)\theta(v).
$$
Here the right hand side means minus the action of
$\lambda_{\theta}(u) \in \agminus$ on the tensor $\theta(v) \in \T$
as a derivation. In other words, the diagram
\begin{equation}
\label{eq:6-4-1}
\begin{CD}
\mathbb{Q}\hat{\pi} \times \mathbb{Q}\pi @>{\sigma}>>
\mathbb{Q}\pi \\
@V{-\lambda_{\theta}\times \theta}VV @VV{\theta}V \\
\mathfrak{a}_g^{-} \times \widehat{T} @>>> \widehat{T},
\end{CD}
\end{equation}
where the bottom horizontal arrow means the derivation,
commutes.
\end{thm}
\begin{proof}
From the definition of the two $\mathcal{C}$'s 
we have a commutative diagram
\begin{equation}
\begin{CD}
\widehat{\mathbb{Q}\pi}^c\otimes_{\mathbb{Q}\pi}(\mathbb{Q}\pi^r\otimes
\mathbb{Q}\pi^l) @>{\mathcal{C}}>> \widehat{\mathbb{Q}\pi}^t\\
@V{\theta\otimes\theta\otimes\theta}VV @V{\theta}VV\\
\T^c\otimes_{\T}(\T^r\hotimes\T^l) @>{\mathcal{C}}>>\T^t.
\end{CD}
\label{eq:6-4-2}
\end{equation}
By Propositions \ref{prop:3-5-2}, \ref{prop:6-2-3} and Lemma \ref{lem:6-4-2}, 
we have
$$
\theta(\sigma(u)v) = \theta\mathcal{C}_*(\lambda(u)\cdot\xi(v)) 
= \mathcal{C}((\theta_*H_*(\theta)\lambda(u))\cdot
(\theta_*H_*(\theta)\xi(v))) = \mathcal{C}(\lambda_{\theta}(u)\cdot
\xi\theta(v)),
$$
which equals $-\lambda_{\theta}(u)\theta(v)$ from Proposition
\ref{prop:5-4-3}. Hence we obtain $\theta(\sigma(u)v) = 
-\lambda_{\theta}(u)\theta(v)$. 
This completes the proof of the theorem.
\end{proof}

\subsection{The key formula}
Recall from \S3 the map $|\cdot |\colon \mathbb{Q}\pi
\rightarrow \mathbb{Q}\hat{\pi}$ and
define $\sigma\colon \mathbb{Q}\pi \times \mathbb{Q}\pi
\rightarrow \mathbb{Q}\pi$ by
$\sigma(u,v)=\sigma(|u|)v$.

\begin{lem}
\label{lem:6-5-1}
For integers $p,q \ge 1$, we have
$$\sigma(I\pi^p \times I\pi^q)\subset I\pi^{p+q-2}.$$
\end{lem}
\begin{proof}
Since $\theta^{-1}(\widehat{T}_p)=I\pi^p$,
it suffices to show the following: if $u\in I\pi^p$ and $v\in I\pi^q$,
then $\theta(\sigma(u,v))\in \widehat{T}_{p+q-2}$.
By Lemma \ref{lem:6-3-2} and Theorem \ref{thm:6-4-3},
$\theta(\sigma(u,v))=-\lambda_{\theta}(u)\theta(v)=
(N\theta(u))\theta(v)$. On the other hand, we have
$N\theta(u)\in \widehat{T}_p$ and $\theta(v)\in \widehat{T}_q$,
by assumption. Hence $(N\theta(u))\theta(v)\in \widehat{T}_{p+q-2}$.
\end{proof}

By this lemma, we see that $\sigma$ naturally extends to
$\sigma\colon \widehat{\mathbb{Q}\pi}\times \widehat{\mathbb{Q}\pi}
\rightarrow \widehat{\mathbb{Q}\pi}$ and the diagram
\begin{equation}
\label{eq:6-5-1}
\begin{CD}
\widehat{\mathbb{Q}\pi} \times \widehat{\mathbb{Q}\pi}
@>{\sigma}>> \widehat{\mathbb{Q}\pi} \\
@V{-\lambda_{\theta}\times \theta}VV @VV{\theta}V \\
\mathfrak{a}_g^{-} \times \widehat{T} @>>> \widehat{T},
\end{CD}
\end{equation}
which is an extension of the diagram (\ref{eq:6-4-1}),
commutes. Let $f(x)$ be a power series in $x-1$.
Then for $x\in \pi$, $N(\theta(f(x))\in \mathfrak{a}_g^{-}=N(\widehat{T}_1)$
is defined. For example, if $f(x)=\log x$, then
$N(\theta(f(x)))=N(\ell^{\theta}(x))=[x]$, as we have seen
in the proof of Lemma \ref{lem:6-3-2}.

Let $f(x)=(\log x)^2$. Then
$N(\theta(f(x)))=N(\ell^{\theta}(x)\ell^{\theta}(x))=2L^{\theta}(x)$.
Therefore, from (\ref{eq:6-5-1}) we obtain the following key formula
which will derive Theorem \ref{thm:1-1-1}.

\begin{thm}
\label{thm:6-5-2}
For $x,y\in \pi$,
$$\theta(\sigma((\log x)^2)y)=-2L^{\theta}(x)\theta(y).$$
\end{thm}

As an immediate consequence, we have the following:

\begin{cor}
\label{cor:6-5-3}
Let $\alpha$ be a free loop and $\beta$ a based loop
on $\Sigma$. Suppose $\alpha\cap \beta =\emptyset$.
Then $L^{\theta}(\alpha)\theta(\beta)
=L^{\theta}(\alpha)\ell^{\theta}(\beta)=0$.
\end{cor}

\begin{proof}
By assumption, $\sigma(\alpha^n)\beta=0$ for each $n\ge 0$,
hence $\sigma((\log \alpha)^2)\beta=0$.
Using Theorem \ref{thm:6-5-2}, we have
$L^{\theta}(\alpha)\theta(\beta)=0$.
Since $\ell^{\theta}(\beta)=\log \theta(\beta)$ and
$L^{\theta}(\alpha)$ is a derivation, we also have
$L^{\theta}(\alpha)\ell^{\theta}(\beta)=0$.
\end{proof}

\section{Proof of the main results}
In this section we prove Theorems \ref{thm:1-1-1} and
\ref{thm:1-1-2} in Introduction, and derive some formulas
of $\tau_k^{\theta}(t_C)$, which matches the computations
by Morita.

Let us recall some notations. As in \S6.3, we  denote by
$\ell^\theta_p(x) \in \mathcal{L}_p = 
\widehat{\mathcal{L}}\cap H^{\otimes p}$ the degree $p$-part 
of $\ell^\theta(x) \in \widehat{\mathcal{L}}$ for $x \in \pi$. 
Further we denote 
$$
L^{\theta}(x) = \sum_{i=2}^\infty L^\theta_i(x), \quad 
L^{\theta}_i(x) \in H^{\otimes i}.
$$
Then we have 
$$
L^{\theta}_i(x) = \frac{1}{2}N\left( \sum_{p=1}^{i-1} 
\ell^\theta_p(x)\ell^\theta_{i-p}(x)\right).
$$
$L^{\theta}(x)$ and $L^{\theta}_i(x)$ are regarded as
a derivation of the algebra $\widehat{T}$, and if $\theta$ is
group-like, they belong to $\mathfrak{l}_g$ (see \S2.7).

\subsection{The logarithms of Dehn twists}

\begin{thm}
\label{thm:7-1-1}
Let $\theta$ be a symplectic expansion and $C$ a simple closed curve
on $\Sigma$. Then the total Johnson map $T^{\theta}(t_C)$ is described as
$$T^{\theta}(t_C)=e^{-L^{\theta}(C)}.$$
Here, the right hand side is the algebra automorphism of $\widehat{T}$
defined by the exponential of the derivation $-L^{\theta}(C)$.
\end{thm}

Let us give an orientation on $C$ and denote by $[C]\in H$
its homology class. Then the square of $L_2^{\theta}(C)=[C][C]$
acts on $H$ trivially. See Lemma \ref{lem:7-6-1} and Proposition \ref{prop:7-7-1}.
Recall from \S2.5 that
$T^{\theta}(t_C)=\tau^{\theta}(t_C)\circ |t_C|$.
Let $X\in H$. Modulo $\widehat{T}_2$, we compute
$$|t_C|X\equiv \tau^{\theta}(t_C)\circ |t_C|X=T^{\theta}(t_C)X
=e^{-L^{\theta}(C)}X \equiv X-L^{\theta}_2(C)X
=X-(X\cdot [C])[C].$$
Namely,
$$|t_C|X=X-(X\cdot [C])[C],\ X\in H.$$
This is nothing but the classical formula
as stated in Introduction (\ref{eq:1-1-2}).

A simple closed curve $C$ on $\Sigma$ is called
non-separating (resp. separating) if $\Sigma \setminus C$
is connected (resp. not connected).
The proof of Theorem \ref{thm:7-1-1} is divided
into two cases according to whether $C$ is separating or not.
We take symplectic generators suitable to $C$
and compute $L^{\theta}(C)\theta(x_i)$, hence
$e^{-L^{\theta}(C)}\theta(x_i)$ by using Theorem \ref{thm:6-5-2}
where $x_i$ is one of the generators.
Next we observe this value coincides with
$T^{\theta}(t_C)\theta(x_i)=\theta(t_C(x_i))$.
Together with the fact that $\{ \theta(x_i)\}_i$ generates
$\widehat{T}$ as a complete algebra, we will get the conclusion.

\subsection{Non-separating case}
Suppose $C$ is non-separating. We take
symplectic generators
$\alpha_1,\beta_1,\ldots,\alpha_g,\beta_g$
such that $|\alpha_1|$ is homotopic to $C$ as unoriented loops.
Then the action of $t_C$ on $\pi$ is given by
\begin{equation}
\label{eq:7-2-1}
\begin{cases}
t_C(\alpha_i)=\alpha_i, & 1\le i\le g, \\
t_C(\beta_1)=\beta_1\alpha_1, & \\
t_C(\beta_i)=\beta_i, & 2\le i\le g. \\
\end{cases}
\end{equation}

\begin{lem}
\label{lem:7-2-1}
Notations are as above. Then
$$\begin{cases}
L^{\theta}(C)\theta(\alpha_i)=0, & 1\le i\le g, \\
L^{\theta}(C)\theta(\beta_1)=-\theta(\beta_1)\ell^{\theta}(\alpha_1), &  \\
L^{\theta}(C)\theta(\beta_i)=0, & 2 \le i \le g. \\
\end{cases}$$
\end{lem}

\begin{proof}
Since $C\cap \alpha_i=\emptyset$ for $1\le i\le g$
and $C\cap \beta_i=\emptyset$ for $2\le i\le g$,
we have $L^{\theta}(C)\theta(\alpha_i)=0$
for $1\le i\le g$ and
$L^{\theta}(C)\theta(\beta_i)=0$
for $2 \le i \le g$ by Corollary \ref{cor:6-5-3}.

It remains to prove
$L^{\theta}(C)\theta(\beta_1)=
-\theta(\beta_1)\ell^{\theta}(\alpha_1)$. We have
$$\sigma(\alpha_1^n)\beta_1=n\beta_1\alpha_1^n,\ {\rm for\ }n\ge 0.$$
Thus for $m\ge 0$, we compute
\begin{eqnarray}
\label{eq:7-2-2}
\sigma((\alpha_1-1)^m)\beta_1
&=&
\sum_{n=0}^m(-1)^{m-n} {m\choose n}\sigma(\alpha_1^n)\beta_1
=\sum_{n=1}^m(-1)^{m-n}n{m\choose n}\beta_1\alpha_1^n \nonumber \\
&=& m\beta_1\alpha_1(\alpha_1-1)^{m-1}.
\end{eqnarray}
Here we use $n{m\choose n}=m{m-1 \choose n-1}$.
This implies if $f(\alpha_1)$ is a power series in
$\alpha_1-1$, then $\sigma(f(\alpha_1))\beta_1=
\beta_1\alpha_1f^{\prime}(\alpha_1)$, where
$f^{\prime}(\alpha_1)$ is the derivative of $f(\alpha_1)$.
If $f(\alpha_1)=(\log \alpha_1)^2$, then
$\alpha_1 f^{\prime}(\alpha_1)=2\log \alpha_1$.
Therefore, $\sigma((\log \alpha_1)^2\beta_1)=2\beta_1\log \alpha_1$.

Substituting this into Theorem \ref{thm:6-5-2}, we have
$$L^{\theta}(C)\theta(\beta_1)=
-\frac{1}{2}\theta(\sigma((\log \alpha_1)^2)\beta_1)
=-\theta(\beta_1\log \alpha_1)=-\theta(\beta_1)\ell^{\theta}(\alpha_1).$$
This completes the proof.
\end{proof}

\begin{proof}[Proof of Theorem \ref{thm:7-1-1} for non-separating $C$]
By Lemma \ref{lem:7-2-1}, we have
$e^{-L^{\theta}(C)}\theta(\alpha_i)=\theta(\alpha_i)$
for $1\le i\le g$, and
$e^{-L^{\theta}(C)}\theta(\beta_i)=\theta(\beta_i)$ for $2\le i\le g$.
Also Lemma \ref{lem:7-2-1} implies
$L^{\theta}(C)^i\theta(\beta_1)
=(-1)^i\theta(\beta_1)\ell^{\theta}(\alpha_1)^i$ for $i\ge 0$. 
Hence
$$e^{-L^{\theta}(C)}\theta(\beta_1)=
\sum_{i=0}^{\infty}(-1)^i \frac{1}{i!}L^{\theta}(C)^i\theta(\beta_1)
=\theta(\beta_1)\sum_{i=0}^{\infty}\frac{1}{i!}\ell^{\theta}(\alpha_1)^i
=\theta(\beta_1)\theta(\alpha_1)=\theta(\beta_1\alpha_1).$$

On the other hand, (\ref{eq:7-2-1}) implies that the total
Johnson map $T^{\theta}(t_C)$
satisfies $T^{\theta}(t_C)(\theta(\alpha_i))=\theta(\alpha_i)$
for $1\le i \le g$, $T^{\theta}(t_C)(\theta(\beta_1))=\theta(\beta_1\alpha_1)$,
and $T^{\theta}(t_C)(\theta(\beta_i))=\theta(\beta_i)$
for $2\le i\le g$.

In summary, the values of $e^{-L^{\theta}(C)}$ and $T^{\theta}(t_C)$
coincide on $\{ \theta(\alpha_i),\theta(\beta_i)\}_i$.
Since $\{ \theta(\alpha_i),\theta(\beta_i)\}_i$ generates $\widehat{T}$
as a complete algebra, this shows the equality
$e^{-L^{\theta}(C)}=T^{\theta}(t_C)\in {\rm Aut}(\widehat{T})$.
This completes the proof of Theorem \ref{thm:7-1-1}
for the case $C$ is non-separating.
\end{proof}

\subsection{Separating case}
Suppose $C$ is separating. We take symplectic generators
$\alpha_1,\beta_1,\ldots,\alpha_g,\beta_g$
such that $C$ is homotopic to $|\gamma_h|$ as unoriented loops,
where $\gamma_h=\prod_{i=1}^h [\alpha_i,\beta_i]$ for some $h$.
Then the action of $t_C$ on $\pi$ is given by
\begin{equation}
\label{eq:7-3-1}
\begin{cases}
t_C(\alpha_i)=\gamma_h^{-1}\alpha_i\gamma_h, & 1\le i\le h, \\
t_C(\alpha_i)=\alpha_i, & h+1\le i\le g, \\
t_C(\beta_i)=\gamma_h^{-1}\beta_i\gamma_h, & 1\le i\le h, \\
t_C(\beta_i)=\beta_i, & h+1\le i\le g. \\
\end{cases}
\end{equation}

\begin{lem}
\label{lem:7-3-1}
Notations are as above. Then
$$L^{\theta}(C)\theta(\alpha_i)=\begin{cases}
[\ell^{\theta}(\gamma_h),\theta(\alpha_i)],\ {\rm if\ } 1\le i\le h,\\
0,\ {\rm if\ } h+1\le i\le g,\ {\rm and}
\end{cases}$$
$$L^{\theta}(C)\theta(\beta_i)=\begin{cases}
[\ell^{\theta}(\gamma_h),\theta(\beta_i)],\ {\rm if\ } 1\le i\le h,\\
0,\ {\rm if\ } h+1\le i\le g.
\end{cases}$$
\end{lem}

\begin{proof}
Suppose $i \ge h+1$.
Since $C \cap \alpha_i=C \cap \beta_i=\emptyset$
if $i \ge h+1$, we have
$L^{\theta}(C)\theta(\alpha_i)
=L^{\theta}(C)\theta(\beta_i)=0$ by Corollary \ref{cor:6-5-3}.

Suppose $i\le h$. Then we have
$$\sigma(\gamma_h^n)\alpha_i
=-n\gamma_h^n\alpha_i+n\alpha_i\gamma_h^n,\ {\rm for\ } n\ge 0.$$
By a computation similar to (\ref{eq:7-2-2}), we get
$$\sigma((\gamma_h-1)^m)\alpha_i
=m\alpha_i\gamma_h(\gamma_h-1)^{m-1}
-m\gamma_h(\gamma_h-1)^{m-1}\alpha_i$$
for $m\ge 0$. This implies if $f(\gamma_h)$ is a power series
in $\gamma_h-1$, then $\sigma(f(\gamma_h))\alpha_i
=\alpha_i\gamma_hf^{\prime}(\gamma_h)
-\gamma_hf^{\prime}(\gamma_h)\alpha_i$.
Therefore, $\sigma((\log \gamma_h)^2)\alpha_i
=2(\alpha_i \log \gamma_h-(\log \gamma_h) \alpha_i)$.

Substituting this into Theorem \ref{thm:6-5-2}, we have
$$L^{\theta}(C)\theta(\alpha_i)
=-\theta(\alpha_i \log \gamma_h-(\log \gamma_h) \alpha_i)
=[\ell^{\theta}(\gamma_h), \theta(\alpha_i)].$$
The proof of $L^{\theta}(C)\theta(\beta_i)
=[\ell^{\theta}(\gamma_h), \theta(\beta_i)]$
is similar. This completes the proof.
\end{proof}

\begin{proof}[Proof of Theorem \ref{thm:7-1-1} for separating $C$]
By Lemma \ref{lem:7-3-1}, we have
$e^{-L^{\theta}(C)}\theta(\alpha_i)=\theta(\alpha_i)$ and
$e^{-L^{\theta}(C)}\theta(\beta_i)=\theta(\beta_i)$ if
$i\ge h+1$.
Suppose $i\le h$.
By Corollary \ref{cor:6-5-3}, we have $L^{\theta}(C)\ell^{\theta}(\gamma_h)=0$.
Combining this with Lemma \ref{lem:7-3-1}, we have
$L^{\theta}(C)^m\theta(\alpha_i)
={\rm ad}(\ell^{\theta}(\gamma_h))^m\theta(\alpha_i)$
for $m\ge 0$. Hence we have
$$e^{-L^{\theta}(C)}\theta(\alpha_i)=
\sum_{m=0}^{\infty}\frac{1}{m!}
{\rm ad}(-\ell^{\theta}(\gamma_h))^m\theta(\alpha_i)
=e^{-\ell^{\theta}(\gamma_h)}\theta(\alpha_i)
e^{\ell^{\theta}(\gamma_h)}
=\theta(\gamma_h^{-1}\alpha_i\gamma_h).$$
Similarly we have $e^{-L^{\theta}(C)}\theta(\beta_i)
=\theta(\gamma_h^{-1}\beta_i\gamma_h)$ for $i\le h$.

On the other hand (\ref{eq:7-3-1})
implies that $T^{\theta}(t_C)\theta(\alpha_i)
=\theta(\gamma_h^{-1}\alpha_i\gamma_h)$,
$T^{\theta}(t_C)\theta(\beta_i)
=\theta(\gamma_h^{-1}\beta_i\gamma_h)$ for $1\le i\le h$,
and $T^{\theta}(t_C)\theta(\alpha_i)=\theta(\alpha_i)$,
$T^{\theta}(t_C)\theta(\beta_i)=\theta(\beta_i)$ for $h+1\le i\le g$.

In summary the values of
$e^{-L^{\theta}(C)}$ and $T^{\theta}(t_C)$
coincide on $\{ \theta(\alpha_i),\theta(\beta_i) \}_i$.
As the proof for non-separating $C$, this leads to
the equality $e^{-L^{\theta}(C)}=T^{\theta}(t_C)$.
This completes the proof of Theorem \ref{thm:7-1-1} for the case $C$ is separating.
\end{proof}

\subsection{Action on the nilpotent quotients}
Let $\Gamma_k=\Gamma_k(\pi)$, $k\ge 1$ be the lower
central series of $\pi$. Namely $\Gamma_1=\pi$,
and define $\Gamma_k$ successively by
$\Gamma_k=[\Gamma_{k-1},\pi]$ for $k\ge 2$.
For $k\ge 0$, the $k$-th nilpotent quotient of $\pi$
is defined as the quotient group
$N_k=N_k(\pi)=\pi/\Gamma_{k+1}$. Note that $N_1=\pi/[\pi,\pi]$
is nothing but the abelianization of $\pi$. Since any automorphism
of $\pi$ preserves $\Gamma_k$, the mapping class group
$\mathcal{M}_{g,1}$ naturally acts on $N_k$ for each $k$.

Let $\theta$ be a (not necessary symplectic)
Magnus expansion of $\pi$.
For each $k\ge 1$ we have
\begin{equation}
\label{eq:7-4-1}
\theta^{-1}(1+\widehat{T}_k)=\Gamma_k.
\end{equation}
See Bourbaki \cite{Bou} ch.2, \S5, no.4, Theorem 2.
Therefore, $\theta$ induces an injective homomorphism
$\theta\colon N_k \rightarrow (1+\widehat{T}_1)/(1+\widehat{T}_{k+1})$.
Note that $1+\widehat{T}_{k+1}$ is a normal subgroup of $1+\widehat{T}_1$.
By post-composing the natural injection
$(1+\widehat{T}_1)/(1+\widehat{T}_{k+1})\hookrightarrow
\widehat{T}/\widehat{T}_{k+1}$, we get an injection
\begin{equation}
\label{eq:7-4-2}
\theta\colon N_k \rightarrow \widehat{T}/\widehat{T}_{k+1}.
\end{equation}
Since the total Johnson map $T^{\theta}(\varphi)$
of $\varphi\in \mathcal{M}_{g,1}$ is filter-preserving,
it naturally induces a filter-preserving automorphism of
the quotient algebra $\widehat{T}/\widehat{T}_{k+1}$.
Using the same letter we denote it by $T^{\theta}(\varphi)$.
By construction the injection (\ref{eq:7-4-2}) is compatible
with the action of $\mathcal{M}_{g,1}$: we have
$T^{\theta}(\varphi)\circ \theta(x)=\theta\circ \varphi(x)$
for any $x\in N_k$.

For a group $G$, let $\bar{G}$ be the quotient set of $G$
by conjugation and the relation $g\sim g^{-1}$, $g\in G$.
Let $C$ be a simple closed curve on $\Sigma$.
Choose any $x\in \pi$ such that $x$ is freely
homotopic to $C$ as unoriented loops. Then the element of $\bar{\pi}$
represented by $x$ is independent of the choice of $x$.
For each $k\ge 0$, let $\bar{C}_k\in \bar{N}_k$ be
the image of this element under the natural surjection
$\bar{\pi}\rightarrow \bar{N}_k$.

\begin{thm}
\label{thm:7-4-1}
For each $k\ge 1$, the action of $t_C$ on $N_k$
depends only on the class $\bar{C}_k\in \bar{N}_k$.
If $C$ is separating, it depends only on the class
$\bar{C}_{k-1}\in \bar{N}_{k-1}$.
\end{thm}

\begin{proof}
Fix a symplectic expansion $\theta$.
By Theorem \ref{thm:7-1-1}, we have
$T^{\theta}(t_C)=e^{-L^{\theta}(C)}\in {\rm Aut}(\widehat{T})$.
Remark that the action of $e^{-L^{\theta}(C)}$ on $\widehat{T}/\widehat{T}_{k+1}$
depends only on $L^{\theta}_i(C)$, $2\le i\le k+1$.

Pick $x\in \pi$ such that $x$ is freely
homotopic to $C$ as unoriented loops.
Let $x^{\prime}\in \pi$ such that $x^{-1}x^{\prime}\in \Gamma_{k+1}$.
By (\ref{eq:7-4-1}), it follows that
$\ell^{\theta}_i(x)=\ell^{\theta}_i(x^{\prime})$ for $1\le i\le k$.
Since $L^{\theta}(C)=L^{\theta}(x)=\frac{1}{2}N(\ell^{\theta}(x)\ell^{\theta}(x))$,
this observation together with Lemma \ref{lem:2-6-4}
shows that $L^{\theta}_i(C)$, $2\le i\le k+1$,
depend only on the class $\bar{C}_k\in \bar{N}_k$.
This proves the first part.

If $C$ is separating, $x\in \Gamma_1$ hence $\ell^{\theta}_1(x)=0$.
Thus if $x^{\prime}\in \Gamma_1$ is a representative
of another separating simple closed curve $C^{\prime}$, satisfying
$x^{-1}x^{\prime}\in \Gamma_k$, then $L^{\theta}_i(x)=L^{\theta}_i(x^{\prime})$
for $2\le i\le k+1$.
Therefore, $L^{\theta}_i(C)$, $2\le i\le k+1$ depend
only on the class $\bar{C}_{k-1}
\in \bar{N}_{k-1}$. This completes the proof.
\end{proof}

This theorem is a generalization of the following well-known facts:
1) the action of $t_C$ on $N_1=H_1(\Sigma;\mathbb{Z})$ depends
only on the class $\pm [C]$;
2) if $C$ is separating, then $t_C$ belongs to the Johnson kernel
$\mathcal{K}_{g,1}=\mathcal{M}_{g,1}[2]$, the subgroup of the mapping classes
acting on $N_2$ as the identity.

\subsection{The formula of $\tau_k^{\theta}(t_C)$ for separating $C$}
In the rest of this section we derive formulas of the $k$-th Johnson map
(see Definition \ref{def:2-5-1}) of $t_C$ with associated to
a symplectic expansion
from Theorem \ref{thm:7-1-1}.
For simplicity, we often write $L^{\theta}(C)=L$, $L^{\theta}_k(C)=L_k$, etc.

In this subsection we treat the case of separating curves.

\begin{thm}
\label{thm:7-5-1}
Let $\theta$ be a symplectic expansion and
$C$ a separating simple closed curve on $\Sigma$. Then for $k\ge 1$,
the $k$-th Johnson map $\tau^{\theta}_k(t_C)$ is given by
$$\tau_k^{\theta}(t_C)=\sum_{1\le n\le [k/2]}
\frac{(-1)^n}{n!}\sum_{\substack{(m_1,\ldots,m_n), m_i\ge 4, \\
m_1+\cdots +m_n=2n+k}}L_{m_1}\cdots L_{m_n}.$$
For example, we have $\tau_1^{\theta}(t_C)=0$, and
\begin{eqnarray*}
\tau_2^{\theta}(t_C) &=& -L_4; \\
\tau_3^{\theta}(t_C) &=& -L_5; \\
\tau_4^{\theta}(t_C) &=& -L_6+\frac{1}{2}L_4^2; \\
\tau_5^{\theta}(t_C) &=& -L_7+\frac{1}{2}(L_4L_5+L_5L_4); \\
\tau_6^{\theta}(t_C) &=& -L_8+\frac{1}{2}(L_4L_6+L_5^2+L_6L_4)
-\frac{1}{6}L_4^3. \\
\end{eqnarray*}
Here, $L_4^2$ is the composition
$L_4\circ L_4\colon H\rightarrow H^{\otimes 3}
\rightarrow H^{\otimes 5}$, etc.
\end{thm}

\begin{proof}
Since $C$ is separating, $|t_C|={\rm id}$
hence $\tau^{\theta}(t_C)=T^{\theta}(t_C)$, and
$L_2^{\theta}(C)=L_3^{\theta}(C)=0$. Thus,
$L^{\theta}(C)=L_4+L_5+\cdots$.
For $X\in H$, the degree $k+1$-part of $L(C)^nX$ is equal to
$$\sum_{\substack{(m_1,\ldots,m_n), m_i\ge 4, \\
m_1+\cdots +m_n=2n+k}}L_{m_1}\cdots L_{m_n}X.$$
In particular if $n>[k/2]$, the degree $k+1$-part of
$L(C)^nX$ is zero. By Theorem \ref{thm:7-1-1}, the
conclusion follows.
\end{proof}

\begin{rem}{\rm 
In \cite{Mo} Proposition 1.1, Morita computed $\tau_2(t_C)$
for separating $C$, and our formula $\tau_2^{\theta}(t_C)=-L_4^{\theta}$
coincides with his formula.
In fact, we have $t_C\in \mathcal{K}_{g,1}$
as we remarked at the end of \S7.4, and
$\tau_2^{\theta}(t_C)$ does not depend on the choice
of $\theta$.
}
\end{rem}

\subsection{Computations of $L^{\theta}_k(x)$ for small $k$}
Compared with the separating case, the non-separating case is
more complicated since $L_2^{\theta}(C)\neq 0$ for non-separating $C$.
So far we don't have a complete formula of $\tau^{\theta}_k(t_C)$, $k\ge 1$
for non-separating $C$, and in this paper we only give formulas of
$\tau^{\theta}_1(t_C)$ and $\tau^{\theta}_2(t_C)$. Even in these cases,
we need considerable computations. This subsection is a preparation
for the computations.

Let $\Lambda^kH$ be the $k$-th exterior product of $H$. We can realize
$\Lambda^kH$ as a subspace of $H^{\otimes k}$ by the embedding
$$\Lambda^kH\rightarrow H^{\otimes k},\ 
X_1\wedge \cdots \wedge X_k \mapsto \sum_{\sigma\in \mathfrak{S}_k}
{\rm sign}(\sigma)X_{\sigma(1)}\otimes \cdots \otimes X_{\sigma(k)}.$$
Note that $\Lambda^2H=\mathcal{L}_2$ and $X\wedge Y=[X,Y]$.

\begin{lem}
\label{lem:7-6-1}
Let $\theta$ be a group-like expansion.
Then for each $x\in \pi$,
\begin{enumerate}
\item[{\rm (1)}]
$L^{\theta}_2(x)=[x][x]$,
\item[{\rm (2)}]
$L^{\theta}_3(x)=[x]\wedge \ell_2^{\theta}(x)\in \Lambda^3H$.
\end{enumerate}
\end{lem}

\begin{proof}
For simplicity, we write $\ell^{\theta}=\ell$.
Since $\ell_1(x)=[x]$, we have
$L^{\theta}_2(x)=\frac{1}{2}N([x][x])=[x][x]$.
By Lemma \ref{lem:2-6-2},
$L^{\theta}_3(x)=\frac{1}{2}N([x]\ell_2(x)+\ell_2(x)[x])=N([x]\ell_2(x))$.

Now we claim that if $X\in H$ and $u\in \Lambda^2H$ then
$N(Xu)=X\wedge u$. In fact, if $u=Y\wedge Z=YZ-ZY$
for some $Y,Z\in H$, then
\begin{eqnarray*}
N(Xu)=N(XYZ-XZY) &=& XYZ+YZX+ZXY-XZY-ZYX-YXZ \\
&=& X\wedge Y\wedge Z=X\wedge u.
\end{eqnarray*}
This proves the claim, hence proves (2).
\end{proof}

Let $\theta$ and $\theta^{\prime}$ be symplectic expansions.
As we saw in \S2.8, there uniquely exists
$U=U(\theta,\theta^{\prime})\in {\rm IA}(\widehat{T})$
such that $\theta^{\prime}=U\circ \theta$,
$U(H)\subset \widehat{\mathcal{L}}$, and $U(\omega)=\omega$.
The restriction of $U$ to $H$ is uniquely written as
$$U|_H=1_H+\sum_{k=1}^{\infty}u_k,\
u_k\in {\rm Hom}(H,\mathcal{L}_{k+1}).$$

\begin{lem}
\label{lem:7-6-2}
Notations are as above.
\begin{enumerate}
\item[{\rm (1)}] By the Poincar\'e duality
{\rm (\ref{eq:2-7-1})} we regard $u_k\in H\otimes \mathcal{L}_{k+1}$.
Then $u_1\in \Lambda^3H\subset H\otimes \mathcal{L}_2$.
\item[{\rm (2)}] For $x\in \pi$, we have
\begin{eqnarray*}
\ell^{\theta^{\prime}}_2(x) &=& \ell^{\theta}_2(x)+u_1([x]); \\
\ell^{\theta^{\prime}}_3(x) &=& \ell^{\theta}_3(x)+
u_1(\ell^{\theta}_2(x))+u_2([x]).
\end{eqnarray*}
\end{enumerate}
Here, $u_1(\ell^{\theta}_2(x))$ means
$(1\otimes u_1+u_1\otimes 1)\ell^{\theta}_2(x)$.
\end{lem}

\begin{proof}
Modulo $\widehat{T}_4$, we compute
\begin{eqnarray*}
\omega =U(\omega) &=& \sum_{i=1}^g U(A_i)U(B_i)-U(B_i)U(A_i) \\
&\equiv & \sum_{i=1}^g (A_i+u_1(A_i))(B_i+u_1(B_i))
-(B_i+u_1(B_i))(A_i+u_1(A_i)) \\
&\equiv &
\omega+\sum_{i=1}^g \left( A_iu_1(B_i)+u_1(A_i)B_i
-B_iu_1(A_i)-u_1(B_i)A_i \right)
\end{eqnarray*}
By the same reason as the discussion in \S2.8,
this implies $u_1\in {\rm Ker}([\ ,\ ]\colon
H\otimes \mathcal{L}_2 \rightarrow \mathcal{L}_3)$.
Also, we have ${\rm Ker}([\ ,\ ]\colon
H\otimes \mathcal{L}_2 \rightarrow \mathcal{L}_3)=\Lambda^3H$.
In fact, if $v\in {\rm Ker}([\ ,\ ]\colon
H\otimes \mathcal{L}_2 \rightarrow \mathcal{L}_3)$
then $\nu(v)=v$ by Lemma \ref{lem:2-6-2}, thus
$v=\frac{1}{3}(v+\nu(v)+\nu^2(v))$. This shows $v\in \Lambda^3H$.
The other inclusion follows from the Jacobi identity.
This proves the first part.

Again modulo $\widehat{T}_4$, we compute
\begin{eqnarray*}
\ell^{\theta^{\prime}}(x) &\equiv &
U([x]+\ell^{\theta}_2(x)+\ell^{\theta}_3(x)) \\
&\equiv & [x]+u_1([x])+u_2([x]) \\
& & +\ell^{\theta}_2(x)+u_1(\ell^{\theta}_2(x))
+\ell^{\theta}_3(x).
\end{eqnarray*}
This proves (2).
\end{proof}

\begin{cor}
\label{cor:7-6-3}
Notations are the same as Lemma \ref{lem:7-6-2}.
For $x\in \pi$, we have
\begin{enumerate}
\item[{\rm (1)}]
$L^{\theta^{\prime}}_3(x)-L^{\theta}_3(x)=[x]\wedge u_1([x])$,
\item[{\rm (2)}]
$L^{\theta^{\prime}}_4(x)-L^{\theta}_4(x)
=N([x]u_1(\ell^{\theta}_2(x))+N([x]u_2([x]))
+N(\ell^{\theta}_2(x)u_1([x]))+\displaystyle\frac{1}{2}N(u_1([x])u_1([x]))$.
\end{enumerate}
\end{cor}

\begin{proof}
The first part is clear from Lemmas \ref{lem:7-6-1} and \ref{lem:7-6-2}.
The second part follows from
$$L^{\theta}_4(x)=N([x]\ell^{\theta}_3(x))
+\frac{1}{2}N(\ell^{\theta}_2(x)\ell^{\theta}_2(x))$$
and Lemma \ref{lem:7-6-2}.
\end{proof}

\subsection{The formulas of $\tau^{\theta}_1(t_C)$
and $\tau^{\theta}_2(t_C)$ for non-separating $C$}

Let $C$ be a non-separating simple closed curve on $\Sigma$.
As we did in \S7.2, we take symplectic generators
$\alpha_1,\beta_1,\ldots,\alpha_g,\beta_g$
such that $|\alpha_1|$ is freely homotopic to $C$
as unoriented loops.
In this situation, Massuyeau \cite{Mas},
Example 2.19 gave a partial example of a symplectic
expansion $\theta^0$ whose values of $\ell^{\theta^0}(\alpha_1)$
and $\ell^{\theta^0}(\beta_1)$ modulo $\widehat{T}_5$
are as follows:
\begin{eqnarray}
\ell^{\theta^0}(\alpha_1)&\equiv & A_1+\frac{1}{2}[A_1,B_1]
+\frac{-1}{12}[B_1,[A_1,B_1]]+\frac{1}{24}[A_1,[A_1,[A_1,B_1]]];
\nonumber \\
\ell^{\theta^0}(\beta_1)&\equiv & B_1-\frac{1}{2}[A_1,B_1]
+\frac{1}{12}[A_1,[A_1,B_1]]+\frac{1}{4}[B_1,[A_1,B_1]] \nonumber \\
 & &-\frac{1}{24}[B_1,[B_1,[B_1,A_1]]]. \label{eq:7-7-1}
\end{eqnarray}
Here, $A_1=[\alpha_1]$ and $B_1=[\beta_1]$.

Note that our conventions about symplectic generators and
symplectic expansions are different from
Massuyeau \cite{Mas}, Definition 2.15. Therefore
(\ref{eq:7-7-1}) equals the equations of \cite{Mas}, Example 2.19,
only up to sign.

\begin{prop}
\label{prop:7-7-1}
Let $\theta$ be a symplectic expansion and $C$ a non-separating
simple closed curve on $\Sigma$. Let $L_k=L^{\theta}_k(C)$.
We regard them as derivations of $\widehat{T}$.
Then $L_2^2=L_2L_3=L_3L_2=0$ on $H$. In particular,
as linear endomorphisms of $\widehat{T}$,
$L_2^{n+1}|_{H^{\otimes n}}=0$ and $L_2L_3=L_3L_2$.
\end{prop}

\begin{proof}
We take symplectic generators as above.
Since $L_2=A_1^2$, $L_2^2(X)=(X\cdot A_1)L_2A_1=0$ for $X\in H$.
Therefore, $L_2^2=0$ on $H$.

Let $\theta^0$ be a symplectic expansion of
(\ref{eq:7-7-1}). By Lemma \ref{lem:7-6-1} (2) we have
$L_3^{\theta^0}(C)=\frac{1}{2}A_1\wedge A_1\wedge B_1=0$.
Thus $L_2L_3=L_3L_2=0$ for $\theta^0$.

Let $\theta^{\prime}$ be another symplectic expansion
and let $U=U(\theta^0,\theta^{\prime})$. We need to show
$L_2^{\theta^{\prime}}L_3^{\theta^{\prime}}
=L_3^{\theta^{\prime}}L_2^{\theta^{\prime}}=0$ on $H$.
If $U={\rm id}$, this is true by what we have shown.
Therefore, the proposition follows from Corollary \ref{cor:7-6-3} (1)
and the following lemma.
\end{proof}

\begin{lem}
\label{lem:7-7-2}
Let $L_2=A_1^2$ and let
$L_3^{\prime \prime}=A_1\wedge u_1(A_1)$, where $u_1\in \Lambda^3H$.
We regard them as derivations of $\widehat{T}$. Then
$L_2L_3^{\prime \prime}=L_3^{\prime \prime}L_2=0$ on $H$.
\end{lem}

\begin{proof}
For simplicity we write $A_1,B_1,\ldots,A_g,B_g=X_1,\ldots,X_{2g}$.
By linearity, it suffices to prove the lemma when $u_1$ is of the form
$u_1=X_i \wedge X_j \wedge X_k$ with $i\neq j\neq k\neq i$.
Note that for $Y\in H$ we have
$$u_1(Y)=(Y\cdot X_i)X_j\wedge X_k+(Y\cdot X_j)X_k\wedge X_i
+(Y\cdot X_k)X_i\wedge X_j.$$
We divide the argument in two cases. First
suppose none of $X_i$, $X_j$, and $X_k$
are equal to $B_1$. Then $u_1(A_1)=0$ hence $L_3^{\prime \prime}=0$.
Therefore $L_2L_3^{\prime \prime}=L_3^{\prime \prime}L_2=0$.

Next suppose $X_i=B_1$. Then $X_j,X_k\neq B_1$, and we have
$u_1(A_1)=X_j\wedge X_k$, hence $L_3^{\prime \prime}=A_1\wedge X_j\wedge X_k$.
Since $L_2A_1=L_2X_j=L_2X_k=0$, it follows that
$L_2L_3^{\prime \prime}Y=0$ for any $Y\in H$.
Since $L_3^{\prime \prime}A_1=0$,
$L_3^{\prime \prime}L_2Y=(Y\cdot A_1)L_3^{\prime \prime}A_1=0$
for any $Y\in H$. This completes the proof.
\end{proof}

\begin{thm}
\label{thm:7-7-3}
Let $\theta$ be a symplectic expansion and
$C$ a non-separating simple closed curve on $\Sigma$. Then we have
\begin{equation}
\label{eq:7-7-2}
\tau_1^{\theta}(t_C)= -L_3^{\theta}(C).
\end{equation}
\end{thm}

\begin{proof}
For $X\in H$, we have
$$\tau^{\theta}(t_C)X=T^{\theta}(t_C)(|t_C|^{-1}X)
=e^{-L}(X+L_2X).$$
Thus $\tau_1^{\theta}(t_C)X$ is equal to the degree two-part
of $e^{-L}(X+L_2X)$.
Modulo $\widehat{T}_3$, we compute
$$e^{-L}(X+L_2X)\equiv X+L_2X-L_2(X+L_2X)-L_3(X+L_2X)
=X-L_3X,$$
using Proposition \ref{prop:7-7-1}. This completes the proof.
\end{proof}

This theorem is compatible with the computation
by Morita \cite{Mo2}, Proposition 4.2.
One reason for the choice of our convention
about the Poincar\'e duality (\ref{eq:2-7-1}) is
to make our formula compatible with his computation.

We next compute $\tau_2^{\theta}(t_C)$ for non-separating $C$.

\begin{prop}
\label{prop:7-7-4}
Let $\theta$ be a symplectic expansion and $C$ a non-separating
curve on $\Sigma$. Let $L_k=L^{\theta}_k$.
We regard them as derivations of $\widehat{T}$.
Then $L_2L_2L_2L_4=L_2L_2L_4L_2=0$, and
$2L_2L_4L_2=L_2L_2L_4$ on $H$.
\end{prop}

\begin{proof}
Let $\theta^0$ be a symplectic expansion of
(\ref{eq:7-7-1}). We first prove the proposition for $\theta=\theta^0$.
We have $L_2=L^{\theta^0}_2(C)=A_1^2$, and
\begin{eqnarray*}
L^{\theta^0}_4(C)
&=& N(A_1\ell^{\theta^0}_3(\alpha_1))
+\frac{1}{2}N(\ell^{\theta^0}_2(\alpha_1)\ell^{\theta^0}_2(\alpha_1)) \\
&=& -\frac{1}{12}N(A_1[B_1,[A_1,B_1]])+\frac{1}{8}N([A_1,B_1][A_1,B_1]) \\
&=& -\frac{1}{12}N([A_1,B_1][A_1,B_1]])+\frac{1}{8}N([A_1,B_1][A_1,B_1]) \\
&=& \frac{1}{24}N([A_1,B_1][A_1,B_1]]).
\end{eqnarray*}
Here we use Lemma \ref{lem:2-6-2}.

As we did in the proof of Lemma \ref{lem:7-7-2},
we write $A_1,B_1,\ldots,A_g,B_g=X_1,\ldots,X_{2g}$.
For simplicity, we write $L_k^{\theta^0}(C)=L_k^0$.
If $i\ge 3$, clearly we have $L_4^0X_i=L_2^0X_i=0$.
By a direct computation, we have
$$L_4^0X_1=L_4^0A_1=-\frac{1}{24}[A_1,[A_1,B_1]],$$
$$L_4^0X_2=L_4^0B_1=-\frac{1}{24}[B_1,[A_1,B_1]].$$
From these we conclude $L_2^0L_4^0X_1=0$ and $L_2^0L_2^0L_4^0X_2=0$.
It follows that $L_2^0L_2^0L_4^0=L_2^0L_4^0L_2^0=0$ on $H$,
hence $L_2^0L_2^0L_2^0L_4^0=L_2^0L_2^0L_4^0L_2^0=0$ and
$2L_2^0L_4^0L_2^0=L_2^0L_2^0L_4^0(=0)$ on $H$.
The proposition is proven for $\theta=\theta^0$.

We next consider the general case. Let $\theta^{\prime}$
be another symplectic expansion and $U=U(\theta^0,\theta^{\prime})$.
Let $L_2=A_1^2$ and $L_4^{\prime \prime}
=L_4^{\theta^{\prime}}(C)-L_4^{\theta^0}(C)$.
It suffices to prove $L_2^3L_4^{\prime \prime}
=L_2^2L_4^{\prime \prime}L_2=0$, and
$2L_2L_4^{\prime \prime}L_2=L_2^2L_4^{\prime \prime}$ on $H$.

By Corollary \ref{cor:7-6-3} (2), we have
\begin{eqnarray*}
L_4^{\prime \prime} &=& N(A_1u_1(\ell^{\theta^0}_2(\alpha_1)))
+N(A_1u_2(A_1)) \\
& & +N(\ell^{\theta^0}_2(\alpha_1)u_1(A_1))
+\frac{1}{2}N(u_1(A_1)u_1(A_1)).
\end{eqnarray*}
Since $\ell^{\theta^0}_2(\alpha_1)=\frac{1}{2}[A_1,B_1]$, we compute
\begin{eqnarray*}
& &
N(A_1u_1(\ell^{\theta^0}_2(\alpha_1)))
+N(\ell^{\theta^0}_2(\alpha_1)u_1(A_1)) \\
&=& \frac{1}{2}N(A_1([A_1,u_1(B_1)]+[u_1(A_1),B_1]))
+\frac{1}{2}N([A_1,B_1]u_1(A_1))=0
\end{eqnarray*}
using Lemma \ref{lem:2-6-2}. Therefore,
$$L_4^{\prime \prime}=N(A_1u_2(A_1))+\frac{1}{2}N(u_1(A_1)u_1(A_1)).$$
Since $u_2(A_1)\in \mathcal{L}_3$ and $u_1\in \Lambda^3H$,
it follows that $L_4^{\prime \prime}\in H^{\otimes 4}$
is a linear combination of monomials
in $X_i$'s with the number of the occurrences of $B_1=X_2$ at most two.
By this observation we have $L_2^3L_4^{\prime \prime}
=L_2^2L_4^{\prime \prime}L_2=0$ on $H$.

It remains to prove the assertion
$2L_2L_4^{\prime \prime}L_2=L_2^2L_4^{\prime \prime}$ on $H$.
Let $L_4^{\prime \prime \prime}=\frac{1}{2}N(u_1(A_1)u_1(A_1))$.
Since $u_1\in \Lambda^3H$, $L_4^{\prime \prime \prime}$
is a linear combination of monomials in $X_i$'s with no occurrence of $X_2$.
It follows that $2L_2L_4^{\prime \prime \prime}L_2
=L_2^2L_4^{\prime \prime \prime}=0$ on $H$.
Now the assertion follows from the following lemma.
\end{proof}

\begin{lem}
\label{lem:7-7-5}
Let $u\in \mathcal{L}_3$, $X\in H$ and set $L_X=X^2$,
$L_4=N(Xu)$.
We regard $L_X$ and $L_4$ as a derivation of $\widehat{T}$.
Then $2L_XL_4L_X=L_X^2L_4$ on $H$.
\end{lem}

\begin{proof}
By linearity, we may assume that $u=[Y_1,[Y_2,Y_3]]$
where $Y_i\in H$. For $Z\in H$,
$$L_4Z=(Z\cdot X)[Y_1,[Y_2,Y_3]]-(Z\cdot Y_1)[X,[Y_2,Y_3]]
+(Z\cdot Y_2)[Y_3,[X,Y_1]]-(Z\cdot Y_3)[Y_2,[X,Y_1]]$$
(see Lemma \ref{lem:2-7-1}). Using this, we have
$$L_X^2L_4Z=2(Z\cdot X)(Y_1\cdot X)
\left\{ (Y_2\cdot X)[X,[X,Y_3]]+(Y_3\cdot X)[X,[Y_2,X]] \right\}$$
by a direct computation. On the other hand, we compute
\begin{eqnarray*}
L_XL_4L_X Z &=& (Z\cdot X)L_XL_4(X) \\
&=& (Z\cdot X)L_X(-(X\cdot Y_1)[X,[Y_2,Y_3]]+(X\cdot Y_2)[Y_3,[X,Y_1]]
-(X\cdot Y_3)[Y_2,[X,Y_1]]) \\
&=& (Z\cdot X)\{
-(X\cdot Y_1)((Y_2\cdot X)[X,[X,Y_3]]+(Y_3\cdot X)[X,[Y_2,X]]) \\
& & +(X\cdot Y_2)(Y_3\cdot X)[X,[X,Y_1]]
-(X\cdot Y_3)(Y_2\cdot X)[X,[X,Y_1]] \} \\
&=& (Z\cdot X)(Y_1\cdot X) \{
(Y_2\cdot X)[X,[X,Y_3]]+(Y_3\cdot X)[X,[Y_2,X]] \}.
\end{eqnarray*}
This proves the lemma.
\end{proof}

\begin{thm}
\label{thm:7-7-6}
Let $\theta$ be a symplectic expansion and
$C$ a non-separating simple closed curve on $\Sigma$. Then we have
\begin{equation}
\label{eq:7-7-3}
\tau_2^{\theta}(t_C)= -L_4+\frac{1}{2}[L_2,L_4]+\frac{1}{2}L_3^2.
\end{equation}
\end{thm}

\begin{proof}
Let $X\in H$. Modulo $\widehat{T}_4$, we compute
\begin{eqnarray*}
LX &\equiv & L_2X+L_3X+L_4X, \\
LLX &\equiv & L_2(L_2X+L_3X+L_4X)+L_3(L_2X+L_3X)+L_4L_2X \\
 &=& L_2L_4X+L_3L_3X+L_4L_2X, \\
LLLX &\equiv & L_2L_2L_4X+L_2L_3L_3X+L_2L_4L_2X=L_2L_2L_4X+L_2L_4L_2X, \\
LLLLX &\equiv & L_2L_2L_2L_4X+L_2L_2L_4L_2X=0,\ {\rm and} \\
L(L_2X) &\equiv & L_4L_2X, \\
LL(L_2X) &\equiv &L_2L_4L_2X, \\
LLL(L_2X) &\equiv & 0.
\end{eqnarray*}
Here we use Proposition \ref{prop:7-7-1} and the first part
of Proposition \ref{prop:7-7-4}.
Note that $L_2L_3L_3X=L_3L_2L_3X=0$.
Therefore, the degree 4-part of
$\tau^{\theta}(t_C)X=e^{-L}(X+L_2X)$ is equal to
$$-L_4X-L_4L_2X+\frac{1}{2}(L_2L_4X+L_3L_3X+L_4L_2X)
+\frac{1}{2}L_2L_4L_2X-\frac{1}{6}(L_2L_2L_4X+L_2L_4L_2X).$$
Using the second part of Proposition \ref{prop:7-7-4},
the formula follows.
\end{proof}

\section{The case of $\mathcal{M}_{g,*}$}
We close this paper by deriving similar results
for the mapping class group of a once punctured surface.
Let $\Sigma_g$ be a closed oriented $C^{\infty}$-surface
of genus $g$. Choose a basepoint $*^{\prime}\in \Sigma_g$
and let $\pi_1(\Sigma_g)=\pi_1(\Sigma_g,*^{\prime})$.

\subsection{The mapping class group $\mathcal{M}_{g,*}$}
Let $\mathcal{M}_{g,*}$ be
the mapping class group of $\Sigma_g$ relative to $*^{\prime}$,
namely the group of orientation-preserving 
diffeomorphisms of $\Sigma_g$ fixing $*^{\prime}$, modulo
isotopies fixing $*^{\prime}$.
By the theorem of Dehn-Nielsen, we have a natural identification
\begin{equation}
\label{eq:8-1-1}
\mathcal{M}_{g,*}={\rm Aut}^+(\pi_1(\Sigma_g)),
\end{equation}
where $+$ means acting on $H_2(\pi_1(\Sigma_g);\mathbb{Z})\cong \mathbb{Z}$
as the identity.

We take a small disk $D$ around $*^{\prime}$ and fix an identification
$$\Sigma_g \setminus {\rm Int}(D)\cong \Sigma.$$
We can extend any diffeomorphism of $\Sigma$ to
a diffeomorphism of $\Sigma_g$ by defining the extension as the identity on $D$.
In this way we have a natural surjective homomorphism
\begin{equation}
\label{eq:8-1-2}
\mathcal{M}_{g,1}\rightarrow \mathcal{M}_{g,*}.
\end{equation}
For simplicity let us write
${\rm Aut}_{\zeta}(\pi)=\{ \varphi\in {\rm Aut}(\pi); \varphi(\zeta)=\zeta \}$
(see (\ref{eq:2-1-1})).
We have a natural surjection from $\pi=\pi_1(\Sigma,*)$
to $\pi_1(\Sigma_g)=\pi_1(\Sigma_g,*^{\prime})$.
This naturally induces a homomorphism
${\rm Aut}_{\zeta}(\pi)\rightarrow {\rm Aut}^+(\pi_1(\Sigma_g))$.
This map is compatible with (\ref{eq:2-1-1}) and (\ref{eq:8-1-1}).
Namely, the diagram
$$\begin{CD}
\mathcal{M}_{g,1} @>>> \mathcal{M}_{g,*} \\
@V{\cong}VV @V{\cong}VV\\
{\rm Aut}_{\zeta}(\pi) @>>> {\rm Aut}^+(\pi_1(\Sigma_g))
\end{CD}$$
commutes.

\subsection{Action on the competed group ring of $\pi_1(\Sigma)$}
Let $\mathcal{N}$ be the two-sided ideal of $\widehat{T}$ generated by $\omega$,
and $\widehat{T}/\mathcal{N}$ the quotient algebra.
It naturally inherits a decreasing filtration $(\widehat{T}/\mathcal{N})_p$,
$p\ge 1$ and a structure of complete
Hopf algebra from $\widehat{T}$.
We denote by $\varpi$ the projection
$\widehat{T}\rightarrow \widehat{T}/\mathcal{N}$.

If $\theta$ is a symplectic expansion of $\pi$,
$\theta(\zeta)=\exp (\omega)\in 1+\mathcal{N}$. Thus 
it induces a group homomorphism
$\bar{\theta}\colon \pi_1(\Sigma_g)\rightarrow 1+ (\widehat{T}/\mathcal{N})_1$.

\begin{lem}
\label{lem:8-2-1}
Let $\theta$ be a symplectic expansion of $\pi$.
Then the induced map
\begin{equation}
\label{eq:8-2-1}
\bar{\theta}\colon \widehat{\mathbb{Q}\pi_1(\Sigma_g)}
\rightarrow \widehat{T}/\mathcal{N}
\end{equation}
is an isomorphism of complete Hopf algebras.
Here $\widehat{\mathbb{Q}\pi_1(\Sigma_g)}$
is the completed group ring of $\pi_1(\Sigma_g)$, namely the
completion of $\mathbb{Q}\pi_1(\Sigma_g)$ by
the augmentation ideal.
\end{lem}

\begin{proof}
It is clear that (\ref{eq:8-2-1}) is a homomorphism
of complete Hopf algebras.
Consider the following commutative diagram:
\begin{equation}
\label{eq:8-2-2}
\begin{CD}
\widehat{\mathbb{Q}\pi} @>{\theta}>{\cong}> \widehat{T} \\
@VVV @VV{\varpi}V\\
\widehat{\mathbb{Q}\pi_1(\Sigma_g)} @>{\bar{\theta}}>> \widehat{T}/\mathcal{N}
\end{CD}
\end{equation}
Let $\eta$ be the inverse of the
isomorphism $\theta\colon \widehat{\mathbb{Q}\pi} \rightarrow \widehat{T}$.
Then $\eta(\omega)=\log \zeta$, which is mapped to zero under the map
$\widehat{\mathbb{Q}\pi} \rightarrow \widehat{\mathbb{Q}\pi_1(\Sigma_g)}$.
Therefore, $\eta$ induces a morphism $\bar{\eta}\colon \widehat{T}/\mathcal{N}
\rightarrow \widehat{\mathbb{Q}\pi_1(\Sigma_g)}$.
We claim $\bar{\eta}$ is the inverse of $\bar{\theta}$.
Since the diagram (\ref{eq:8-2-2}) commutes and $\eta=\theta^{-1}$,
it suffices to show the surjectivities of the horizontal arrows in
(\ref{eq:8-2-2}). The surjectivity of $\varpi$ is clear. To show
the surjectivity of the left horizontal arrow, we can use
the criterion of Quillen \cite{Qui} Appendix A, Proposition 1.6.
This completes the proof.
\end{proof}

The isomorphism (\ref{eq:8-2-1}) leads to the
definition of a counterpart of the total Johnson map
$T^{\theta}\colon \mathcal{M}_{g,1}\rightarrow {\rm Aut}(\widehat{T})$.
Let ${\rm Aut}(\widehat{T}/\mathcal{N})$ be the group
of the filter-preserving algebra automorphisms of $\widehat{T}/\mathcal{N}$.
Let $\bar{\varphi}\in \mathcal{M}_{g,*}$. As a consequence of
(\ref{eq:8-2-1}) there uniquely
exists $\bar{T}^{\theta}(\bar{\varphi})\in {\rm Aut}(\widehat{T}/\mathcal{N})$
such that $\bar{T}^{\theta}(\bar{\varphi}) \circ \bar{\theta}
=\bar{\theta} \circ \bar{\varphi}$. In this way
we have the group homomorphism
$$\bar{T}^{\theta}\colon \mathcal{M}_{g,*}\rightarrow
{\rm Aut}(\widehat{T}/\mathcal{N}).$$
It is known that $\bigcap_{m=1}^{\infty}I\pi_1(\Sigma_g)^m=0$, where
$I\pi_1(\Sigma_g)$ is the augmentation ideal.
See, for example, Chen \cite{Ch} p.193, Corollary 1 and p.197, Corollary 4.
It follows that the natural map $\pi_1(\Sigma_g)
\rightarrow \widehat{\mathbb{Q}\pi_1(\Sigma_g)}$ is injective,
so is the homomorphism $\bar{T}^{\theta}$.

Let $\varphi\in \mathcal{M}_{g,1}$.
Since $\theta$ is symplectic,
$T^{\theta}(\varphi)(\omega)=
T^{\theta}(\varphi)(\ell^{\theta}(\zeta))=
\ell^{\theta}(\varphi(\zeta))=\ell^{\theta}(\zeta)=\omega$.
This shows $T^{\theta}(\varphi)\in {\rm Aut}(\widehat{T})$ preserves $\mathcal{N}$.
By construction, we have
\begin{equation}
\label{eq:8-2-3}
\varpi \circ T^{\theta}(\varphi)=\bar{T}^{\theta}(\bar{\varphi})\circ \varpi,
\end{equation}
where $\bar{\varphi}\in \mathcal{M}_{g,*}$ is the image of $\varphi$
by (\ref{eq:8-1-2}).

Let $C$ be a simple closed curve on $\Sigma_g\setminus \{ *^{\prime}\}$.
Then $t_C$, the Dehn twist along $C$, is defined as
an element of $\mathcal{M}_{g,*}$.
Since $\Sigma$ is a deformation retract of $\Sigma_g\setminus \{ *^{\prime}\}$,
we can regard $C$ as a simple closed curve on $\Sigma$.
Thus, $t_C$ is also defined as an element of $\mathcal{M}_{g,1}$.
By (\ref{eq:8-2-3}), we have
$$\varpi \circ T^{\theta}(t_C)=\bar{T}^{\theta}(t_C)\circ \varpi.$$
Also, $L^{\theta}(C)\in \widehat{T}_2$ is well-defined.
Since $L^{\theta}(C)\in \mathfrak{l}_g$ (see \S2.7),
$L^{\theta}(C)\omega=0$. Therefore, $L^{\theta}(C)$ preserves $\mathcal{N}$
and it defines a derivation of $\widehat{T}/\mathcal{N}$. We denote
it by $\bar{L}^{\theta}(C)$. By construction, we have
$\varpi \circ L^{\theta}(C)=\bar{L}^{\theta}(C)\circ \varpi$ and moreover,
$$\varpi \circ e^{-L^{\theta}(C)}=e^{-\bar{L}^{\theta}(C)}\circ \varpi.$$
By Theorem \ref{thm:7-1-1}, we have
$\bar{T}^{\theta}(t_C)\circ \varpi=e^{-\bar{L}^{\theta}(C)}\circ \varpi$
and since $\varpi$ is surjective, $\bar{T}^{\theta}(t_C)=e^{-\bar{L}^{\theta}(C)}$.
In summary, we have proved the following theorem.
\begin{thm}
\label{thm:8-2-2}
Let $\theta$ be a symplectic expansion and $C$ a simple closed
curve on $\Sigma_g\setminus \{ *^{\prime}\}$. Let $t_C\in \mathcal{M}_{g,*}$
be the Dehn twist along $C$. Then
$$\bar{T}^{\theta}(t_C)=e^{-\bar{L}^{\theta}(C)}.$$
Here the right hand side is the algebra automorphism of $\widehat{T}/\mathcal{N}$
defined by the exponential of the derivation $-\bar{L}^{\theta}(C)$.
\end{thm}

\subsection{Action on $N_k(\pi_1(\Sigma_g))$}
Finally we prove a result similar to Theorem \ref{thm:7-4-1}.

Let $C$ be a simple closed curve on $\Sigma_g\setminus \{ *^{\prime}\}$.
As we saw in \S8.2, we can regard $C$ as a simple closed
curve on $\Sigma$. As we did in \S7.4, for each $k\ge 0$,
$\bar{C}_k\in \bar{N}_k=\bar{N}_k(\pi)$ is defined.

For each $k\ge 0$, let $N_k(\pi_1(\Sigma_g))$ be the $k$-th nilpotent quotient of
$\pi_1(\Sigma_g)$, defined similarly to $N_k=N_k(\pi)$.
The mapping class group $\mathcal{M}_{g,*}$
naturally acts on $N_k(\pi_1(\Sigma_g))$.

\begin{thm}
\label{thm:8-3-1}
For each $k\ge 1$,
the action of $t_C$ on $N_k(\pi_1(\Sigma_g))$ depends only
on $\bar{C}_k\in \bar{N}_k$. If $C$ is separating,
it depends only on $\bar{C}_{k-1}\in \bar{N}_{k-1}$.
\end{thm}

\begin{proof}
Let $N_k(\pi)\rightarrow N_k(\pi_1(\Sigma_g))$
be the natural surjection. This map is compatible
with (\ref{eq:8-1-2}) and the actions of the two mapping class groups
on the nilpotent quotients. The result follows from
Theorem \ref{thm:7-4-1}.
\end{proof}

\section{Appendix: Examples of symplectic expansions}
We show first few terms of the symplectic expansion associated to
symplectic generators given by the method
mentioned in Example \ref{ex:2-4-4}. As the reader might
notice from the below,
this symplectic expansion has a certain kind of symmetry.
For details, see \cite{Ku}.

\vspace{0.5cm}
\noindent \textbf{The case of genus 1.}
For simplicity, write $\alpha_1=\alpha$, $\beta_1=\beta$
and $A_1=A$, $B_1=B$. Then,
there is a symplectic expansion $\theta$ of the following form:
modulo $\widehat{T}_6$,
\begin{eqnarray*}
\ell^{\theta}(\alpha) &\equiv &
A+\frac{1}{2}[A,B]+\frac{1}{12}[B,[B,A]]
-\frac{1}{8}[A,[A,B]]+\frac{1}{24}[A,[A,[A,B]]] \\
& & -\frac{1}{720}[B,[B,[B,[B,A]]]]-\frac{1}{288}
[A,[A,[A,[A,B]]]]-\frac{1}{288}[A,[B,[B,[B,A]]]] \\
& & -\frac{1}{288}[B,[A,[A,[A,B]]]]
+\frac{1}{144}[[A,B],[B,[B,A]]]+\frac{1}{128}[[A,B],[A,[A,B]]];
\end{eqnarray*}
\begin{eqnarray*}
\ell^{\theta}(\beta) &\equiv &
B-\frac{1}{2}[A,B]+\frac{1}{12}[A,[A,B]]
-\frac{1}{8}[B,[B,A]]+\frac{1}{24}[B,[B,[B,A]]] \\
& & -\frac{1}{720}[A,[A,[A,[A,B]]]]-\frac{1}{288}
[B,[B,[B,[B,A]]]]-\frac{1}{288}[B,[A,[A,[A,B]]]] \\
& & -\frac{1}{288}[A,[B,[B,[B,A]]]]
-\frac{1}{144}[[A,B],[A,[A,B]]]-\frac{1}{128}[[A,B],[B,[B,A]]].
\end{eqnarray*}

\vspace{0.5cm}
\noindent \textbf{The case of genus 2.} There is a symplectic
expansion $\theta$ of the following form: modulo $\widehat{T}_5$,
\begin{eqnarray*}
\ell^{\theta}(\alpha_1) &\equiv &
A_1+\frac{1}{2}[A_1,B_1] \\
& & +\frac{1}{12}[B_1,[B_1,A_1]]-\frac{1}{8}[A_1,[A_1,B_1]]
-\frac{1}{4}[A_1,[A_2,B_2]] \\
& & +\frac{1}{24}[A_1,[A_1,[A_1,B_1]]]-\frac{1}{10}[[A_1,B_1],[A_2,B_2]]
+\frac{1}{40}[A_1,[B_1,[A_2,B_2]]] \\
& & +\frac{1}{40}[A_1,[B_2,[A_2,B_2]]]+\frac{1}{40}[A_1,[A_1,[A_2,B_2]]]
+\frac{1}{40}[A_1,[A_2,[A_2,B_2]]];
\end{eqnarray*}

\begin{eqnarray*}
\ell^{\theta}(\beta_1) &\equiv &
B_1-\frac{1}{2}[A_1,B_1] \\
& & +\frac{1}{12}[A_1,[A_1,B_1]]-\frac{1}{8}[B_1,[B_1,A_1]]
-\frac{1}{4}[B_1,[A_2,B_2]] \\
& & +\frac{1}{24}[B_1,[B_1,[B_1,A_1]]]+\frac{1}{10}[[A_1,B_1],[A_2,B_2]]
+\frac{1}{40}[B_1,[A_1,[A_2,B_2]]] \\
& & +\frac{1}{40}[B_1,[A_2,[A_2,B_2]]]+\frac{1}{40}[B_1,[B_1,[A_2,B_2]]]
+\frac{1}{40}[B_1,[B_2,[A_2,B_2]]];
\end{eqnarray*}

\begin{eqnarray*}
\ell^{\theta}(\alpha_2) &\equiv &
A_2+\frac{1}{2}[A_2,B_2] \\
& & +\frac{1}{12}[B_2,[B_2,A_2]]-\frac{1}{8}[A_2,[A_2,B_2]]
+\frac{1}{4}[A_2,[A_1,B_1]] \\
& & +\frac{1}{24}[A_2,[A_2,[A_2,B_2]]]-\frac{1}{10}[[A_1,B_1],[A_2,B_2]]
-\frac{1}{40}[A_2,[B_2,[A_1,B_1]]] \\
& & -\frac{1}{40}[A_2,[B_1,[A_1,B_1]]]-\frac{1}{40}[A_2,[A_2,[A_1,B_1]]]
-\frac{1}{40}[A_2,[A_1,[A_1,B_1]]];
\end{eqnarray*}

\begin{eqnarray*}
\ell^{\theta}(\beta_2) &\equiv &
B_2-\frac{1}{2}[A_2,B_2] \\
& & +\frac{1}{12}[A_2,[A_2,B_2]]-\frac{1}{8}[B_2,[B_2,A_2]]
+\frac{1}{4}[B_2,[A_1,B_1]] \\
& & +\frac{1}{24}[B_2,[B_2,[B_2,A_2]]]+\frac{1}{10}[[A_1,B_1],[A_2,B_2]]
-\frac{1}{40}[B_2,[A_2,[A_1,B_1]]] \\
& & -\frac{1}{40}[B_2,[A_1,[A_1,B_1]]]-\frac{1}{40}[B_2,[B_2,[A_1,B_1]]]
-\frac{1}{40}[B_2,[B_1,[A_1,B_1]]].
\end{eqnarray*}

\noindent \textsc{Nariya Kawazumi\\
Department of Mathematical Sciences,\\
University of Tokyo,\\
3-8-1 Komaba Meguro-ku Tokyo 153-8914 JAPAN}

\noindent \texttt{E-mail address: kawazumi@ms.u-tokyo.ac.jp}

\vspace{0.5cm}

\noindent \textsc{Yusuke Kuno\\
Department of Mathematics,\\
Graduate School of Science,\\
Hiroshima University,\\
1-3-1 Kagamiyama, Higashi-Hiroshima, Hiroshima 739-8526 JAPAN}

\noindent \texttt{E-mail address: kunotti@hiroshima-u.ac.jp}

\end{document}